\documentclass[12pt,reqno]{amsart}

\usepackage{a4wide}

\pdfoutput=1

\usepackage{amsmath,amsfonts,amssymb,amscd,amsbsy}
\usepackage{amsthm}
\usepackage{mathtools}
\usepackage[dvipsnames]{xcolor}
\usepackage{setspace} 
\usepackage[hang, flushmargin]{footmisc}
\usepackage{hyperref}
\hypersetup{
    colorlinks=true,
    linkcolor=blue,
    filecolor=magenta,      
    urlcolor=cyan,
    citecolor=Green
    }
\usepackage{footnotebackref}

\usepackage{xpatch}
\makeatletter
\xpretocmd{\@adminfootnotes}{\let\@makefntext\BHFN@OldMakefntext}{}{}
\renewcommand\@makefntext[1]{%
  \@ifundefined{@makefnmark}
    {}
    {%
     \renewcommand\@makefnmark{%
       \mbox{%
         \textsuperscript{%
           \normalfont
           \hyperref[\BackrefFootnoteTag]{\@thefnmark}%
         }%
       }\,%
     }%
     \BHFN@OldMakefntext{#1}%
  }%
}
\makeatother

\usepackage{multicol}
\usepackage{parskip} 
\usepackage{fancyhdr}
\usepackage{chngcntr}
\usepackage[shortlabels]{enumitem} 
\usepackage{longtable} 
\usepackage{array} 
\usepackage{booktabs} 
\usepackage{stackengine} 

\usepackage{euscript}
\usepackage{upgreek} 
\usepackage{mathrsfs}
\usepackage{stmaryrd} 
\usepackage{bm} 

\allowdisplaybreaks

\usepackage[all]{xy}
\usepackage{graphicx}
\usepackage{pict2e}
\usepackage{tikz-cd}
\usetikzlibrary{arrows}
\usepackage[
mark=o,
root-radius=.035cm,
edge-length=0.46cm,
indefinite-edge-ratio=2,
indefinite-edge={
draw=black,fill=white,thin,densely dashed}]{dynkin-diagrams}
\usepackage{makecell} 
\usepackage{cancel} 
\usepackage{tcolorbox} 

\usepackage{perpage} 
\MakePerPage{footnote}

\DeclareMathOperator{\Ker}{Ker}

\DeclareMathOperator{\Aut}{Aut}
\DeclareMathOperator{\Inn}{Inn}
\DeclareMathOperator{\Der}{Der}

\DeclareMathOperator{\Id}{Id}
\DeclareMathOperator{\codim}{codim}
\DeclareMathOperator{\rk}{rk}

\DeclareMathOperator{\pr}{pr}
\DeclareMathOperator{\vspan}{span}

\DeclareMathOperator{\Ad}{Ad}
\DeclareMathOperator{\ad}{ad}

\DeclareMathOperator{\GL}{GL}
\DeclareMathOperator{\SL}{SL}
\DeclareMathOperator{\Lie}{Lie}
\DeclareMathOperator{\Gr}{Gr}

\DeclareMathOperator{\DD}{DD}

\DeclareMathOperator{\height}{ht}

\DeclareMathOperator{\hgt}{ht}

\newcommand{\Z}{\mathbb{Z}}

\newcommand{\R}{\mathbb{R}}
\newcommand{\C}{\mathbb{C}}
\newcommand{\Oo}{\mathbb{O}}

\newcommand{\E}{\mathbb{E}}

\newcommand{\SO}{\hspace{0.5pt}\mathrm{SO}}
\newcommand{\U}{\hspace{0.5pt}\mathrm{U}}

\newcommand{\Sp}{\hspace{0.5pt}\mathrm{Sp}}

\newcommand{\Spin}{\hspace{0.5pt}\mathrm{Spin}}

\newcommand{\defeq}{\vcentcolon=}
\newcommand{\defqe}{=\vcentcolon}
\newcommand{\ccdot}{\hspace{-1.2pt} \cdot}

\newcommand{\wt}[1]{\widetilde{{#1}}}

\newcommand{\mk}[1]{\mathfrak{{#1}}}

\newcommand{\mm}[1]{\mathrm{{#1}}}
\newcommand{\mc}[1]{\mathcal{{#1}}}
\newcommand{\wh}[1]{\widehat{{#1}}}
\newcommand{\set}[1]{\hspace{-0.8pt} \left \{ \hspace{0.03em} {#1} \hspace{0.03em} \right \}}
\newcommand{\bilin}[2]{\langle \hspace{1pt} {#1} \hspace{0.5pt} , {#2} \hspace{1pt} \rangle \hspace{0.02em}}
\newcommand{\cross}[2]{\langle \hspace{1pt} {#1} \hspace{1pt} | \hspace{1.5pt} {#2} \hspace{1pt} \rangle \hspace{0.02em}}
\newcommand{\restr}[2]{{\left.\kern-\nulldelimiterspace #1 \vphantom{\big|} \right|_{#2}}}

\renewcommand{\H}{\mathbb{H}}

\renewcommand{\S}{\mathbb{S}}

\DeclareFontFamily{U}{mathx}{}
\DeclareFontShape{U}{mathx}{m}{n}{<-> mathx10}{}
\DeclareSymbolFont{mathx}{U}{mathx}{m}{n}
\DeclareMathAccent{\widecheck}{0}{mathx}{"71}

\newcommand\altxrightarrow[2][0pt]{\mathrel{\ensurestackMath{\stackengine%
  {\dimexpr#1-7.5pt}{\xrightarrow{\phantom{#2}}}{\scriptstyle\!#2\,}%
  {O}{c}{F}{F}{S}}}}
\newcommand{\isoto}{\altxrightarrow[1pt]{\sim}}

\newcommand{\mysetminusD}{\hbox{\tikz{\useasboundingbox (-0.5pt,-0.5pt) rectangle (5pt,8pt); \draw[line width=0.6pt,line cap=round] (3.5pt,-1.5pt) -- (0,7.25pt);}} \hspace{-1pt}}
\newcommand{\mysetminusT}{\mysetminusD}
\newcommand{\mysetminusS}{\hbox{\tikz{\draw[line width=0.45pt,line cap=round] (2pt,0) -- (-0.5pt,5pt);}} \hspace{1pt}}
\newcommand{\mysetminusSS}{\hbox{\tikz{\draw[line width=0.4pt,line cap=round] (1.5pt,0) -- (0,3pt);}}}

\newcommand{\mysetminus}{\mathbin{\mathchoice{\mysetminusD}{\mysetminusT}{\mysetminusS}{\mysetminusSS}}}

\makeatletter
\newcommand{\extp}{\@ifnextchar^\@extp{\@extp^{\,}}}
\def\@extp^#1{\mathop{\bigwedge\nolimits^{\!#1}}}
\makeatother

\newcommand{\overbar}[1]{\mkern 1.5mu\overline{\mkern-2mu#1\mkern-1.5mu}\mkern 1.5mu}

\makeatletter
\DeclareRobustCommand{\loplus}{\mathbin{\mathpalette\dog@lsemi{+}}}
\DeclareRobustCommand{\roplus}{\mathbin{\mathpalette\dog@rsemi{+}}}

\newcommand{\dog@rsemi}[2]{\dog@semi{#1}{#2}{-90,90}}
\newcommand{\dog@lsemi}[2]{\dog@semi{#1}{#2}{270,90}}
\newcommand{\dog@semi}[3]{%
  \begingroup
  \sbox\z@{$\m@th#1#2$}%
  \setlength{\unitlength}{\dimexpr\ht\z@+\dp\z@\relax}%
  \makebox[\wd\z@]{\raisebox{-\dp\z@}{%
    \begin{picture}(1,1)
    \linethickness{\variable@rule{#1}}
    \roundcap
    \put(0.5,0.5){\makebox(0,0){\raisebox{\dp\z@}{$\m@th#1#2$}}}
    \put(0.5,0.5){\arc[#3]{0.5}}
    \end{picture}%
  }}%
  \endgroup
}
\newcommand{\variable@rule}[1]{%
  \fontdimen8  
  \ifx#1\displaystyle\textfont3\else
    \ifx#1\textstyle\textfont3\else
      \ifx#1\scriptstyle\scriptfont3\else
        \scriptscriptfont3\relax
  \fi\fi\fi
}
\makeatother

\makeatletter
\def\th@plain{%
  \thm@notefont{}
  \itshape 
}
\def\th@definition{%
  \thm@notefont{}
  \normalfont 
}
\makeatother

\usepackage{thm-restate}

\usepackage[capitalise,noabbrev,nameinlink]{cleveref}

\newtheorem{thm}{Theorem}[section]
  \crefname{thm}{Theorem}{Theorems} 
\newtheorem*{thm*}{Theorem}
\newtheorem*{mainthm}{Main theorem}

  \crefname{conj}{Conjecture}{Conjectures}
\newtheorem{prop}[thm]{Proposition}
  \crefname{prop}{Proposition}{Propositions}
\newtheorem{lem}[thm]{Lemma}
  \crefname{lem}{Lemma}{Lemmas}
\newtheorem{cor}[thm]{Corollary}
  \crefname{cor}{Corollary}{Corollaries}

  \crefname{fact}{Fact}{Facts}

  \crefname{openprA}{Open problem}{Open problems}

  \crefname{conjA}{Conjecture}{Conjectures}
\newtheorem{thmA}{Theorem}
  \crefname{thmA}{Theorem}{Theorems}

\theoremstyle{definition}
\newtheorem{defn}[thm]{Definition}
  \crefname{defn}{Definition}{Definitions}
\newtheorem{rem}[thm]{Remark}
  \crefname{rem}{Remark}{Remarks}
\newtheorem{ex}[thm]{Example}
  \crefname{ex}{Example}{Examples}
\newtheorem{obs}[thm]{Observation}
  \crefname{obs}{Observation}{Observations}

\newtheorem*{agr}{Agreement}

\newtheorem*{openq}{Open question}

\theoremstyle{remark}

\numberwithin{equation}{section}
\crefname{subsection}{Subsection}{Subsections}

\usepackage{microtype}

\makeatletter
\newcommand{\customlabel}[2]{%
   \protected@write \@auxout {}{\string \newlabel {#1}{{#2}{\thepage}{#2}{#1}{}} }%
   \hypertarget{#1}{}
}
\makeatother

\begin{document}

\title[Classification of cohomogeneity-one actions]{Classification of cohomogeneity-one actions on symmetric spaces of noncompact type}
\author{V\'ictor Sanmart\'{i}n-L\'{o}pez, Ivan Solonenko}

\address{CITMAga, 15782 Santiago de Compostela, Spain.\newline\indent Department of Mathematics, Universidade de Santiago de Compostela, Spain}
\email{victor.sanmartin@usc.es}

\address{Universit\"{a}t Stuttgart, Institut f\"{u}r Geometrie und Topologie, Stuttgart, Germany}
\email{ivan.solonenko@mathematik.uni-stuttgart.de}

\begin{abstract}
We complete the classification of isometric cohomogeneity-one actions on all symmetric spaces of noncompact type up to orbit equivalence.
\end{abstract}



\thanks{The first author has been supported by Grant PID2022-138988NB-I00 funded by MICIU/AEI/10.13039/ 501100011033 and by ERDF, EU, and by ED431C 2023/31 (Xunta de Galicia).}

\maketitle

\counterwithout{equation}{section}

\section{Introduction}\label{sec:intro}
A proper isometric action of a Lie group $G$ on a connected Riemannian manifold $M$ is said to be of \textit{cohomogeneity one} if it has an orbit of codimension one. Cohomogeneity-one actions have been in the limelight because they can be used to construct Riemannian metrics with special properties: for instance, Einstein metrics, metrics with special holonomy, or metrics of positive sectional curvature (\cite{berard-bergery_cohomogeneity_one,böhm_C1_einstein,bryant_salamon_c1_g2_spin7,grove_wilking_ziller_pos_curv_c1}). One usually starts with a cohomogeneity-one action $G \curvearrowright M$ and dilates the existing metric $g$ along the orbits of $G$ to produce a new $G$-invariant metric, which might retain some properties of $g$ or gain new ones. Since most geometric structures on manifolds are governed by systems of partial differential equations, one can take advantage of the following principle: if a system of PDEs is invariant under a cohomogeneity-one action, it can be reduced to ODEs. For example, this can be applied to the Ricci flow, since it preserves isometries and hence the property of a metric of being invariant under a cohomogeneity-one action (see, e.g., \cite{bettiol_krishnan_4dim_C1_nonnegative,bettiol_krishnan_4dim_C1_positive}). Cohomogeneity-one actions are also an important object of study in submanifold geometry, since their principal orbits are homogeneous hypersurfaces and thus are isoparametric and have constant principal curvatures. Besides, cohomogeneity-one actions are one of the chief examples of hyperpolar (and hence polar) actions. Originally, these attracted interest due to their applications to topology: in \cite{bott_variationally_complete,bott_samelson_variationally_complete}, Bott and Samelson studied so-called variationally complete actions on symmetric spaces, which enabled them to apply Morse theory to the spaces of loops and paths on the ambient space; and in \cite{conlon_hyperpolar_implies_variationally_complete}, Conlon showed that hyperpolar actions are variationally complete.

Due to their prominence, it is natural to ask, what are all possible cohomogeneity-one actions on a given Riemannian manifold? This becomes a classification problem when we set a suitable geometric notion of equivalence: two actions are called \textit{orbit-equivalent} if there is an ambient isometry identifying their orbits. This problem is essentially equivalent to classifying homogeneous hypersurfaces up to isometric congruence. In this paper, by \textit{classification} we mean providing a complete and explicit list of actions up to orbit-equivalence---but we do allow some actions in the resulting list to be orbit-equivalent to each other. Figuring out exactly which ones are is referred to as the \textit{problem of congruence}; we do not pursue it here, although we will make occasional remarks about it along the way. 

The existence of at least one cohomogeneity-one action (from now on, \textit{C1-action}) implies a large isometry group, so it is sensible to first pose the classification problem for homogeneous spaces and in particular their most well-studied subclass of symmetric spaces. Historically, the first two symmetric spaces where the problem was solved were the Euclidean space $\E^n$ and the real hyperbolic space $\R H^n$. In them, it follows from the classification of isoparametric hypersurfaces---which all turn out to be homogeneous---by Somigliana, Levi-Civita, and Segre for $\E^n$ (\cite{somigliana_isoparametric,levi-civita_isoparametric,segre_isoparametric}) and by Cartan for $\R H^n$ (\cite{cartan_isoparametric}). The same approach does not work as smoothly for the sphere: famously, there are plenty of inhomogeneous isoparametric hypersurfaces in $\S^n$, but more importantly, the classification of isoparametric hypersurfaces in spheres still appears to be unfinished (see \cite{siffert_isoparam}). In \cite{hsiang_lawson}, Hsiang and Lawson proved that every C1-action on $\S^n$ is orbit-equivalent to the action induced by the isotropy representation of a symmetric space of rank 2 and dimension $n+1$. C1-actions on the other simply connected compact symmetric spaces of rank one---that is, on the projective spaces over the real division algebras $\C, \H$, and $\Oo$---were classified by Takagi (\cite{takagi_C1}), D'Atri (\cite{d'atri_C1_HP^n}), and Iwata (\cite{iwata_C1_HP^n,iwata_C1_OP^2}). In \cite{kollross_hyperpolar_c1_classification}, Kollross classified C1-actions on all irreducible symmetric spaces of compact type. His proof relied on Dynkin's classification of maximal subgroups of semisimple Lie groups (\cite{dynkin_subgroups_classical_english,dynkin_subalgebras_english}). In particular, Kollross showed that the moduli space of C1-actions on any such space is finite. On reducible symmetric spaces of compact type, the classification problem is still open.

On the other hand, the story of C1-actions on symmetric spaces of \textit{noncompact} type has been much more intricate and convoluted. The reason for this is that, informally speaking, noncompact semisimple Lie groups have ``way more" subgroups than compact ones. Excluding Cartan's classification for $\R H^n$, the story essentially begins with \cite{berndt_bruck}, where Berndt and Br\"{u}ck constructed a plethora of new examples of C1-actions on symmetric spaces of noncompact type and rank one. In particular, they obtained a one-parameter family of mutually orbit-non-equivalent C1-actions on the complex hyperbolic space $\C H^n$, thus showing that the moduli space of such actions can be not only infinite but even nondiscrete. Berndt and Br\"{u}ck's main tool was a novel method that allows one to construct C1-actions on a given space by solving a certain problem in the representation theory of reductive groups. They went on to solve that problem for $\C H^n$ and the Cayley hyperbolic plane $\Oo H^2$ but left the case of the quaternionic hyperbolic space $\H H^n$ open as it led to a pretty elaborate problem in quaternion algebra.

A C1-action on a symmetric space of noncompact type can have either one or zero singular orbits. In \cite{berndt_tamaru_foliations}, Berndt and Tamaru classified C1-actions without singular orbits on all irreducible symmetric spaces of noncompact type. They discovered that there are two distinct types of such actions: a continuous family parametrized by the projectivization of a maximal flat in $M$, and a discrete family parametrized by the simple roots of the restricted root system of $M$. They also solved the problem of congruence for these actions. In a later article \cite{hyperpolar_homogeneous_foliations} with D\'{i}az-Ramos, they generalized that result and classified homogeneous hyperpolar foliations on all symmetric spaces of noncompact type. In particular, that extends the classification in \cite{berndt_tamaru_foliations} to reducible spaces. In \cite{solonenko_foliations}, the second author also solved the problem of congruence in the reducible case. Berndt and Tamaru made another advancement in \cite{berndt_tamaru_totally_geodesic_singular_orbit}: they classified C1-actions with a totally geodesic singular orbit on all irreducible symmetric spaces of noncompact type. They showed that such an orbit is almost always a reflective submanifold and then used the classification of those obtained by Leung in \cite{leung_classification_of_reflective_submanifolds,leung_congruence}. Later in \cite{berndt_tamaru_rank_one}, Berndt and Tamaru proved that, in rank one, every C1-action with a non-totally-geodesic singular orbit arises via the construction method introduced by Berndt and Br\"{u}ck. As a consequence, that completed the classification for $\C H^n$ and $\Oo H^2$. The aforementioned quaternion algebra problem was solved much later in \cite{protohomogeneous}, which concluded the classification for $\H H^n$ and thus all rank-one spaces.

The next big development occurred in \cite{berndt_tamaru_cohomogeneity_one}. In that paper, Berndt and Tamaru generalized the method for constructing C1-actions from \cite{berndt_bruck} to spaces of arbitrary rank---the resulting procedure was given the name \textit{nilpotent construction}. The gist of this method is as follows: given a space $M$, one constructs a series of representations of certain reductive subgroups of the isometry group $I(M)$; in each of those representations, one has to look for subspaces satisfying certain properties---each such subspace gives rise to a C1-action on $M$. We will refer to the problem of finding all such subspaces as the \textit{nilpotent construction problem}. Berndt and Tamaru also invented another procedure---called the \textit{canonical extension}---allowing one to construct actions on spaces of higher rank from those on their certain totally geodesic submanifolds called boundary components. Since a boundary component is itself a symmetric space of noncompact type, one gets an inductive procedure for extending C1-actions into higher rank. The main result of \cite{berndt_tamaru_cohomogeneity_one} asserts that, on any (possibly reducible) symmetric space of noncompact type, if a C1-action has a non-totally-geodesic singular orbit, it must arise via either the canonical extension or the nilpotent construction. One of the big underpinnings of that result was Mostow's description of maximal subalgebras of real semisimple Lie algebras in \cite{mostow_maximal_subgroups}. By that point, there remained only two obstacles on the path to full classification. The first of them was understanding C1-actions with a totally geodesic singular orbit on \textit{reducible} spaces. This problem was bypassed in \cite{DR_DV_Otero_C1} by means of Dynkin's structural result \cite[Th.\hspace{2pt}15.1]{dynkin_subalgebras_english} (see also \cite[Th.\hspace{2pt}2.1]{kollross_hyperpolar_c1_classification}) on maximal subalgebras of semisimple Lie algebras---which largely complements the one by Mostow.

To complete the classification, one is left to solve the nilpotent construction problem. To this day, this has only been done in rank one, on some spaces of rank two (\cite{berndt_tamaru_cohomogeneity_one,berndt_DV_C1_rank_2,solonenko_hypersurfaces}), as well as on $\SL(n,\R)/\SO(n), \, n \ge 3$ (\cite{DR_DV_Otero_C1}). The aim of the current article is to resolve this problem for all symmetric spaces of noncompact type and thus conclude the classification---see the \hyperlink{thm:main}{Main theorem} below.

We now briefly introduce the necessary context and lay out our main result. (For more details, see Sections \ref{sec:preliminaries} and \ref{sec:c1-actions}.) Let $M = G/K$ be a symmetric space of noncompact type, where $G = I^0(M)$ is the identity component of the isometry group of $M$ and $K$ is the isotropy group of some point $o \in M$. The isometry Lie algebra $\mk{g}$ splits into a Cartan decomposition $\mk{g} = \mk{k} \oplus \mk{p}$, where $\mk{k} = \Lie(K)$. If we write $\uptheta$ for the corresponding Cartan involution and $B$ for the Killing form of $\mk{g}$, then we can define an inner product on $\mk{g}$ by the formula $B_\uptheta(X,Y) = -B(X,\uptheta Y)$. Pick a maximal abelian subspace $\mk{a}$ in $\mk{p}$. This induces the restricted root space decomposition $\mk{g} = \mk{g}_0 \oplus \bigoplus_{\upalpha \in \Upsigma} \mk{g}_\upalpha$ with respect to the restricted root system $\Upsigma \subset \mk{a}^*$. We have $\mk{g}_0 = \mk{k}_0 \oplus \mk{a}$, where $\mk{k}_0$ is the centralizer of $\mk{a}$ in $\mk{k}$. Fix a choice of simple roots $\Uplambda \subset \Upsigma$; this yields a nilpotent subalgebra $\mk{n} = \bigoplus_{\upalpha \in \Upsigma^+} \mk{g}_\upalpha$ of $\mk{g}$ and an Iwasawa decomposition $\mk{g} = \mk{k} \oplus \mk{a} \oplus \mk{n}$.

Take a one-dimensional subspace $\ell \subseteq \mk{a}$ and consider $\mk{h}_\ell = (\mk{a} \ominus \ell) \oplus \mk{n}$, where the first summand stands for the orthogonal complement of $\ell$ in $\mk{a}$. This is a Lie subalgebra of $\mk{g}$, and its corresponding Lie subgroup $H_\ell$ acts on $M$ with cohomogeneity one and no singular orbits---hence its orbits induce a Riemannian foliation. Similarly, take a simple root $\upalpha_i \in \Uplambda$ and any one-dimensional subspace $\ell_{\upalpha_i} \subseteq \mk{g}_{\upalpha_i}$. Then $\mk{h}_{\upalpha_i} = \mk{a} \oplus (\mk{n} \ominus \ell_{\upalpha_i})$ is a Lie subalgebra of $\mk{g}$ whose corresponding Lie subgroup $H_{\upalpha_i}$ also acts on $M$ with cohomogeneity one and no singular orbits. A different choice of $\ell_{\upalpha_i}$ in $\mk{g}_{\upalpha_i}$ leads to an orbit-equivalent action.

Suppose for a moment that  $M = M_1 \times M_2$, where $M_1$ and $M_2$ are homothetic symmetric spaces of noncompact type and rank one, and write $\mk{g}_i$ for their isometry Lie algebras. Pick an isomorphism $\upvarphi \colon \mk{g}_1 \isoto \mk{g}_2$ and consider the diagonal subalgebra $\mk{g}_\upvarphi = \set{X + \upvarphi(X) \mid X \in \mk{g}_1} \subset \mk{g}_1 \oplus \mk{g}_2 = \mk{g}$. Then the corresponding connected Lie subgroup $G_\upvarphi \subset G_1 \times G_2 = G$ (here $G_i = I^0(M_i)$) acts on $M$ with cohomogeneity one and a totally geodesic singular orbit---we call this a \textit{diagonal C1-action} on $M$.

Coming back to a general $M$, take any subset $\Upphi \subseteq \Uplambda$ of simple roots. We can define a subspace $\mk{a}_\Upphi = \bigcap_{\upalpha \in \Upphi} \Ker(\upalpha)$ of $\mk{a}$ together with its centralizer $\mk{l}_\Upphi = Z_\mk{g}(\mk{a}_\Upphi)$ and the subalgebra $\mk{m}_\Upphi = \mk{l}_\Upphi \ominus \mk{a}_\Upphi$. The $o$-orbit of the connected Lie subgroup corresponding to $\mk{m}_\Upphi$ is called a \textit{boundary component} of $M$ and denoted by $B_\Upphi$. This is a totally geodesic submanifold and a symmetric space of noncompact type in its own right. Its restricted root system $\Upsigma_\Upphi$ can be realized as the root subsystem of $\Upsigma$ generated by $\Upphi$. Note that $\Upsigma_\Upphi$ inherits a choice $\Upphi$ of simple roots. This allows us to define a subalgebra $\mk{n}_\Upphi = \bigoplus_{\uplambda \in \Upsigma^+ \mysetminus \Upsigma_\Upphi^+} \mk{g}_\uplambda$. The isometry Lie algebra $\mk{g}'_\Upphi$ of $B_\Upphi$ can be realized as the subalgebra of $\mk{g}$ generated by $\mk{g}_\upalpha$ as $\upalpha$ runs through $\Upsigma_\Upphi$. The connected Lie subgroup of $G$ corresponding to $\mk{g}'_\Upphi$ is denoted by $G'_\Upphi$ and is a finite covering of $I^0(B_\Upphi)$. Let $H_\Upphi \subseteq G'_\Upphi$ be a Lie subgroup with Lie algebra $\mk{h}_\Upphi \subseteq \mk{g}'_\Upphi$. We can then define a subalgebra $\mk{h}_\Upphi^\Uplambda = \mk{h}_\Upphi \oplus \mk{a}_\Upphi \oplus \mk{n}_\Upphi$ of $\mk{g}$, whose corresponding connected Lie subgroup $H_\Upphi^\Uplambda \subseteq G$ acts on $M$ with the same cohomogeneity as that of $H_\Upphi \curvearrowright B_\Upphi$. This procedure is called the \textit{canonical extension}.

Write $\Uplambda = \set{\upalpha_1, \ldots, \upalpha_r}$, let $H^1, \ldots, H^r$ be the dual basis for $\mk{a}$, and denote $H^\Upphi = \sum_{\upalpha_i \in \Upphi} H^i$. The subalgebra $\mk{n}_\Upphi$ is graded: $\mk{n}_\Upphi = \bigoplus_{\upnu \ge 1} \mk{n}_\Upphi^\upnu$, where $\mk{n}_\Upphi^\upnu$ is the eigenspace of $\ad(H^\Upphi)$ in $\mk{n}_\Upphi$ with eigenvalue $\upnu$. Define $L_\Upphi = Z_G(\mk{a}_\Upphi)$ and $K_\Upphi = K \cap L_\Upphi$. The orbit $L_\Upphi \cdot o$ contains $B_\Upphi$ and is denoted by $F_\Upphi$. The adjoint representation of $K_\Upphi$ on $\mk{g}$ is orthogonal and preserves the grading on $\mk{n}_\Upphi$. For every subspace $\mk{w} \subseteq \mk{n}_\Upphi^1$, the sum $\mk{n}_{\Upphi,\mk{w}} = \mk{w} \oplus \bigoplus_{\upnu \ge 2} \mk{n}_\Upphi^\upnu$ is a subalgebra, and its corresponding Lie subgroup is denoted by $N_{\Upphi,\mk{w}}$. Such a subspace $\mk{w}$ is called \textit{admissible} if $N_{L_\Upphi}(\mk{w})$ acts transitively on $F_\Upphi$ and \textit{protohomogeneous} if $N_{K_\Upphi}(\mk{w})$ acts transitively on the unit sphere in $\mk{v} = \mk{n}_\Upphi^1 \ominus \mk{w}$ (see Remark \ref{rem:w_vs_v}). If $\mk{w}$ satisfies these conditions and has $\codim_{\mk{n}_\Upphi^1}(\mk{w}) \ge 2$, then the subgroup $N^0_{L_\Upphi}(\mk{w})N_{\Upphi,\mk{w}}$ acts on $M$ with cohomogeneity one and has a singular orbit of codimension equal to $\codim_{\mk{n}_\Upphi^1}(\mk{w})$. This is called the \textit{nilpotent construction}, and the problem of finding all such subspaces $\mk{w}$ is called the \textit{nilpotent construction problem}. Note that if two admissible and protohomogeneous subspaces $\mk{w}, \mk{w}' \subseteq \mk{n}_\Upphi^1$ differ by an element of $K_\Upphi$, they give rise to orbit-equivalent actions. For reasons to be explained in Section \ref{sec:c1-actions} (see Theorem \ref{thm:classification_of_c1_actions}), one only has to consider this problem when $\Upphi$ is of the form $\Uplambda \mysetminus \set{\upalpha_j}$ for some simple root $\upalpha_j$. In that case, it is customary to replace $\Upphi$ with $j$ in subscripts and superscripts---e.g., $\mk{a}_j$ instead of $\mk{a}_\Upphi$, $\mk{n}_{j,\mk{w}}$ instead of $\mk{n}_{\Upphi,\mk{w}}$, and so on. The nilpotent construction problem has only been solved for a handful of spaces, and only on five of them it has given rise to new C1-actions not obtainable via other methods described above. Let us go through each of them briefly.

For rank-one spaces, we have $\Uplambda = \set{\upalpha_1}$ and $\mk{n}_1^1 = \mk{g}_{\upalpha_1}$, and every subspace of $\mk{n}_1^1$ is trivially admissible. In this case, it is more convenient to describe $\mk{w}$ in terms of its orthogonal complement $\mk{v}$. If $M = \C H^{n+1}$, we have $\mk{n}_1^1 \simeq \C^n, \, K_1 \simeq \U(n)$, and the representation of $K_1$ on $\mk{n}_1^1$ is the tautological representation of the unitary group. It was shown in \cite{berndt_bruck} that $\mk{w} \subseteq \mk{n}_1^1$ is protohomogeneous if and only if for every nonzero $v \in \mk{v}$, the angle between $Jv$ (here $J$ is the complex structure on $\mk{n}_1^1 \simeq \C^n$) and $\mk{v}$ is a constant not depending on $v$; this angle is then called the \textit{K\"{a}hler angle} of $\mk{v}$. Such a subspace $\mk{v}$---and thus the resulting C1-action---is determined by its K\"{a}hler angle and dimension up to congruence by $\U(n)$. Disregarding those subspaces that lead to C1-actions with a totally geodesic singular orbit, we obtain from \cite{berndt_bruck} that the moduli space of C1-actions on $\C H^{n+1}$ coming from the nilpotent construction and no other method is given by $\bm{\mc{M}_{\C H^{n+1}}} = (0,\frac{\uppi}{2}) \times \set{2, 4, \ldots, 2 \lfloor \frac{n}{2} \rfloor} \hspace{1pt} \sqcup \hspace{1pt} \set{\frac{\uppi}{2}} \times \set{2, 3, \ldots, n}$. If $M = \H H^{n+1}$, the representation $K_1 \curvearrowright \mk{n}_1^1$ is equivalent to the standard representation $\Sp(n)\Sp(1) \curvearrowright \H^n$. In that same paper, Berndt and Br\"{u}ck devised a generalization of the notion of K\"{a}hler angle to real subspaces of $\H^n$, which they called the \textit{quaternion-K\"{a}hler angle}. This time, this is a \textit{triple} of angles $(\upvarphi_1, \upvarphi_2, \upvarphi_3) \in [0,\frac{\uppi}{2}]^3$. As was shown in \cite{protohomogeneous}, for every protohomogeneous subspace $\mk{w}$ of $\H^n$, $\mk{v}$ has constant quaternion-K\"{a}hler angle and is \textit{almost} determined by it together with its dimension up to $\Sp(n)\Sp(1)$: for certain angles and dimensions, there might be two noncongruent subspaces. Therefore, the moduli space $\bm{\mc{M}_{\H H^{n+1}}}$ of C1-actions on $\H H^{n+1}$ arising from the nilpotent construction but no other method is a certain subset of the disjoint union of a collection of cubes $[0,\frac{\uppi}{2}]^3$. That subset (as well as the corresponding subspaces) was given an explicit description in \cite{protohomogeneous} (see Theorem A therein). If $M = \Oo H^2$, the representation $K_1 \curvearrowright \mk{n}_1^1$ is equivalent to the spin representation $\Spin(7) \curvearrowright \R^8$. The corresponding protohomogeneous subspaces were classified in \cite{berndt_bruck,berndt_tamaru_rank_one}, and the resulting moduli space can be described as $\bm{\mc{M}_{\Oo H^2}} = \set{2,3,6,7} \hspace{1pt} \sqcup \hspace{1pt} [0,1] \times \set{4}$. For each of these rank-one spaces, the C1-action corresponding to $\mk{w}$ is given on the level of Lie algebras by $ \mk{h}_{1,\mk{w}} = N_{\mk{k}_0}(\mk{w}) \oplus \mk{a} \oplus \mk{w} \oplus \mk{g}_{2\upalpha_1}$. Finally, consider the rank-two spaces $\mm{G}_2^2/\SO(4)$ and $\mm{G}_2(\C)/\mm{G}_2$. Let us write $\Uplambda = \set{\upalpha_1, \upalpha_2}$, where $\upalpha_2$ is the short simple root. As shown in \cite{berndt_tamaru_cohomogeneity_one}, for either space, the nilpotent construction yields only one new C1-action, which corresponds to the choice $j = 2$ and $\mk{w} = \set{0}$; the corresponding subgroup $H_{2,0}$ has Lie algebra $\mk{h}_{2,0} = \mk{l}_2 \oplus \mk{n}_2^2 \oplus \mk{n}_2^3$.

Coming back to the general case one more time, suppose $M = \prod_i M_i$ is the de Rham decomposition of $M$. The group $G$ decomposes as $G = \prod_i G_i$, where $G_i = I^0(M_i)$. An isometric C1-action $H \curvearrowright M$ is called \textit{decomposable} if it is orbit-equivalent to the action of a subgroup of $G$ of the form $H_i \times \prod_{l \ne i} G_l$ for some $i$, where $H_i \subset G_i$. In this case, $H_i$ acts on the de Rham factor $M_i$ also with cohomogeneity one. We are now in a position to formulate the main result of the article.

\begin{mainthm}\hypertarget{thm:main}{}
    Let $M = G/K$ be a symmetric space of noncompact type and $H$ a connected Lie group acting on $M$ properly and isometrically. Then $H$ acts with cohomogeneity one if and only if its action is orbit-equivalent to one of the following:
    
    \begin{enumerate}[\normalfont (a)]
        \item\customlabel{thm:main:a}{a} The action of $H_\ell$ for some one-dimensional linear subspace $\ell \subseteq \mk{a}$.
        \item\customlabel{thm:main:b}{b} The action of $H_{\upalpha_i}$ for some simple root $\upalpha_i \in \Uplambda$.
        \item\customlabel{thm:main:c}{c} The canonical extension of a C1-action with a totally geodesic singular orbit on an irreducible boundary component $B_\Upphi$ of $M$.
        \item\customlabel{thm:main:d}{d} The canonical extension of a diagonal C1-action on a reducible rank-2 boundary component $B_\Upphi$ of $M$ whose de Rham factors are homothetic.
        \item\customlabel{thm:main:e}{e} A decomposable action $H_i \times \prod_{l \ne i} G_l \curvearrowright M_i \times \prod_{l \ne i} M_l = M$, where the de Rham factor $M_i$ and the action $H_i \curvearrowright M_i$ are one of the following:

        \begin{enumerate}[\normalfont (1)]
            \item\customlabel{thm:main:e:1}{1} $M_i$ is isometric to
            
            \begin{itemize}
                \item $\Gr^*(k,\C^{2k+n}) \; (k \ge 1, n \ge 2)$,
                \item $\SO(2n+1, \H)/\U(2n+1) \; (n \ge 2)$, or
                \item $\mm{E}_6^{-14}/\Spin(10)\U(1)$,
            \end{itemize}
            and $H_i \curvearrowright M_i$ is the canonical extension of an action on the boundary component $\C H^{n+1}$ {\normalfont(}resp., $\C H^3$ or $\C H^5${\normalfont)} of $M_i$ belonging to $\mc{M}_{\C H^{n+1}}$ {\normalfont(}resp., $\mc{M}_{\C H^3}$ or $\mc{M}_{\C H^5}${\normalfont)}.
            \item\customlabel{thm:main:e:2}{2} $M_i \simeq \Gr^*(k,\H^{2k+n}) \; (k,n \ge 1)$ and $H_i \curvearrowright M_i$ is the canonical extension of an action on the boundary component $\H H^{n+1}$ of $M_i$ belonging to $\mc{M}_{\H H^{n+1}}$.
            \item\customlabel{thm:main:e:3}{3} $M_i \simeq \Oo H^2$ and $H_i \curvearrowright M_i$ is an action belonging to $\mc{M}_{\Oo H^2}$.
            \item\customlabel{thm:main:e:4}{4} $M_i \simeq \mm{G}_2^2/\SO(4)$ and $H_i = H_{2,0}$.
            \item\customlabel{thm:main:e:5}{5} $M_i \simeq \mm{G}_2(\C)/\mm{G}_2$ and $H_i = H_{2,0}$.
        \end{enumerate}
    \end{enumerate}
\end{mainthm}

Notice that the actions in \eqref{thm:main:b} and \eqref{thm:main:c} are also decomposable (see \cite[Th.\hspace{2pt}C]{DR_DV_Otero_C1}). It is worth pointing out that this theorem does not give a complete description of the moduli space of C1-actions yet, as one still needs to solve the problem of congruence for the actions described in \eqref{thm:main:a}-\eqref{thm:main:e}. We will make some occasional remarks about that in Section \ref{sec:c1-actions} (see also \cite[Rem.\hspace{2pt}1.1]{DR_DV_Otero_C1}), but we will not attempt to resolve the problem of congruence in general, since that would go beyond the scope of the paper.

Let us briefly lay out our strategy. In order to prove the \hyperlink{thm:main}{Main theorem}, we will have to resolve the nilpotent construction problem. As our first step, we will tackle the admissibility condition on the subspace $\mk{w}$---which has to do with \textit{transitive} actions on symmetric spaces. In the compact irreducible case, such actions were explicitly classified by Onishchik in \cite{onishchik_transitive}. There is little hope of getting a similar result in the noncompact case, as the list would be infinite for most spaces. As a substitute, we obtain a general structural result asserting that any Lie group acting transitively and isometrically on a symmetric space of noncompact type contains a skeleton of a very specific form.

\begin{thmA}\label{thm:B}\hspace{-5pt}\footnote{In the special cases of $\R H^n$ and $\C H^n$, a similar result was obtained in \cite[Th.\hspace{2pt}3.1, 4.1]{transitive_CH^n}}
    Let $M = G/K$ be a symmetric space of noncompact type and $H \subseteq G$ a Lie subgroup acting on $M$ transitively. Then there exist an Iwasawa decomposition $\mk{g} = \mk{k} \oplus \mk{a} \oplus \mk{n}$ and a maximal abelian subspace $\mk{t}_0 \subseteq \mk{k}_0$ such that the Lie algebra $\mk{h}$ of $H$ contains a solvable subalgebra of the form $V \oplus \mk{n}$, where $V \subseteq \mk{t}_0 \oplus \mk{a}$ projects surjectively onto $\mk{a}$.
\end{thmA}

Let us write $A$ and $N$ for the connected Lie subgroups of $G$ corresponding to $\mk{a}$ and $\mk{n}$, respectively. We have $G = KAN$, which is also called an Iwasawa decomposition. Informally speaking, Theorem \ref{thm:B} says that every transitive group of isometries of $M$ contains $N$ and ``almost contains" $A$ for some Iwasawa decomposition of $G$. In regard to the nilpotent construction, this theorem will let us show that the subspace $\mk{w}$ can always be assumed to respect the restricted root space decomposition of $\mk{g}$, which will significantly simplify the problem. Together with the protohomogeneity condition---or rather its Lie-algebraic reformulation---this will let us establish deep connections between the subspace $\mk{w}$ and root system $\Upsigma$ of $M$. For our next step, we observe that we should only care about actions that do not arise via any nontrivial canonical extension. We reinterpret this assumption in terms of how $\mk{w}$ interacts with the root spaces corresponding to the simple roots---which leads to very strong restrictions on $\Upsigma$, prohibits most choices of $\upalpha_j$, and rules out many spaces altogether solely based on their root systems. We then proceed to carry out a careful analysis of the singular orbit of $H_{j,\mk{w}}$ and its extrinsic geometry within the framework of the solvable model of $M$ (once again, see Section \ref{sec:preliminaries} for definitions and details). This will give us a much better understanding of whether the singular orbit is totally geodesic---and if it is, such an action can also be disregarded. Having done that, we will be able to rule out most of the remaining spaces simply by looking at their root multiplicities. At that stage, we will only be left to deal with spaces whose root systems are of type $\mm{B}_r$ or $\mm{BC}_r$. For these, we will have to resort to a more brute-force approach. We will examine the shape operators of the singular orbit using the Lie bracket relations between the root spaces of $\mk{g}$---and deduce that this orbit is bound to be totally geodesic, provided $r > 1$. Eventually, all these considerations will let us conclude that the nilpotent construction produces no actions other than the known examples or their canonical extensions, which will prove the \hyperlink{thm:main}{Main theorem}.

In our investigation of the nilpotent construction problem, we obtain two incidental results of independent interest. One of them is Theorem \ref{thm:B}, which we have already introduced above. The other side result concerns certain common classes of submanifolds of $M$ and gives them a useful uniform description---also in terms of an Iwasawa decomposition.

\begin{thmA}\label{thm:C}
    Let $M = G/K$ be a symmetric space of noncompact type and $S \subseteq M$ either of the following:
    \begin{enumerate}[\normalfont (a)]
        \item\customlabel{thm:C:a}{a} a complete connected totally geodesic submanifold,
        \item\customlabel{thm:C:b}{b} a singular orbit of a proper isometric C1-action.
    \end{enumerate}
    Then there exist an Iwasawa decomposition $G = KAN$ and a connected Lie subgroup $H \subseteq AN$ of the form $H = (H \cap A) \ltimes (H \cap N)$ such that $S = H \cdot o$.
\end{thmA}

Since $AN$ acts on $M$ simply transitively, we can think of $M$ as a solvable Lie group with a left-invariant metric. From this viewpoint, Theorem \ref{thm:C} asserts that every submanifold of $M$ of type \eqref{thm:C:a} or \eqref{thm:C:b} can be realized as a Lie subgroup---and is thus a solvmanifold. This makes studying extrinsic and intrinsic geometry of the submanifold considerably easier (see Subsection \ref{sec:preliminaries:extrinsic_geometry_cpc}). We expect that part \eqref{thm:C:a} of the theorem might aid the ongoing efforts to classify totally geodesic submanifolds in symmetric spaces (see \cite{t.g._exceptional_ss} and further references therein). Note that submanifolds in \eqref{thm:C:a} and \eqref{thm:C:b} are examples of homogeneous minimal submanifolds (for \eqref{thm:C:b}, this is shown in \cite[Prop.\hspace{2pt}3]{berndt_bruck}).

The article is organized as follows. Section \ref{sec:preliminaries} consists of preliminaries necessary for understanding everything that follows. In Section \ref{sec:c1-actions}, we talk in depth about the known types and methods of construction of C1-actions on symmetric spaces of noncompact type and discuss the current state of their classification. In Section \ref{sec:admissibility}, we deal with the admissibility condition in the nilpotent construction and prove Theorem \ref{thm:B}. This is the backbone of the whole paper that essentially renders the nilpotent construction solvable. Section \ref{sec:t.g._and_singular_orbit} is dedicated to the proof of Theorem~\ref{thm:C}. Finally, in Section \ref{sec:solving_the_NC}, we resolve the nilpotent construction problem and prove the \hyperlink{thm:main}{Main theorem}. 

\textbf{Acknowledgments.} The authors are indebted to Miguel Dom\'{i}nguez-V\'{a}zquez for his countless comments and suggestions, as well as to Tom\'{a}s Otero for his fruitful feedback on Section \ref{sec:admissibility}. The second author is particularly grateful to Miguel for organizing his visit to Santiago de Compostela in May 2024, which largely kick-started the project that led to this paper. I am also thankful to J\"{u}rgen Berndt for introducing me to the study of C1-actions and guiding me through it over the years of my PhD. Finally, we would like to thank Hiroshi Tamaru, Alberto Rodr\'{i}guez-V\'{a}zquez, and Jos\'{e} Carlos D\'{i}az-Ramos for sharing their insights and valuable advice.

\section{Preliminaries}\label{sec:preliminaries}
This section serves as a collection of preliminaries required to understand the rest of the article. We begin with a brief refresher on some aspects of the theory of symmetric spaces of noncompact type in Subsection \ref{sec:preliminaries:symmetric_spaces_of_nc}. After that, in Subsection \ref{sec:preliminaries:parabolic}, we give a quick overview of the theory of parabolic subgroups of real semisimple Lie groups. Lastly, in Subsection \ref{sec:preliminaries:extrinsic_geometry_cpc}, we talk about a certain type of homogeneous submanifolds in symmetric spaces of noncompact type, whose extrinsic geometry such as shape operators can be studied effectively by means of the Iwasawa and restricted root space decompositions.

\subsection{Symmetric spaces of noncompact type}\label{sec:preliminaries:symmetric_spaces_of_nc}

Let $M$ be a symmetric space of noncompact type represented by a Riemannian symmetric pair $(G,K)$ with $G$ a connected noncompact semisimple Lie group and $K \subset G$ a maximal compact subgroup. We will usually write this simply as ``let $M = G/K$ be a symmetric space of noncompact type".

\begin{rem}
    By representing $M$ as $G/K$, we have a distinguished base point $eK \in M$, whose isotropy group is $K$. However, in certain circumstances, it might be beneficial to change the base point to a new one, say $o \in M$, that is better positioned. In that case, if there is no confusion, we will sometimes commit a slight abuse of notation and reuse the letter $K$ for the isotropy group of the new point $o$. This is equivalent to replacing $K$ with a conjugate subgroup but keeping the letter.
\end{rem}

The \textit{ineffectiveness kernel} of $(G,K)$ is the subgroup $I$ of elements of $G$ that act on $M$ trivially. Note that this is a closed normal subgroup. When unambiguous, we will refer to its Lie algebra $\mk{i}$, which is an ideal of $\mk{g} = \Lie(G)$, as the \textit{ineffectiveness kernel} as well.

\begin{agr}
    From now on, whenever we have a symmetric space $M$ of noncompact type and write it as $M = G/K$, we assume by default that $(G,K)$ is an \textit{almost effective} Riemannian symmetric pair. This means that the ineffectiveness kernel $I \subset G$ is finite. It follows a posteriori that $K$ is compact and $I$ coincides with the center of $G$. We refer to \cite[Sect.\hspace{2pt}2.4.1]{solonenko_thesis} for details.
\end{agr}

Fix a base point $o \in M$. The subgroup $K = G_o$ is the fixed-point set of a unique involutive automorphism $\Uptheta$ of $G$. The corresponding automorphism $\uptheta$ of $\mk{g}$ is a Cartan involution, and we will denote its corresponding Cartan decomposition by $\mk{g} = \mk{k} \oplus \mk{p}$. Here $\mk{k} = \Ker(\Id_\mk{g}\hspace{-1pt} -\hspace{1.5pt} \uptheta) = \Lie(K)$ and $\mk{p} = \Ker(\Id_\mk{g}\hspace{-1pt} +\hspace{1.5pt} \uptheta)$. The reductive complement $\mk{p}$ gets naturally identified with $T_oM$ by sending $X \in \mk{p}$ to the value of its corresponding Killing vector field $\wh{X}$ at $o$. Having fixed $\uptheta$, we have a $K$-invariant inner product on $\mk{g}$ given by $B_\uptheta(X,Y) = -B(X,\uptheta Y)$, where $B$ is the Killing form. Throughout the article, we will use $B_\uptheta$ as a default inner product on $\mk{g}$. We will also denote it by $\cross{-}{-}_{B_\uptheta}$ or even $\cross{-}{-}$ if there is no ground for confusion. Given two subspaces $U \subseteq V \subseteq \mk{g}$, $V \ominus U$ stands for the orthogonal complement of $U$ inside $V$. The summands $\mk{k}$ and $\mk{p}$ are orthogonal with respect to both $B_\uptheta$ and $B$; $B_\uptheta$ coincides with $B$ on $\mk{p} \times \mk{p}$ and equals $-B$ on $\mk{k} \times \mk{k}$. Being a $K$-invariant inner product on $\mk{p} \cong T_oM$, $\restr{B_\uptheta}{\mk{p} \times \mk{p}}$ can be extended to a $G$-invariant Riemannian metric $g_B$ on $M$. This is called the \textit{Killing metric}; it is symmetric, does not depend on the choice of a base point, and can be obtained by rescaling the original metric $g$ along each de Rham factor of $M$---in fact, any $G$-invariant metric on $M$ can be obtained in this way (and is thus also symmetric). The Killing metric enjoys the following property: if two de Rham factors of $M$ are homothetic, then they are isometric with respect to $g_B$. In general, a $G$-invariant metric on $M$ satisfying this property is said to be \textit{almost Killing}. Almost Killing metrics have the largest isometry group (all equal to $I(M,g_B)$), of which $I(M,g)$ is an open subgroup; in particular, we always have $I^0(M,g) = I^0(M,g_B)$. It can be shown that $I(M,g_B)$ is naturally isomorphic to $\Aut(\mk{g})$ (see \cite[Sect.\hspace{2pt}3]{solonenko_automorphisms} for details).

Pick a maximal abelian subspace $\mk{a}$ in $\mk{p}$. Given $\upalpha \in \mk{a}^*$, we let
$$
\mk{g}_\upalpha = \set{X \in \mk{g} \mid [H,X] = \upalpha(H)X \; \text{for every} \; H \in \mk{a}},
$$
and we also put $\Upsigma = \left\{\upalpha \in \mk{a}^* \mysetminus \{0\} \mid \mk{g}_\upalpha \ne \{0\}\right\}$. This set is called the restricted root system of $\mk{g}$ (or $M$) and its elements are called restricted roots. These induce a restricted root space decomposition $\mk{g} = \mk{g}_0 \oplus \bigoplus_{\upalpha \in \Upsigma} \mk{g}_\upalpha$. The summands of this decomposition are all pairwise orthogonal with respect to $B_\uptheta$, and we can also decompose $\mk{g}_0 = \mk{k}_0 \oplus \mk{a}$, where $\mk{k}_0 = Z_\mk{k}(\mk{a})$. The dimension of $\mk{g}_\upalpha$ is called the \textit{multiplicity} of the root $\upalpha$. It is customary to denote $\Upsigma_0 = \Upsigma \cup \set{0}$. For any $\upalpha, \upbeta \in \Upsigma_0, \, \upalpha \ne - \upbeta$, we have $[\mk{g}_\upalpha, \mk{g}_\upbeta] = \mk{g}_{\upalpha + \upbeta}$. This implies that the adjoint representation of $\mk{k}_0$ on $\mk{g}$ preserves each root space $\mk{g}_\upalpha$. We can restrict $B_\uptheta$ to an inner product on $\mk{a}$ and use it to identify $\mk{a}$ with $\mk{a}^*$ in a standard way, which in turn induces an inner product on $\mk{a}^*$. With respect to this inner product, $\Upsigma$ becomes a root system in $\mk{a}^*$. We can pull each root $\upalpha \in \Upsigma$ back under the isomorphism $\mk{a} \isoto \mk{a}^*$ and define a vector $H_\upalpha \in \mk{a}$; in other words, $H_\upalpha$ is determined by $\cross{H_\upalpha}{H} = \upalpha(H) \defqe \bilin{\upalpha}{H}$ for every $H \in \mk{a}$.

We make a choice of positive roots $\Upsigma^+ \subset \Upsigma$ and write $\Uplambda = \set{\upalpha_1, \ldots, \upalpha_r}$ for the corresponding set of simple roots. This allows us to define a nilpotent subalgebra $\mk{n} = \bigoplus_{\upalpha \in \Upsigma^+} \mk{g}_\upalpha$, which leads to a vector space decomposition $\mk{g} = \mk{k} \oplus \mk{a} \oplus \mk{n}$, also known as an Iwasawa decomposition. We will write $\Upsigma^-$ for the set $-\Upsigma^+$ of negative roots and $\DD$ for the Dynkin diagram of $\Upsigma$. Any root $\upalpha \in \Upsigma^+$ can be written as $\upalpha = \sum_{i=1}^r n_i \upalpha_i$ with $n_i \in \Z_{\ge 0}$, and the positive integer $\height(\upalpha)= \sum_{i=1}^r n_i$ is called the \textit{height} of $\upalpha$. Since $\upalpha_1, \ldots, \upalpha_r$ form a basis for $\mk{a}^*$, we can introduce the dual basis $H^1, \ldots, H^r$ for $\mk{a}$; these vectors are determined by $\cross{H^i}{H_{\upalpha_j}} = \bilin{H^i}{\upalpha_j} = \updelta_{ij}$. By definition, if $\upalpha = \sum_{i=1}^r n_i \upalpha_i$, then $\bilin{H^i}{\upalpha} = n_i$. Given a root $\upalpha \in \Upsigma$, we denote $\mk{k_\upalpha} = \mk{k} \cap (\mk{g}_\upalpha \oplus \mk{g}_{-\upalpha})$ and $\mk{p_\upalpha} = \mk{p} \cap (\mk{g}_\upalpha \oplus \mk{g}_{-\upalpha})$. This leads to decompositions $\mk{k} = \mk{k}_0 \oplus \bigoplus_{\upalpha \in \Upsigma^+} \mk{k}_\upalpha$ and $\mk{p} = \mk{a} \oplus \bigoplus_{\upalpha \in \Upsigma^+} \mk{p}_\upalpha$. Notice that we have $\mk{g}_\upalpha \oplus \mk{g}_{-\upalpha} = \mk{k}_\upalpha \oplus \mk{p}_\upalpha$ for every $\upalpha \in \Upsigma$.

We write $A$ and $N$ for the connected Lie subgroups of $G$ corresponding to $\mk{a}$ and $\mk{n}$, respectively. These subgroups are closed and simply connected. The multiplication in $G$ yields a diffeomorphism $K \times A \times N \isoto G$, which is also called an Iwasawa decomposition. The Lie group $A$ normalizes $N$, so the semidirect product $A \ltimes N$ is a solvable subgroup of $G$. (We will normally shorten $A \ltimes N$ to just $AN$.) The action $AN \curvearrowright M$ is simply transitive and the metric pulled back from $M$ to $AN$ along the diffeomorphism $AN \isoto M, \, an \mapsto an \cdot o,$ is left-invariant, hence $M$ can be thought of as a simply connected solvable Lie group with a left-invariant metric---we call this the \textit{solvable model} of $M$. The orbit $A \cdot o$ is a maximal flat in $M$ and it is isometric to the Euclidean space $\E^r$, where $r = \dim(\mk{a})$ is the rank of $M$.

\subsection{Parabolic subalgebras and subgroups}\label{sec:preliminaries:parabolic}

The theory of parabolic subgroups of real semisimple Lie groups plays a pivotal role in the geometry of symmetric spaces of noncompact type because it yields a number of important decomposition results for the space itself, as well as for its isometry group. Among other things, these decompositions can be used to construct various homogeneous objects---like submanifolds or foliations---on the ambient space. They also underpin the two methods of constructing C1-actions with a non-totally-geodesic singular orbit. Here we lay out a brief exposition of the theory and refer to \cite{berndt_tamaru_cohomogeneity_one,solonenko_hypersurfaces,solonenko_thesis} for more details. For a deeper introduction, see \cite[Ch.\hspace{0.2pt}VII, Sect.\hspace{2pt}7]{knapp} and \cite[Sect.\hspace{2pt}13.2]{submanifolds_&_holonomy}.

Let $M = G/K$ be a symmetric space of noncompact type with a fixed choice of $o \in M, \mk{a} \subset \mk{p},$ and $\Upsigma^+ \subset \Upsigma$ as above. Fix any subset~$\Upphi \subseteq \Uplambda$.

\begin{itemize}
    \item We write $\Upsigma_\Upphi$ for the root subsystem of $\Upsigma$ generated by $\Upphi$. It inherits a choice of positive roots $\Upsigma_\Upphi^+ = \Upsigma_\Upphi \cap \Upsigma^+$.
    \item Next, define $\mk{a}_\Upphi = \bigcap_{\upalpha_i \in \Upphi} \Ker(\upalpha_i) = \bigoplus_{\upalpha_j \in \Uplambda \mysetminus \Upphi} \R H^j$ and $\mk{a}^\Upphi = \mk{a} \ominus \mk{a}_\Upphi = \bigoplus_{\upalpha_i \in \Upphi} \R H_{\upalpha_i}$.
    \item The subset $\Updelta_\Upphi = \Upsigma^+ \mysetminus \Upsigma_\Upphi^+$ can be split into parts: $\Updelta_\Upphi = \bigsqcup_{\upnu \ge 1} \Updelta_\Upphi^\upnu$, where $\Updelta_\Upphi^\upnu$ consists of the roots $\uplambda \in \Updelta_\Upphi$ with $\bilin{H^\Upphi}{\uplambda} = \upnu$, and $H^\Upphi = \sum_{\upalpha_j \in \Uplambda \mysetminus \Upphi} H^j$.
    \item Let $\mk{n}_\Upphi = \bigoplus_{\uplambda \in \Updelta_\Upphi} \mk{g}_\uplambda$. This nilpotent subalgebra is graded: $\mk{n}_\Upphi = \bigoplus_{\upnu \ge 1} \mk{n}_\Upphi^\upnu$, where $\mk{n}_\Upphi^\upnu = \bigoplus_{\uplambda \in \Updelta_\Upphi^\upnu} \mk{g}_\uplambda$ is the eigenspace of $\ad(H^\Upphi)$ in $\mk{n}_\Upphi$ with eigenvalue $\upnu$.
    \item Consider $\mk{l}_\Upphi = \mk{g}_0 \oplus \bigoplus_{\upalpha \in \Upsigma_\Upphi} \mk{g}_\upalpha$. This reductive subalgebra of $\mk{g}$ centralizes $\mk{a}_\Upphi$, normalizes $\mk{n}_\Upphi$, and preserves the grading on $\mk{n}_\Upphi$. In fact, one has $\mk{l}_\Upphi = Z_\mk{g}(\mk{a}_\Upphi)$.
    \item $\mk{q}_\Upphi = \mk{l}_\Upphi \loplus \mk{n}_\Upphi$ is a parabolic subalgebra of $\mk{g}$, and this semidirect sum is called a Chevalley decomposition. Every parabolic subalgebra of $\mk{g}$ is inner-conjugate to $\mk{q}_\Upphi$ for some choice of $\Upphi$. One can show that $\mk{n}_\Upphi$ is the nilradical of $\mk{q}_\Upphi$.
    \item Going further, let $\mk{m}_\Upphi = \mk{l}_\Upphi \ominus \mk{a}_\Upphi = \mk{k}_0 \oplus \mk{a}^\Upphi \oplus \bigoplus_{\upalpha \in \Upsigma_\Upphi} \mk{g}_\upalpha \subseteq \mk{l}_\Upphi$. This is also a reductive subalgebra of $\mk{g}$, and we have a direct sum decomposition $\mk{l}_\Upphi = \mk{m}_\Upphi \oplus \mk{a}_\Upphi$. The induced splitting $\mk{q}_\Upphi = (\mk{m}_\Upphi \oplus \mk{a}_\Upphi) \loplus \mk{n}_\Upphi$ is called a Langlands decomposition. We will normally omit the parentheses and simply write $\mk{m}_\Upphi \oplus \mk{a}_\Upphi \loplus \mk{n}_\Upphi$.
    \item The $\mk{k}$- and $\mk{p}$-parts of $\mk{m}_\Upphi$ are $\mk{k}_\Upphi = \mk{k}_0 \oplus \bigoplus_{\upalpha \in \Upsigma_\Upphi^+} \mk{k}_\upalpha$ and $\mk{b}_\Upphi = \mk{a}^\Upphi \oplus \bigoplus_{\upalpha \in \Upsigma_\Upphi^+} \mk{p}_\upalpha$, respectively.
    \item Being reductive, $\mk{m}_\Upphi$ is the direct sum of its center $\mk{z}_\Upphi \subseteq \mk{k}_0$ and its semisimple part $\mk{g}_\Upphi = [\mk{m}_\Upphi, \mk{m}_\Upphi] = [\mk{l}_\Upphi,\mk{l}_\Upphi]$.
    \item $\mk{g}_\Upphi$ splits as the direct sum of its compact and noncompact semisimple parts. The latter can be described as $\mk{g}'_\Upphi = [\mk{b}_\Upphi, \mk{b}_\Upphi] \oplus \mk{b}_\Upphi$, or alternatively, as the subalgebra of $\mk{g}$ generated by $\mk{g}_\upalpha, \, \upalpha \in \Upsigma_\Upphi$. The compact semisimple part of $\mk{g}_\Upphi$ is given by $Z_{\mk{k}_0}(\mk{g}'_\Upphi) \ominus \mk{z}_\Upphi$. The sum of the center of $\mk{m}_\Upphi$ and its compact semisimple part can thus be described as $Z_{\mk{k}_0}(\mk{g}'_\Upphi) = Z_{\mk{k}_0}(\mk{b}_\Upphi)$.
    \item $\mk{g}'_\Upphi$ comes with an Iwasawa decomposition: $\mk{g}'_\Upphi = \mk{k}^\Upphi \oplus \mk{a}^\Upphi \oplus \mk{n}^\Upphi$, where $\mk{k}^\Upphi = \mk{g}'_\Upphi \cap \mk{k}_\Upphi = \mk{g}'_\Upphi \cap \mk{k}$ and $\mk{n}^\Upphi = \mk{g}'_\Upphi \cap \mk{n} = \bigoplus_{\upalpha \in \Upsigma_\Upphi^+} \mk{g}_\upalpha$.
    \item We denote the $\mk{k}_0$-part of $\mk{g}'_\Upphi$ by $\mk{k}^\Upphi_0$. Note that \label{k_0_decomposition}
    $$
    \mk{k}_0 = Z_{\mk{k}_0}(\mk{g}'_\Upphi) \oplus \mk{k}^\Upphi_0 = \mk{z}_\Upphi \oplus (Z_{\mk{k}_0}(\mk{g}'_\Upphi) \ominus \mk{z}_\Upphi) \oplus \mk{k}^\Upphi_0.
    $$
    This is a direct sum of Lie algebras, and the summands are pairwise orthogonal with respect to $B_\uptheta$ (or $B$).
    \item The connected (closed) Lie subgroups of $G$ corresponding to $\mk{g}'_\Upphi, \mk{a}_\Upphi,$ and $\mk{n}_\Upphi$ are denoted by $G'_\Upphi, A_\Upphi,$ and $N_\Upphi$, respectively. Furthermore, we define:
    $$
    \hspace{6.5ex} L_\Upphi = Z_G(\mk{a}_\Upphi), \quad Q_\Upphi = L_\Upphi \ltimes N_\Upphi = N_G(\mk{q}_\Upphi), \quad K_\Upphi = K \cap L_\Upphi, \quad M_\Upphi = K_\Upphi G'_\Upphi.
    $$
    These are closed subgroups of $G$ with Lie algebras $\mk{l}_\Upphi, \mk{q}_\Upphi, \mk{k}_\Upphi,$ and $\mk{m}_\Upphi$, respectively. We have a direct product decomposition $L_\Upphi = M_\Upphi \times A_\Upphi$. The decompositions $Q_\Upphi = L_\Upphi \ltimes N_\Upphi$ and $Q_\Upphi = (M_\Upphi \times A_\Upphi) \ltimes N_\Upphi \defqe M_\Upphi \times A_\Upphi \ltimes N_\Upphi$ are also called Chevalley and Langlands decompositions of the parabolic subgroup $Q_\Upphi$, respectively. The adjoint representation of $L_\Upphi$ on $\mk{n}_\Upphi$ respects its grading. If we write $K^\Upphi, A^\Upphi,$ and $N^\Upphi$ for the connected Lie subgroups of $G'_\Upphi$ with Lie algebras $\mk{k}^\Upphi, \mk{a}^\Upphi,$ and $\mk{n}^\Upphi$, respectively, then we have a global Iwasawa decomposition $G'_\Upphi = K^\Upphi A^\Upphi N^\Upphi$.
    \item The orbit $B_\Upphi = G'_\Upphi \ccdot o = M_\Upphi \ccdot o$ is a \textit{boundary component} of $M$; it is a totally geodesic submanifold (corresponding to the Lie triple system $\mk{b}_\Upphi$) and itself a symmetric space of noncompact type and rank $r_\Upphi = |\Upphi|$. We can think of its Dynkin diagram $\DD_\Upphi$ as the subdiagram of $\DD$ spanned by the nodes corresponding to the roots in $\Upphi$. This space is represented by the almost effective Riemannian symmetric pair $(G'_\Upphi, K^\Upphi)$, as well as by the Riemannian symmetric pair $(M^0_\Upphi, K^0_\Upphi)$. Note that the latter has ineffectiveness kernel $Z_{\mk{k}_0}(\mk{g}'_\Upphi)$, which is compact. 
    \item The Langlands decomposition of $Q_\Upphi$ induces---via the transitive action $Q_\Upphi \curvearrowright M$---a horospherical decomposition $B_\Upphi \times A_\Upphi \times N_\Upphi \cong M$ (just a diffeomorphism). We can summarize this with the following diagram:
    \begin{equation*}
    \begin{tikzcd}
    M_\Upphi \times A_\Upphi \ltimes N_\Upphi \ar[r, "\sim" yshift = -0.2ex] \ar[d, twoheadrightarrow] & Q_\Upphi \ar[d, twoheadrightarrow] \ar[dr, twoheadrightarrow] \\
    B_\Upphi \times A_\Upphi \times N_\Upphi \ar[r, "\sim" yshift = -0.2ex] & Q_\Upphi/K_\Upphi \ar[r, "\sim" yshift = -0.2ex] & M
    \end{tikzcd}
    \end{equation*}
    With respect to the Langlands and horospherical decompositions, the action of $Q_\Upphi$ on $M$ can be written as:
    \begin{gather}\label{horospherical_decomposition_action}
    M_\Upphi \times A_\Upphi \ltimes N_\Upphi \curvearrowright B_\Upphi \times A_\Upphi \times N_\Upphi, \nonumber \\
    (m,a,n) \cdot (m' \cdot o,a',n') = ((mm') \cdot o,aa',(m'a')^{-1} n (m'a') n').
    \end{gather}
    \item The submanifold $F_\Upphi = L_\Upphi \cdot o \cong B_\Upphi \times A_\Upphi \subseteq M$ is also totally geodesic and is a symmetric space of rank $r$ with an $(r-r_\Upphi)$-dimensional flat factor; its splitting into the product of $B_\Upphi$ and $A_\Upphi \cong A_\Upphi \cdot o$ is a Riemannian product decomposition. For lack of a better word, we call $F_\Upphi$ an \textit{extended boundary component}. It corresponds to the Lie triple system $\mk{f}_\Upphi = \mk{b}_\Upphi \oplus \mk{a}_\Upphi$.
\end{itemize}

We will be mostly interested in the situations where $\Upphi$ is of the form $\Uplambda \mysetminus \set{\upalpha_j}$ for some simple root $\upalpha_j$. In that case, in all of the objects defined above, we simply replace the subscript or superscript $\Upphi$ with $j$; we also write $\Uplambda_j = \Uplambda \mysetminus \set{\upalpha_j}$. We then have $H^\Upphi = H^j$, and for each $\upalpha \in \Upsigma^+$, $\bilin{\upalpha}{H^j}$ gives the coefficient corresponding to $\upalpha_j$ in the expression of $\upalpha$ with respect to the simple system $\Uplambda$. Hence, the set $\Updelta_j^1$ can be described as
\begin{equation}\label{equation:levelone:string}
\Updelta_j^1 = \Bigl\{ \upalpha_j + \sum_{i \ne j} n_i \upalpha_i \in \Upsigma^+ \mid n_i \in \Z_{\ge 0} \Bigr\}.   
\end{equation}
In Section \ref{sec:solving_the_NC} we will need to determine explicitly the set $\Updelta_j^1$ for some particular choices of $\upalpha_j$. In order to do so, we will make use of the following generalization of the classical notion of a root string. Let $\Upphi \subseteq \Uplambda$ once again be an arbitrary subset of simple roots, and let $\uplambda \in \Upsigma$ be any root. We define the \emph{$\Upphi$-string of $\uplambda$} as the set of all the elements of $\Upsigma_0$ of the form $\uplambda + \sum_{\upalpha_i \in \Upphi} n_i \upalpha_i$ with $n_i \in \mathbb{Z}$. This notion of string was introduced by the first author in~\cite{sanmartin-strings}, where he also calculated all possible $\Upphi$-strings for all choices of $\Upsigma, \Upphi,$ and $\uplambda$. The base idea of that computation is simple: if $\uplambda$ is the root of minimum height in its $\Upphi$-string, then $\Upphi \cup \{ \uplambda \}$ is a simple system for the root subsystem $\Upsigma_{\Upphi, \uplambda} = \vspan_\Z(\Upphi \cup \{ \uplambda \}) \cap \Upsigma$ of $\Upsigma$. Note that for such $\uplambda$, every root $\uplambda + \sum_{\upalpha_i \in \Upphi} n_i \upalpha_i$ in the $\Upphi$-string has all $n_i$ nonnegative (see \cite[Prop.\hspace{2pt}2.49]{knapp}). The argument is then based on calculating explicitly the $\Upphi$-string of $\uplambda$ for any possible pair~$(\Upsigma, \Upsigma_{\Upphi, \uplambda})$. Under this new perspective, $\Updelta_j^1$ is nothing but the $\Uplambda_j$-string of $\upalpha_j$, as one can see from~\eqref{equation:levelone:string} and the fact that $\upalpha_j$ has minimum height in its $\Uplambda_j$-string. Hence, for any root system $\Upsigma$ and any choice of $j$, we have an explicit description of $\Updelta_j^1$. We will make great use of that in Section \ref{sec:solving_the_NC}.

\subsection{The extrinsic geometry of singular orbits}\label{sec:preliminaries:extrinsic_geometry_cpc}

As we have explained, any symmetric space $M$ of noncompact type can be regarded as a solvable Lie group $AN$ endowed with a certain left-invariant metric. One of the key underpinnings of this work is the fact that the singular orbit of a nilpotent construction action\footnote{Recall from Theorem \ref{thm:C} that this is in fact true for any C1-action. We will show this in Section \ref{sec:t.g._and_singular_orbit}.} on $M$ can always be realized as a Lie subgroup of $AN$. This property singles out a very special class of submanifolds of $M$, whose geometry is tied particularly closely to the structure of the isometry Lie algebra $\mk{g}$. In the following lines, we introduce a few tools that will help us study the extrinsic geometry of such submanifolds.

Let $M = G/K$ be a symmetric space of noncompact type. In this subsection, we assume that the metric on $M$ is Killing (see Remark \ref{rem:choice_of_metric}). As usual, we fix a point $o \in M$, a maximal abelian subspace $\mk{a} \subset \mk{p}$, and a set of positive roots $\Upsigma^+ \subset \Upsigma$. These choices give rise to an Iwasawa decomposition $G = KAN$ and an identification $AN \cong M$. The isomorphism $\mk{a} \oplus \mk{n} \isoto T_oM \cong \mk{p}$ is given simply by the projection onto $\mk{p}$ along $\mk{k}$. As follows from the proof of~\cite[Prop.\hspace{2pt}4.4]{tamaru_parabolic_Einstein}, the induced left-invariant metric on $AN$ is related to the inner product on $\mk{g}$ via the equation
\begin{equation}\label{equation:inner:relation:b:theta:an}
\cross{X}{Y}_{AN} = \cross{X_\mk{a}}{Y_\mk{a}}_{B_\uptheta} +\frac{1}{2} \cross{X_\mk{n}}{Y_\mk{n}}_{B_\uptheta} \quad \text{for any $X$, $Y \in \mk{a} \oplus \mk{n}$},
\end{equation}
where $\mk{a}$ and $\mk{n}$ in the subscript denote the $\mk{a}$- and $\mk{n}$-components of a vector, respectively. Since $\mk{a}$ and $\mk{n}$ are mutually orthogonal with respect to both inner products, we see from \eqref{equation:inner:relation:b:theta:an} that $\cross{X}{Y}_{AN} = \cross{X}{Y}_{B_\uptheta}$ whenever $X$ or $Y$ lies in $\mk{a}$; similarly, $\cross{X}{Y}_{AN} = \frac{1}{2}\cross{X}{Y}_{B_\uptheta}$ if $X$ or $Y$ lies in $\mk{n}$. The Levi-Civita connection on $AN$ also admits a neat expression. First, we need the following simple equality:
\begin{equation}\label{equation:inner:b:theta}
\cross{\ad(X) Y}{Z}_{B_\uptheta} = - \cross{Y}{\ad(\uptheta X) Z}_{B_\uptheta} \quad \text{for any $X$, $Y$, $Z \in \mk{g}$}.
\end{equation}
Combining this with the Koszul formula and \eqref{equation:inner:relation:b:theta:an} leads to:
\begin{equation}\label{equation:levi:civita:connection:an}
\cross{\nabla_X Y}{Z}_{AN} = \frac{1}{4} \cross{[X,Y] + [\uptheta X, Y] - [X, \uptheta Y]}{Z}_{B_\uptheta} \quad \text{for any $X$, $Y$, $Z \in \mk{a} \oplus \mk{n}$}.
\end{equation}
Now, let $H$ be a connected Lie subgroup of $AN$ with Lie algebra $\mk{h}$. Under the identification between $M$ and $AN$, the orbit $H \cdot o$ corresponds to $H$, the tangent space $T_o(H \cdot o)$ becomes $\mk{h}$, and the normal space $N_o(H \cdot o)$ becomes $\mk{v} = (\mk{a} \oplus \mk{n}) \ominus \mk{h}$, where the orthogonal complement is taken with respect to $\cross{-}{-}_{AN}$. We would like to have a handy description of the shape operators of $H \cdot o$. Since this is a homogeneous submanifold, it suffices to examine its shape operators only at $o$.
\begin{lem}\label{lem:AN_shape_operators_general}
    Given $\upxi \in \mk{v}$, the shape operator $\mc{A}_\upxi$ of $H \cdot o$ at $o$ is given by the following formula:
    \begin{equation}\label{AN_shape_operators_general}
        \cross{\mc{A}_\upxi X}{Y}_{AN} = \frac{1}{4}\cross{[\upxi,X] - [\uptheta\upxi,X]}{Y}_{B_\uptheta} \quad \text{for any $X, Y \in \mk{h}$}.
    \end{equation} 
\end{lem}

\begin{proof}
    By definition, $\mc{A}_\upxi X = -(\nabla_X \upxi)^\top$, where $\top$ denotes the orthogonal projection to $\mk{h}$ along $\mk{v}$. Given $Y \in \mk{h}$, we use \eqref{equation:levi:civita:connection:an} to compute:
    $$
    \cross{\mc{A}_\upxi X}{Y}_{AN} = -\cross{(\nabla_X \upxi)^\top}{Y}_{AN} = -\cross{\nabla_X \upxi}{Y}_{AN} = \frac{1}{4} \cross{[\upxi,X] - [\uptheta X, \upxi] - [\uptheta \upxi,X]}{Y}_{B_\uptheta}.
    $$
    By comparing this to \eqref{AN_shape_operators_general}, it only remains to show that $\cross{[\uptheta X, \upxi]}{Y}_{B_\uptheta} = 0$. We calculate:
    $$
    \cross{[\uptheta X, \upxi]}{Y}_{B_\uptheta} = -\cross{\upxi}{[X,Y]}_{B_\uptheta} = -\cross{\upxi}{[X,Y]}_{AN} = 0.
    $$
    Here the first equality follows from \eqref{equation:inner:b:theta}, the second from the fact that $[X,Y] \in \mk{n}$, and the last one from $[X,Y] \in \mk{h}$. This completes the proof.
\end{proof}
Formula \eqref{AN_shape_operators_general} can be simplified still further provided we know some additional information about $\mk{h}$: for example, if $\mk{h} = (\mk{h} \cap \mk{a}) \loplus (\mk{h} \cap \mk{n})$ (which is often the case, see Theorem \ref{thm:C}) or if $\mk{h}$ contains nonzero vectors lying in a single root space. We will see a particular instance of that in the context of the nilpotent construction in Section \ref{sec:solving_the_NC} (see Proposition \ref{proposition:shape:orbit:an}). As a preparation for that, we need the following two results, which provide an insight into how different root spaces behave with respect to the Lie bracket.

\begin{lem}\label{lemma:a:k0} 
 Let $\upalpha \in \Sigma^{+}$, and let $X,Y \in \mk{g}_{\upalpha}$ be orthogonal. Then:
\begin{enumerate}[\normalfont (a)]
\item\customlabel{lemma:a:k0:1}{a} $[\uptheta X, X] = 2 ||X||^2_{AN} H_{\upalpha} = ||X||^2_{B_{\uptheta}} H_{\upalpha}$.
\item\customlabel{lemma:a:k0:2}{b} $[\uptheta X, Y] \in \mk{k}_0$.
\end{enumerate} 
\end{lem}

For a proof, see \cite[Lem.\hspace{2pt}2.3]{berndt_SL_CPC}. Note that we do not need to specify whether $X$ and $Y$ are orthogonal with respect to $\cross{-}{-}_{AN}$ or $B_\uptheta$, as those are proportional on $\mk{g}_\upalpha$.

\begin{lem}\label{lem:string_injective}
    Let $\upalpha,\upbeta \in \Upsigma$ be non-proportional roots and suppose that the $\upbeta$-string of $\upalpha$ has the form $\set{\upalpha, \upalpha + \upbeta, \ldots, \upalpha + m\upbeta}$ for some positive integer $m$. Then for any nonzero $X \in \mk{g}_\upbeta$ and any $1 \le k \le m$, the operator $\ad(X)^k \colon \mk{g}_\upalpha \to \mk{g}_{\upalpha + k\upbeta}$ is injective.
\end{lem}

\begin{proof}
    By choosing a different system of simple roots if necessary, we may assume without loss of generality that both $\upalpha$ and $\upbeta$ are positive. It follows from \cite[Prop.\hspace{2pt}2.48]{knapp} that $m \in \set{1,2,3}$. The case of $k=1$ (and any $m$) was done in \cite[Lem.\hspace{2pt}2.4(i)]{berndt_SL_CPC}, and the case $k = m = 2$ was shown in Proposition 4.2(iv) of the same paper. Therefore, we only need to deal with\footnote{This can only occur when the irreducible component of $\Upsigma$ containing $\upalpha$ and $\upbeta$ is isomorphic to $\mm{G}_2$.} $m = 3, \, k = 2,3$. Take a nonzero vector $Y \in \mk{g}_\upalpha$. By combining \cite[Lem.\hspace{2pt}2.4]{berndt_SL_CPC} with \cite[Prop.\hspace{2pt}2.48(g)]{knapp}, we obtain that $\ad(\uptheta X) \circ \ad(X)^3(Y)$ is a nonzero multiple of $\ad(X)^2(Y)$, and $\ad(\uptheta X) \circ \ad(X)^2(Y)$ is a nonzero multiple of $\ad(X)(Y)$. But we already know that this latter vector is nonzero. Consequently, $\ad(X)^2(Y)$ must be nonzero, and hence so does $\ad(X)^3(Y)$. This concludes the proof.
\end{proof}

In Section \ref{sec:solving_the_NC}, we will constantly use another important property of singular orbits of C1-actions: their principal curvatures, counted with multiplicities, do not depend on the normal unit direction. More rigorously, given such an orbit $S \subset M$ and two normal unit vectors $\upxi, \upnu \in NS$ (possibly at different points of $S$), the spectra of the shape operators $\mc{A}_\upxi$ and $\mc{A}_\upnu$ coincide. This is a simple consequence of the facts that $S$ is a homogeneous submanifold and that the slice representation of the underlying C1-action at any point of $S$ is transitive on the unit sphere in the normal space to $S$ at that point. Submanifolds with this noteworthy geometric property were investigated and called \emph{CPC submanifolds} in~\cite{berndt_SL_CPC}. In fact, one of the goals of that article was to show that singular orbits of C1-actions are far from being the only class of submanifolds that enjoy the CPC property. 

\section{Cohomogeneity-one actions}\label{sec:c1-actions}
In this section, we discuss what is known about cohomogeneity-one actions on symmetric spaces of noncompact type and review the current progress in their classification. We begin by describing the four known types and two methods of construction of such actions.

\subsection{Types of cohomogeneity-one actions}\label{sec:c1-actions:types}

Let $M = G/K$ be a symmetric space of noncompact type, and let $H$ be a connected Lie group acting on $M$ properly, isometrically, and with cohomogeneity one. As was shown in \cite{berndt_bruck}, there can only be two scenarios:

\begin{enumerate}[(a)]
    \item\customlabel{topology_of_C1_actions:a}{a} $H$ has no singular orbits in $M$ and its orbits form a Riemannian foliation by hypersurfaces. We call that a homogeneous codimension-one foliation. In that case, every orbit is principal.
    \item\customlabel{topology_of_C1_actions:b}{b} $H$ has precisely one singular orbit $S$, all the other orbits are principal, and they can be described as the equidistant tubes of various radii around $S$. In that case, $S$ is diffeomorphic to $\R^k$ for some $k$, while each principal orbit is diffeomorphic to $\R^k \times \S^{n-k-1}$, where $n = \dim(M)$.
\end{enumerate}

Two isometric actions $H_1 \curvearrowright M$ and $H_2 \curvearrowright M$ are called \textit{orbit-equivalent} if there exists an isometry $\upvarphi \in I(M)$ that identifies their orbits: for every $p \in M, \, \upvarphi(H_1 \cdot p) = H_2 \cdot \upvarphi(p)$. Similarly, two submanifolds $S_1, S_2 \subseteq M$ are called \textit{congruent} if there exists $\uppsi \in I(M)$ such that $\uppsi(S_1) = S_2$. We also need the notion of \textit{strong} orbit equivalence (resp., \textit{strong} congruence): this is when the above isometry can be chosen in $I^0(M)$. The problem of distinguishing between orbit-equivalent actions or congruent submanifolds is generally referred to as \textit{the problem of congruence}.

From now on, we fix a symmetric space $M = G/K$ of noncompact type, a base point $o \in M$, a maximal abelian subspace $\mk{a} \subset \mk{p}$, and a choice of positive roots $\Upsigma^+ \subset \Upsigma$.

Let $\ell \subseteq \mk{a}$ be a one-dimensional subspace. The connected Lie subgroup $H_\ell$ of $G$ with Lie algebra $\mk{h}_\ell = (\mk{a} \ominus \ell) \oplus \mk{n}$ is closed, and its action on $M$ has cohomogeneity one and no singular orbits. Moreover, all of its orbits are isometrically congruent to each other. The orbit foliation of $H_\ell$ is denoted by $\mc{F}_\ell$ and is said to be of \textit{horospherical type}. This name was coined in \cite{DR_DV_Otero_C1} and comes from the fact that for some choices of $\ell$, the leaves of $\mc{F}_\ell$ are horospheres. Consider the group $\Aut^\mathrm{w}(\DD)$ of symmetries of the Dynkin diagram that preserve the simple root multiplicities. It permutes the basis $\Uplambda$ of $\mk{a}^*$ and thus acts on $\mk{a}^*$ and $\mk{a}$. Provided the metric on $M$ is almost Killing, two such actions by $H_\ell$ and $H_{\ell'}$ are orbit-equivalent if and only if the lines $\ell, \ell' \subseteq \mk{a}$ differ by an element of $\Aut^\mathrm{w}(\DD)$. For more details as well as a more general statement on congruence, see \cite{berndt_tamaru_foliations,solonenko_foliations}.

Let $\upalpha_i \in \Uplambda$ be any simple root and $\ell_{\upalpha_i} \subseteq \mk{g}_{\upalpha_i}$ a one-dimensional subspace. The connected Lie subgroup $H_{\upalpha_i}$ of $G$ with Lie algebra $\mk{h}_{\upalpha_i} = \mk{a} \oplus (\mk{n} \ominus \ell_{\upalpha_i})$ is closed and its action on $M$ has cohomogeneity one and no singular orbits. The resulting foliation is denoted by $\mc{F}_{\upalpha_i}$ and is said to be of \textit{solvable type} (also coined in \cite{DR_DV_Otero_C1}). This foliation has a unique minimal leaf---the one passing through $o$. Moreover, for every $t>0$, the two leaves at distance $t$ from the minimal one are congruent to each other and to no other leaf. Were we to choose another line $\ell'_{\upalpha_i}$ in $\mk{g}_{\upalpha_i}$, the resulting action would be strongly orbit-equivalent to that for $\ell_{\upalpha_i}$. When the metric on $M$ is almost Killing, two such actions by $H_{\upalpha_i}$ and $H_{\upalpha_j}$ are orbit-equivalent if and only if the simple roots $\upalpha_i$ and $\upalpha_j$ differ by an element of $\Aut^\mathrm{w}(\DD)$. Once again, see \cite{berndt_tamaru_foliations,solonenko_foliations} for details and a general version of the congruence criterion.

The next two types of C1-actions are those with a totally geodesic orbit. The only symmetric space of noncompact type that admits a totally geodesic hypersurface is $\R H^n$, and that hypersurface (the hyperbolic hyperplane $\R H^{n-1} \subset \R H^n$) is a leaf of the (unique up to congruence) foliation of solvable type on $\R H^n$. This means that we need only consider C1-actions with a totally geodesic \textit{singular} orbit. According to \eqref{topology_of_C1_actions:b} above, such an action is fully determined by its singular orbit up to orbit equivalence. In \cite{berndt_tamaru_totally_geodesic_singular_orbit}, Berndt and Tamaru classified such actions for $M$ irreducible and discovered that they are intimately related to reflective submanifolds. A submanifold of a Riemannian  manifold is called \textit{reflective} if it can be realized as a connected component of the set of fixed points of an involutive isometry of the ambient space. Such submanifolds are automatically totally geodesic. In \cite{leung_classification_of_reflective_submanifolds,leung_congruence}, Leung classified reflective submanifolds in all irreducible symmetric spaces of compact type up to congruence. Later on, it was shown in \cite{berndt_tamaru_totally_geodesic_singular_orbit} that if a C1-action on $M$ has a totally geodesic singular orbit, that orbit has to be reflective---apart from 5 exceptions, all mysteriously related to the group $\mm{G}_2$ (see Theorem 4.2 therein). Using Leung's classification and the duality between symmetric spaces of compact and noncompact type, Berndt and Tamaru calculated exactly which reflective submanifolds arise in this way. In \cite{leung_congruence}, Leung showed that if two reflective submanifolds in a compact irreducible symmetric space are isometric, then they are congruent, albeit not necessarily strongly congruent. That result also carries over to the noncompact type via the duality. But if the singular orbits of two C1-actions are congruent, then the actions themselves are orbit-equivalent. 
So a C1-action with a totally geodesic singular orbit is determined up to orbit equivalence by the isometry class of that orbit. For the complete list of these orbits, see \cite[Th.\hspace{2pt}3.1, 3.3]{berndt_tamaru_totally_geodesic_singular_orbit}.

The second type of C1-actions with a totally geodesic singular orbit was studied in \cite{DR_DV_Otero_C1} and occurs on reducible spaces---but of very specific form. Suppose $M = M_1 \times M_2$ is the product of two homothetic rank-one spaces; this means that $M_i$ is a hyperbolic space over $\R, \C,$ or $\H$, or the Cayley hyperbolic plane $\Oo H^2$, but the radii of $M_1$ and $M_2$ can be different. We can write $\mk{g} = \mk{g}_1 \oplus \mk{g}_2$, where $\mk{g}_i$ is the isometry Lie algebra of $M_i$. Fix an isomorphism $\upvarphi \colon \mk{g}_1 \isoto \mk{g}_2$  and consider the diagonal subalgebra $\mk{g}_\upvarphi = \set{X + \upvarphi(X) \mid X \in \mk{g}_1}$. The corresponding connected Lie subgroup $G_\upvarphi$ of $G = G_1 \times G_2$ (here we take $G = I^0(M)$ and $G_i = I^0(M_i)$) acts on $M$ with cohomogeneity one and has a totally geodesic (even reflective) singular orbit homothetic to $M_1$ and $M_2$. We call this a \textit{diagonal C1-action} on $M$. Choosing another isomorphism $\uppsi \colon \mk{g}_1 \isoto \mk{g}_2$ results in a strongly orbit-equivalent action provided the isomorphism $\uppsi^{-1} \circ \upvarphi$ of $\mk{g}_1$ is inner. The group $\Aut(\mk{g}_1)$ has two connected components if $M_i \simeq \R H^n$ or $\C H^n$ and only one if $M_i \simeq \H H^n$ or $\Oo H^2$ (see \cite{gundogan}). Consequently, this construction yields one strong-orbit-equivalence class of C1-actions if $M_i \simeq \H H^n$ or $\Oo H^2$ and at most two if $M_i \simeq \R H^n$ or $\C H^n$.

Now, we describe a general construction, introduced in \cite{berndt_tamaru_cohomogeneity_one} and called the \textit{canonical extension}, that allows to construct new C1-actions on higher-rank symmetric spaces from those on their boundary components. We come back to the general case of an arbitrary $M = G/K$ and fix a choice of $o, \, \mk{a},$ and $\Upsigma^+$. Pick a subset $\Upphi \subset \Uplambda$. Suppose $H_\Upphi$ is a connected Lie group acting properly and isometrically on the boundary component $B_\Upphi$. Factoring out the ineffectiveness kernel, we may assume it is a subgroup of $I^0(B_\Upphi)$ and then lift it along $G'_\Upphi \twoheadrightarrow I^0(B_\Upphi)$ to a connected Lie subgroup of $G'_\Upphi \subset M_\Upphi$ (still denoted by the same symbol). Consider the subgroup 
$$
H_\Upphi^\Uplambda = H_\Upphi \times A_\Upphi \ltimes N_\Upphi \subset M_\Upphi \times A_\Upphi \ltimes N_\Upphi = Q_\Upphi.
$$
This group acts on $M$ properly and with the same cohomogeneity as that of $H_\Upphi \curvearrowright B_\Upphi$. Moreover, with respect to the horospherical decomposition, its orbits can be described as $(H_\Upphi \cdot p) \times A_\Upphi \times N_\Upphi$, where $p \in B_\Upphi$. In particular, if $H_\Upphi \cdot p$ is a singular orbit of $H_\Upphi$ in $B_\Upphi$, then $H_\Upphi^\Uplambda \cdot p$ is that of $H_\Upphi^\Uplambda$ in $M$. The Lie algebra of $H_\Upphi^\Uplambda$ can be described as $\mk{h}_\Upphi^\Uplambda = \mk{h}_\Upphi \oplus \mk{a}_\Upphi \loplus \mk{n}_\Upphi$, where $\mk{h}_\Upphi = \Lie(H_\Upphi) \subseteq \mk{g}'_\Upphi$. The action $H_\Upphi^\Uplambda \curvearrowright M$ (as well as the group $H_\Upphi^\Uplambda$ and the Lie algebra $\mk{h}_\Upphi^\Uplambda$) is called the \textit{canonical extension} of $H_\Upphi \curvearrowright B_\Upphi$ (resp., of $H_\Upphi$ and $\mk{h}_\Upphi$). If two actions on $B_\Upphi$ are \textit{strongly} orbit-equivalent, then their canonical extensions are (strongly) orbit-equivalent. However, non-strongly orbit-equivalent actions on $B_\Upphi$ can lead to orbit-non-equivalent actions on $M$ (see the example on p.\ 139 in \cite{berndt_tamaru_cohomogeneity_one}). This subtlety is important when one wants to study the problem of congruence for C1-actions (see \cite[Rem.\hspace{2pt}1.1]{DR_DV_Otero_C1}). It is also worth mentioning that the canonical extension method is transitive. Given two subsets $\Upphi \subset \Uppsi \subset \Uplambda$, we can regard $B_\Upphi$ as a boundary component of $B_\Uppsi$. In particular, given an isometric action on $B_\Upphi$, we could first canonically extend it to $B_\Uppsi$, and then extend the result to $M$, but that would be orbit-equivalent to the canonical extension of the original action directly to $M$. In other words, given a Lie subalgebra $\mk{h}_\Upphi \subseteq \mk{g}'_\Upphi$, we have $(\mk{h}_\Upphi^\Uppsi)^\Uplambda = \mk{h}_\Upphi^\Uplambda$ (see \cite[Lem.\hspace{2pt}4.2]{DR_DV_Otero_C1} for a proof). For more details on the canonical extension method, see \cite{berndt_tamaru_cohomogeneity_one}.

Finally, we have reached the central point of this paper: the \textit{nilpotent construction}. In its current form, this method was also invented in \cite{berndt_tamaru_cohomogeneity_one}, although its simpler version for rank-one spaces goes all the way back to \cite{berndt_bruck}. One more time, fix a subset $\Upphi \subset \Uplambda$. Let $\mk{w} \subseteq \mk{n}_\Upphi^1$ be any subspace. We introduce the following notation:
$$
\mk{n}_{\Upphi,\mk{w}} = \mk{w} \oplus \bigoplus_{\upnu \ge 2} \mk{n}_\Upphi^\upnu, \quad \mk{h}_{\Upphi,\mk{w}} = N_{\mk{l}_\Upphi}(\mk{w}) \loplus \mk{n}_{\Upphi,\mk{w}} \subseteq \mk{l}_\Upphi \loplus \mk{n}_\Upphi = \mk{q}_\Upphi;
$$
and we write $N_{\Upphi,\mk{w}}$ and $H_{\Upphi,\mk{w}}$ for the corresponding connected Lie subgroups of $G$. These are both closed subgroups. Note that $H_{\Upphi,\mk{w}} = N^0_{L_\Upphi}(\mk{w}) \ltimes N_{\Upphi,\mk{w}} \subseteq L_\Upphi \ltimes N_\Upphi = Q_\Upphi$. The following conditions are equivalent:

\begin{enumerate}[(i)]
    \item $F_\Upphi \subseteq H_{\Upphi, \mk{w}} \ccdot o$.
    \item $N_{L_\Upphi}(\mk{w})$ acts transitively on $F_\Upphi$.
    \item The image of the projection of $N_{\mk{l}_\Upphi}(\mk{w})$ to $\mk{p}$ along $\mk{k}$ equals $\mk{b}_\Upphi \oplus \mk{a}_\Upphi$.
\end{enumerate}

In case $\Upphi = \Uplambda_j$ for some $\upalpha_j \in \Uplambda$, we can say a little more. First of all, since $\mk{a}_j = \R H^j$ and $\ad(H^j)$ acts on $\mk{n}_j^1$ as the identity, we have $N_{\mk{l}_j}(\mk{w}) = N_{\mk{m}_j}(\mk{w}) \oplus \mk{a}_j$ and $N_{L_j}(\mk{w}) = N_{M_j}(\mk{w}) \times A_j$. What is more, the above conditions are also equivalent to:

\begin{enumerate}
    \item[(iv)] $N_{M_\Upphi}(\mk{w})$ acts transitively on $B_\Upphi$. \label{admissibility_M}
    \item[(v)] The image of the projection of $N_{\mk{m}_\Upphi}(\mk{w})$ to $\mk{p}$ along $\mk{k}$ equals $\mk{b}_\Upphi$.
\end{enumerate}

Let us come back to the case of a general subset $\Upphi$. We will constantly make use of the following fact, which was stated in its infinitesimal version in \cite{berndt_tamaru_cohomogeneity_one}:

\begin{lem}\label{lem:NC_normalizers_vs_Uptheta}
    Let $\mk{w} \subseteq \mk{n}_\Upphi^1$ be any subspace, and consider $\mk{v} = \mk{n}_\Upphi^1 \ominus \mk{w}$. Then one has:

    \begin{enumerate}[\normalfont (a)]
        \item\customlabel{lem:NC_normalizers_vs_Uptheta:a}{a} $\Uptheta N_{L_\Upphi}(\mk{w}) = N_{L_\Upphi}(\mk{v})$ and $\uptheta N_{\mk{l}_\Upphi}(\mk{w}) = N_{\mk{l}_\Upphi}(\mk{v})$.
        \item\customlabel{lem:NC_normalizers_vs_Uptheta:b}{b} $\Uptheta N_{M_\Upphi}(\mk{w}) = N_{M_\Upphi}(\mk{v})$ and $\uptheta N_{\mk{m}_\Upphi}(\mk{w}) = N_{\mk{m}_\Upphi}(\mk{v})$.
        \item\customlabel{lem:NC_normalizers_vs_Uptheta:c}{c} $N_{K_\Upphi}(\mk{w}) = N_{K_\Upphi}(\mk{v})$ and $N_{\mk{k}_\Upphi}(\mk{w}) = N_{\mk{k}_\Upphi}(\mk{v})$.
    \end{enumerate}
\end{lem}

\begin{proof}
    First of all, we need only prove the first statement in each of the three parts, since the second one is its infinitesimal version. Part \eqref{lem:NC_normalizers_vs_Uptheta:c} will follow from \eqref{lem:NC_normalizers_vs_Uptheta:a} because $K_\Upphi = L_\Upphi \cap K$ is $\Uptheta$-stable. (Alternatively, it follows from the fact that the representation of $K_\Upphi$ on $\mk{n}_\Upphi^1$ is orthogonal.) The proof for \eqref{lem:NC_normalizers_vs_Uptheta:b} is essentially the same as for \eqref{lem:NC_normalizers_vs_Uptheta:a}, so we only deal with \eqref{lem:NC_normalizers_vs_Uptheta:a}. First, by considering the covering $\uppi \colon G \twoheadrightarrow G/Z(G) \cong I^0(M)$, it is easy to see that it suffices to prove the statement for $G/Z(G)$. In other words, we may assume that $G \cong I^0(M)$ and thus the Riemannian symmetric pair $(G,K)$ is effective. In this case, $\Uptheta$ can be described as the conjugation $C_{s_o}$, where $s_o \in I(M)$ is the geodesic symmetry at the base point $o \in M$. Given $g \in N_{L_\Upphi}(\mk{w})$, we need to show that $\Ad(\Uptheta(g)) = \Ad(s_o g s_o) = \uptheta \Ad(g) \uptheta$ preserves $\mk{v}$. Pick any $Y \in \mk{v}$ and $X \in \mk{w}$. We want to show that $\uptheta \hspace{-1pt} \Ad(g) \uptheta(Y)$ is still orthogonal to $X$. We compute:
    $$
    \cross{X}{\uptheta \Ad(g) \uptheta(Y)} = -B(X, \Ad(g) \uptheta(Y)) = -B(\Ad(g^{-1})X,\uptheta Y) = \cross{\Ad(g^{-1})X}{Y} = 0,
    $$
    since $\Ad(g^{-1})X \in \mk{w}$. We deduce that $\Uptheta N_{L_\Upphi}(\mk{w}) \subseteq N_{L_\Upphi}(\mk{v})$. Applying $\Uptheta$ to each side and switching the roles of $\mk{w}$ and $\mk{v}$ gives the reverse inclusion.
\end{proof}

We embrace the notation from the lemma and denote the orthogonal complement $\mk{n}_\Upphi^1 \ominus \mk{w}$ by $\mk{v}$. 

\begin{defn}\label{defn:NC_original}
    A subspace $\mk{w} \subseteq \mk{n}_\Upphi^1$ is called

    \begin{enumerate}[(a)]
    \item \textsc{admissible} if $N_{L_\Upphi}(\mk{w})$ acts transitively on $F_\Upphi$, and
    \item \textsc{protohomogeneous} if $N_{K_\Upphi}(\mk{w})$ acts transitively on the unit sphere in $\mk{v}$.
\end{enumerate}
\end{defn}

The main merit of this construction and these definitions is the following result of Berndt and Tamaru from \cite{berndt_tamaru_cohomogeneity_one}:

\begin{thm}\label{thm:berndt_tamaru_NC}
    Let $M = G/K$ be a symmetric space of noncompact type with a fixed choice of $o, \mk{a}, \Upsigma^+,$ and $\Upphi$.
    \begin{enumerate}[\normalfont (a)]
        \item\customlabel{thm:berndt_tamaru_NC:a}{a} Suppose $\mk{w} \subset \mk{n}_\Upphi^1$ is an admissible and protohomogeneous subspace of $\codim_{\mk{n}_\Upphi^1}(\mk{w}) \ge 2$. Then the subgroup $H_{\Upphi,\mk{w}}$ acts on $M$ with cohomogeneity one and has $F_\Upphi \times N_{\Upphi,\mk{w}} \subset F_\Upphi \times N_\Upphi \cong M$ as its orbit through $o$. That orbit has codimension equal to $\codim_{\mk{n}_\Upphi^1}(\mk{w})$ and is thus singular.
        \item\customlabel{thm:berndt_tamaru_NC:b}{b} Let $\mk{w} \subset \mk{n}_\Upphi^1$ and $k \in K_\Upphi$. Then $\mk{w}' = \Ad(k)(\mk{w})$ is admissible {\normalfont(}resp., protohomogeneous{\normalfont)} if and only if $\mk{w}$ is. Moreover, the actions of $H_{\Upphi,\mk{w}}$ and $H_{\Upphi,\mk{w}'}$ are orbit-equivalent by means of $k$.
    \end{enumerate}
\end{thm}

For any subspace $\mk{w}$ as in Theorem \ref{thm:berndt_tamaru_NC}\eqref{thm:berndt_tamaru_NC:a}, the action of $H_{\Upphi,\mk{w}}$ is called a \textit{nilpotent construction action}. Note that, under the projection $\mk{a} \oplus \mk{n} \isoto \mk{p} \cong T_oM$ along $\mk{k}$, $\mk{v}$ gets identified with the normal space to the orbit $H_{\Upphi,\mk{w}} \cdot o$ at $o$. For more details on the nilpotent construction, see \cite{berndt_tamaru_cohomogeneity_one,berndt_DV_C1_rank_2}.

\begin{rem}\label{rem:w_vs_v}
    A word of caution is in order. Our notation and definitions of admissibility and protohomogeneity are slightly (but inconsequentially) different from those used in other papers on the subject, such as \cite{berndt_tamaru_cohomogeneity_one,protohomogeneous,solonenko_hypersurfaces,DR_DV_Otero_C1}. In these papers, the perspective was not on $\mk{w}$ but on its orthogonal complement $\mk{v}$. More precisely, in a situation where a subspace $\mk{w} \subset \mk{n}_\Upphi^1$ is admissible or protohomogeneous according to Definition \ref{defn:NC_original}, the authors of the aforementioned papers say instead that $\mk{v}$ is admissible or protohomogeneous, respectively. (Note that the terms \textit{protohomogeneous} and \textit{admissible} were only coined in \cite{protohomogeneous} and \cite{solonenko_hypersurfaces}.) They also denote the subgroups $N_{\Upphi,\mk{w}}$ and $H_{\Upphi,\mk{w}}$ by $N_{\Upphi,\mk{v}}$ and $H_{\Upphi,\mk{v}}$. This is of course only a matter of notation and it has no implications for the nilpotent construction---but one should be careful not to confuse the two notational agreements. We will see in Subsection \ref{sec:admissibility:redundancy} how our perspective proves beneficial in certain circumstances.
\end{rem}

\subsection{The current state of classification}\label{sec:c1-actions:current_state}

Now we formulate a result that sums up what is currently known about C1-actions on symmetric spaces of noncompact type and their classification. In the present form, this result was established in \cite{DR_DV_Otero_C1}, although that paper is largely underpinned by \cite{berndt_tamaru_cohomogeneity_one}.

\begin{thm}\label{thm:classification_of_c1_actions}
Let $M = G/K$ be a symmetric space of noncompact type and rank $r$. Let $H$ be a connected Lie group acting on $M$ properly and isometrically. Then $H$ acts with cohomogeneity one if and only if its action is orbit-equivalent to one of the following:

\begin{enumerate}[\normalfont (a)]
        \item\customlabel{thm:classification_of_c1_actions:a}{a} The action of $H_\ell$ for some one-dimensional linear subspace $\ell \subseteq \mk{a}$.
        \item\customlabel{thm:classification_of_c1_actions:b}{b} The action of $H_{\upalpha_i}$ for some simple root $\upalpha_i \in \Uplambda$.
        \item\customlabel{thm:classification_of_c1_actions:c}{c} The canonical extension of a C1-action with a totally geodesic singular orbit on an irreducible boundary component $B_\Upphi$ of $M$.
        \item\customlabel{thm:classification_of_c1_actions:d}{d} The canonical extension of a diagonal C1-action on a reducible rank-2 boundary component $B_\Upphi$ of $M$ whose de Rham factors are homothetic.
        \item\customlabel{thm:classification_of_c1_actions:e}{e} The action of $H_{j, \mk{w}}$ for some $j \in \set{1, \ldots, r}$, where $\mk{w} \subset \mk{n}_j^1$ is an admissible and protohomogeneous subspace of $\codim_{\mk{n}_j^1}(\mk{w}) \ge 2$.
    \end{enumerate}
\end{thm}

From our discussion in Subsection \ref{sec:c1-actions:types}, it is evident that part \eqref{thm:classification_of_c1_actions:e} stands out among the others: it does not produce an explicit list of actions but rather reduces the classification to a problem in representation theory. We also see that the nilpotent construction is only relevant when $\Upphi = \Uplambda_j$ for some $j$. For this reason, for the rest of the article, whenever we say \textit{nilpotent construction}, we tacitly assume that $\Upphi$ is of this form, unless otherwise specified. Note that actions of type \eqref{thm:classification_of_c1_actions:e} in Theorem \ref{thm:classification_of_c1_actions} are decomposable (see \cite[Th.\hspace{2pt}C]{DR_DV_Otero_C1}). Consequently, when dealing with the nilpotent construction problem, we may always assume that $M$ is irreducible.

\begin{rem}[The choice of a metric]\label{rem:choice_of_metric}
    Suppose $M = G/K$ is a symmetric space of noncompact type and $H$ is a connected Lie group acting on $M$ isometrically. Rescaling the existing Riemannian metric along the de Rham factors of $M$ may change the full isometry group $I(M)$ but not its identity component $I^0(M)$ (see \cite[Prop.\hspace{2pt}2.1.60]{solonenko_thesis}). This implies that after any such rescaling, the action $H \curvearrowright M$ would still be isometric. Therefore, when classifying C1-actions---and in particular dealing with the nilpotent construction---we may assume that $M$ is equipped with the Killing metric $g_B$. (See also \cite[Rem.\hspace{2pt}2.5]{DR_DV_Otero_C1}.) This simplifies certain formulas and is particularly useful when one wants to work with the solvable model of $M$. In this article, we adopt this assumption in Subsection \ref{sec:preliminaries:extrinsic_geometry_cpc} and Section \ref{sec:solving_the_NC}. Note that one cannot, however, make this assumption in general when dealing with the problem of congruence---since it affects the full isometry group.
\end{rem}

\section{The admissibility condition}\label{sec:admissibility}
This section serves as the backbone of the article: here, we will prove a structural result that will significantly simplify the nilpotent construction problem and essentially render it solvable. Our line of argument is going to be broken into three steps. First, we prove Theorem \ref{thm:B}---a general result on groups that act transitively and isometrically on symmetric spaces of noncompact type. We then use that result to show that one of the groups involved in the nilpotent construction can be assumed to look particularly nice with respect to an Iwasawa decomposition. Finally, we show how that puts serious restrictions on the subspace $\mk{w} \subset \mk{n}_j^1$ and forces it to be positioned nicely with respect to the restricted root space decomposition.

\subsection{Transitive groups of isometries of symmetric spaces of noncompact type}\label{sec:admissibility:transitive}

In \cite{onishchik_transitive}, Onishchik obtained (among other things) a classification of transitive isometric actions on irreducible symmetric spaces of compact type (see also \cite[Prop.\hspace{2pt}21]{kollross_hyperpolar_reducible}). There is no analogous result for symmetric spaces of noncompact type. The core reason for this lies in the fact that the isometry group in the noncompact case is itself noncompact and thus its subgroups are generally more complicated (for starters, they are not necessarily reductive). Our first goal in this section is to obtain a general result that will give a uniform description of transitive actions on symmetric spaces of noncompact type.

First of all, we recall the classification of maximal solvable subalgebras---also known as \textit{Borel subalgebras}---obtained by Mostow in \cite{mostow_maximal_subgroups}. Let $\mk{g}$ be a real semisimple Lie algebra and $\mk{c} \subset \mk{g}$ a Cartan subalgebra (for a background on Cartan subalgebras in real semisimple Lie algebras, see \cite[Ch.\hspace{0.2pt}VI, Sect.\hspace{2pt}6]{knapp}). There exists a Cartan involution $\uptheta$ that preserves $\mk{c}$; we write $\mk{g} = \mk{k} \oplus \mk{p}$ for the corresponding Cartan decomposition and thus have $\mk{c} = \tilde{\mk{t}} \oplus \tilde{\mk{a}}$, where $\tilde{\mk{t}} = \mk{c} \cap \mk{k}$ and $\tilde{\mk{a}} = \mk{c} \cap \mk{p}$. Note that $\mk{c}$ is abelian and consists of semisimple elements, and its subspaces $\tilde{\mk{t}}$ and $\tilde{\mk{a}}$ can be described invariantly as:
\begin{align*}
    \tilde{\mk{t}} &= \set{X \in \mk{c} \mid \text{all eigenvalues of $\ad_\mk{g}(X)$ are purely imaginary}}, \\
    \tilde{\mk{a}} &= \set{X \in \mk{c} \mid \text{all eigenvalues of $\ad_\mk{g}(X)$ are real}}.
\end{align*}
In particular, for every $X \in \tilde{\mk{a}}$, $\ad(X)$ is diagonalizable as an operator on $\mk{g}$. The dimensions of $\tilde{\mk{t}}$ and $\tilde{\mk{a}}$ are called the compact and noncompact dimensions of $\mk{c}$, respectively. Pick a maximal abelian subspace $\mk{a} \subset \mk{p}$ containing $\tilde{\mk{a}}$, write $\Upsigma \subset \mk{a}^*$ for the corresponding restricted root system, and let $\mk{g} = \mk{g}_0 \oplus \bigoplus_{\upalpha \in \Upsigma} \mk{g}_\upalpha$ stand for the induced restricted root space decomposition.

\begin{ex}\label{ex:max_nonc_Cartan_subalgebras}
    Pick a maximal abelian subspace $\mk{t}_0$ in $\mk{k}_0$. Then $\mk{t}_0 \oplus \mk{a}$ is a Cartan subalgebra of $\mk{g}$. It is said to be \textit{maximally noncompact} since it has the largest possible noncompact dimension. All maximally noncompact Cartan subalgebras are congruent via $\Inn(\mk{g})$.
\end{ex}

The following was shown in \cite{polar-c2-foliations}:

\begin{prop}\label{prop:Cartan_subalgebra_classification}
    There exists a choice of simple roots $\Uplambda$ for $\Upsigma$ and a subset $\Upphi \subseteq \Uplambda$ such that $\tilde{\mk{a}} = \mk{a}_\Upphi$ and $\tilde{\mk{t}}$ is a maximal abelian subspace of $\mk{k}_\Upphi$.
\end{prop}

In the context of this proposition, maximally noncompact Cartan subalgebras correspond precisely to $\Upphi = \varnothing$. In \cite[Sect.\hspace{2pt}4]{mostow_maximal_subgroups}, Mostow showed that every Borel subalgebra of $\mk{g}$ contains a Cartan subalgebra. From that, he was able to get a complete description of all Borel subalgebras. Here we give a reformulation of Mostow's result obtained in \cite[Th.\hspace{2pt}2.2]{polar-c2-foliations}, which is slightly better suited for our purposes:

\begin{thm}\label{thm:Borel_subalgebra_classification}
    Let $\mk{g}$ be a real semisimple Lie algebra and $\mk{b} \subset \mk{g}$ a Borel subalgebra. Then there exists a choice of a Cartan decomposition $\mk{g} = \mk{k} \oplus \mk{p}$, a maximal abelian subspace $\mk{a} \subset \mk{p}$, a set of simple roots $\Uplambda \subset \Upsigma$, a subset $\Upphi \subseteq \Uplambda$, and a maximal abelian subspace $\tilde{\mk{t}} \subseteq \mk{k}_\Upphi$ such that $\mk{b} = \tilde{\mk{t}} \oplus \mk{a}_\Upphi \oplus \mk{n}_\Upphi$.
\end{thm}

We call a subspace $W$ of $\mk{g}$ \textit{Iwasawa-adapted} if $W = (W \cap \mk{k}) \oplus (W \cap \mk{a}) \oplus (W \cap \mk{n})$. We see from Theorem \ref{thm:Borel_subalgebra_classification} that \textit{every Borel subalgebra is Iwasawa-adapted with respect to some Iwasawa decomposition}.

\begin{ex}\label{ex:max_nonc_Borel_subalgebras}
    If $\mk{c} = \mk{t}_0 \oplus \mk{a}$ is a maximally noncompact Cartan subalgebra, then every Borel subalgebra containing it is of the form $\mk{t}_0 \oplus \mk{a} \oplus \mk{n}$ for some choice of $\Upsigma^+ \subset \Upsigma$. Such a Borel subalgebra is also said to be \textit{maximally noncompact}.
\end{ex}

Using Mostow's classification, we can obtain a general structural result on groups acting transitively by isometries on symmetric spaces of noncompact type. First, we need the following

\begin{lem}\label{lem:solvable_subalgebra_transitive}
    Let $M$ be a symmetric space of noncompact type. Suppose $H$ is a Lie group acting on $M$ transitively and isometrically. Then there exists a connected solvable Lie subgroup $S$ of $H$ that still acts transitively on $M$.
\end{lem}

\begin{proof}
    Let $G$ stand for $I^0(M)$ and write $\upvarphi \colon \mk{h} \to \mk{g} = \Lie(G)$ for the homomorphism arising from the action. Pick a Levi decomposition $\mk{h} = \mk{l} \loplus \mm{rad}(\mk{h})$ and an Iwasawa decomposition $\mk{l} = \mk{k} \oplus \mk{a} \oplus \mk{n}$, and write $\mk{l} = \mk{k} \oplus \mk{p}$ for the corresponding Cartan decomposition. Consider the restriction $\uppsi \colon \mk{l} \to \mk{g}$ of $\upvarphi$. Since both $\mk{l}$ and $\mk{g}$ are semisimple, there exists a Cartan decomposition $\mk{g} = \overline{\mk{k}} \oplus \overline{\mk{p}}$ such that $\uppsi(\mk{k}) \subseteq \overline{\mk{k}}$ (see \cite[Ch.\hspace{2pt}6, Cor.\hspace{2pt}1]{onishchik_lectures}). The subalgebra $\overline{\mk{k}} \subset \mk{g}$ is the isotropy Lie algebra of some point $o \in M$. Consider the subalgebra $\mk{s} = (\mk{a} \loplus \mk{n}) \loplus \mm{rad}(\mk{h})$ of $\mk{h}$; it is solvable as a semidirect sum of two solvable Lie algebras. Since $\upvarphi(\mk{k}) \subseteq \overline{\mk{k}}$, $\upvarphi(\mk{s})$ has the same projection in $\mk{g}/\overline{\mk{k}} \cong T_oM$ as $\upvarphi(\mk{h})$, which is the whole $T_oM$---by the transitivity of $H \curvearrowright M$. In other words, the $o$-orbit of the connected Lie subgroup $S \subseteq H$ corresponding to $\mk{s}$ is open. This orbit is then a complete totally geodesic submanifold of $M$ and thus coincides with the whole $M$.
    This completes the proof.
\end{proof}

In the proof, we have used some ideas from \cite[Prop.\hspace{2pt}2.2]{hyperpolar_homogeneous_foliations}\footnote{We would like to point out a slight issue in the proof of that proposition: the authors do not show that the subgroup $K$ is compact. However, this is rectifiable---e.g., by using an argument similar to ours.}. We are now ready to prove the main result of this subsection\footnote{Compare this to \cite[Prop.\hspace{2pt}5.1]{hyperpolar_homogeneous_foliations}.}.

\begin{prop}\label{prop:transitive_noncompact}
    Let $M$ be a symmetric space of noncompact type represented by a Riemannian symmetric pair $(G,K)$ with compact ineffectiveness kernel $\mk{i} \triangleleft \mk{g}$ {\normalfont(}note that $G$ is necessarily reductive{\normalfont)}. Write $\overline{\mk{g}} \trianglelefteq \mk{g}$ for the {\normalfont(}unique{\normalfont)} ideal complementary to $\mk{i}$. Suppose $H \subseteq G$ is a Lie subgroup that acts transitively on $M$. Then there exist
    \begin{enumerate}[\normalfont (a)]
        \item an Iwasawa decomposition $\overline{\mk{g}} = \mk{k} \oplus \mk{a} \oplus \mk{n}$,
        \item a maximal abelian subspace $\mk{t}_0 \subseteq \mk{k}_0$, and
        \item a maximal abelian subspace $\mk{t} \subseteq \mk{i}$,
    \end{enumerate}
    such that $\mk{h} = \Lie(H)$ contains a subalgebra of the form $V \oplus \mk{n}$, where $V \subseteq \mk{t} \oplus \mk{t}_0 \oplus \mk{a}$ is a subspace that projects surjectively onto $\mk{a}$.
\end{prop}

The need to formulate this proposition in such a general and technical form will become evident when we apply it to the nilpotent construction in Proposition \ref{prop:adjusting_the_normalizer}. Note that Theorem \ref{thm:B} is a special case of Proposition \ref{prop:transitive_noncompact} when the pair $(G,K)$ is almost effective.

\begin{proof}[Proof of Proposition {\normalfont \ref{prop:transitive_noncompact}}]
    By Lemma \ref{lem:solvable_subalgebra_transitive}, there exists a solvable subalgebra $\mk{s} \subseteq \mk{h}$ whose corresponding connected Lie subgroup still acts transitively on $M$. Let $\mk{b} \subset \mk{g}$ be a maximal solvable Lie subalgebra containing $\mk{s}$. Since $\mk{g} = \mk{i} \oplus \overline{\mk{g}}$ is a direct sum decomposition, $\mk{b}$ must necessarily be of the form $\mk{t} \oplus \overline{\mk{b}}$, where $\mk{t}$ is maximal solvable in $\mk{i}$ and $\overline{\mk{b}}$ is one in $\overline{\mk{g}}$---otherwise, $\pr_\mk{i}(\mk{b}) \oplus \pr_{\overline{\mk{g}}}(\mk{b})$ would be a larger solvable subalgebra. We claim that $\mk{t}$ is abelian. Indeed, being a subalgebra of the compact Lie algebra $\mk{i}$, $\mk{t}$ is itself compact, hence reductive. But since it is also solvable, it must be abelian. Observe that $\overline{\mk{g}} \cong \mk{g}/\mk{i}$ is naturally isomorphic to the isometry Lie algebra of $M$ and hence is semisimple (\cite[Ch.\hspace{0.2pt}V, Th.\hspace{2pt}4.1]{helgason}). According to Theorem \ref{thm:Borel_subalgebra_classification}, we may assume that $\overline{\mk{b}}$ is Iwasawa-adapted to some Iwasawa decomposition $\overline{\mk{g}} = \mk{k} \oplus \mk{a} \oplus \mk{n}$. Let us denote the underlying Cartan decomposition by $\overline{\mk{g}} = \mk{k} \oplus \mk{p}$. In order for its corresponding Lie subgroup to act transitively, $\mk{s}$ has to project surjectively onto $\mk{a} \oplus \mk{n} \cong \mk{p}$ along $\mk{i} \oplus \mk{k}$. This is only possible when $\overline{\mk{b}} \cap \mk{a} = \mk{a}$ and $\overline{\mk{b}} \cap \mk{n} = \mk{n}$, which means that $\overline{\mk{b}}$ is maximally noncompact, i.e., $\overline{\mk{b}} = \mk{t}_0 \oplus \mk{a} \oplus \mk{n}$ for some maximal abelian subspace $\mk{t}_0 \subseteq \mk{k}_0$ (see Examples \ref{ex:max_nonc_Cartan_subalgebras} and \ref{ex:max_nonc_Borel_subalgebras}).
    
    So far, we have obtained a subalgebra $\mk{s} \subseteq \mk{t} \oplus \mk{t}_0 \oplus \mk{a} \oplus \mk{n}$ of $\mk{h}$ that projects surjectively onto $\mk{a} \oplus \mk{n}$ along $\mk{t} \oplus \mk{t}_0$. What is left to show is that $\mk{s}$ actually contains $\mk{n}$. Write $\Upsigma$ for the restricted root system in $\mk{a}^*$ and take any $\upalpha \in \Upsigma^+$ and $Z \in \mk{g}_\upalpha$. Pick $H \in \mk{a}$ such that $[H,Z] = Z$. There exist $X_1, X_2 \in \mk{t}$ and $Y_1, Y_2 \in \mk{t}_0$ such that $X_1 + Y_1 + H$ and $X_2 + Y_2 + Z$ lie in $\mk{s}$. As $\mk{t} \oplus \mk{t}_0 \oplus \mk{a}$ is abelian and $\mk{i}$ commutes with $\overline{\mk{g}}$, we have:
    $$
    \mk{s} \ni [X_1 + Y_1 + H, X_2 + Y_2 + Z] = [Y_1, Z] + Z.
    $$
    The adjoint representation of $\mk{k}_0$ (and thus $\mk{t}_0$) on $\mk{g}$ preserves each root space and is orthogonal. This means that $[Y_1, Z]$ lies in $\mk{g}_\upalpha$ and is orthogonal to $Z$. We deduce that $\mk{s} \cap \mk{g}_\upalpha$ enjoys the following property: for every $Z \in \mk{g}_\upalpha$, $\mk{s} \cap \mk{g}_\upalpha$ contains a vector whose orthogonal projection in $\R Z$ is $Z$. It is clear that $\mk{s} \cap \mk{g}_\upalpha$ must then coincide with $\mk{g}_\upalpha$. We see that $\mk{s}$ contains all positive root spaces and thus the whole $\mk{n}$. This concludes the proof.
\end{proof}

\subsection{The redundancy}\label{sec:admissibility:redundancy}

Our next goal is to apply Proposition \ref{prop:transitive_noncompact} to the admissibility condition in the nilpotent construction, since it has to do with a certain action on a boundary component being transitive. More precisely, applying that proposition would relate the normalizer $N_{\mk{m}_j}(\mk{w})$ to some Iwasawa decomposition of the subalgebra $\mk{g}'_j \subset \mk{g}$. We have to overcome the following practical difficulty:

\begin{quote}
    \textit{An Iwasawa decomposition of $\mk{g}'_j$ arising from Proposition {\normalfont \ref{prop:transitive_noncompact}} may not agree with the existing ambient Iwasawa decomposition of $\mk{g}$.}
\end{quote}

We will deal with this using a rather technical argument---but in order to make it neat, we are going to embrace a more conceptual moduli-space-like approach. Observe that some of the choices we make in the nilpotent construction---like the point $o$ or the maximal flat $\mk{a}$---turn out to be redundant in the end. Out of all the initial data we fix, what really ends up being used in the nilpotent construction are the Lie algebras $\mk{m}_j, \mk{a}_j,$ and $\mk{n}_j$ and the subspace $\mk{w}$ (and, by implication, the parabolic subalgebra $\mk{q}_j$ as well as the boundary component $B_j$). The grading on $\mk{n}_j$ can be recovered from the representation of $\mk{a}_j$ on it. The key idea is to formalize and make use of the following informal observation: altering some of the redundant data should not affect the nilpotent construction. To that end, we will have to make a slight change of perspective on the conditions of admissibility and protohomogeneity. We begin with a little detour into parabolic subgroups. Most of the exposition here is simply mimicry of the constructions bulleted in Subsection \ref{sec:preliminaries:parabolic} but without making unnecessary choices (see Remark \ref{rem:parabolic_abstract_vs_standard} below).

Let $M = G/K$ be a symmetric space of noncompact type. Recall that the equivalence classes of asymptotic geodesics in $M$ are called points at infinity (for more on this, see \cite{eberlein_nonpositively_curved}). The set of such classes is denoted by $M(\infty)$, and we also write $\overbar{M} = M \sqcup M(\infty)$. The topology on $M$ can be extended to the so-called cone topology on $\overbar{M}$, which makes $\overbar{M}$ homeomorphic to a closed Euclidean ball (with $M(\infty)$ being the boundary sphere and $M$ the interior). Given any $o \in M$ and $X \in T_oM$, we write $\upgamma_X$ for the geodesic emanating from $o$ with $\dot{\upgamma}_X(0) = X$, and we denote its asymptotic class by $\upgamma_X(\infty)$. This way, sending $X$ to $\upgamma_X(\infty)$ establishes a homeomorphism between the unit sphere in $T_oM$ and $M(\infty)$. The action of $I(M)$ (and thus of $G$) on $M$ extends naturally to a continuous action on $\overbar{M}$. For every $x \in M(\infty)$, the stabilizer $G_x$ is a parabolic subgroup, and every proper parabolic subgroup arises in this way. A general such subgroup can stabilize several points at infinity. Even though we do not need that, it is not hard to show that for maximal proper parabolic subgroups, this is no longer the case---hence they are parametrized by a subset\footnote{More precisely, by a collection of orbits of $G$ in $M(\infty)$, which is called the set of \textit{maximally singular} points at infinity in \cite{eberlein_nonpositively_curved}.} of $M(\infty)$.

Let $\wh{Q} = G_x$ be a maximal proper parabolic subgroup of $G$. Fix a base point $o \in M$ and take the unique unit vector $X \in T_oM$ with $\upgamma_X(\infty) = x$. Write $R$ for the curvature tensor of $M$ and consider the subspace $\wh{\mk{f}} = \set{Y \in T_oM \mid R(Y,X) = 0}$ of $T_oM$. This subspace is curvature-invariant (meaning, $R(\wh{\mk{f}},\wh{\mk{f}}) \hspace{1pt}\wh{\mk{f}} \subseteq \wh{\mk{f}}$), hence it corresponds to a totally geodesic submanifold $\wh{F} = \exp(\hspace{0.5pt}\wh{\mk{f}}\hspace{0.5pt})$, which is an extended boundary component.\label{parabolic_abstract} We can split $\wh{\mk{f}}$ orthogonally as the sum $\wh{\mk{b}} \oplus \R X$ of two curvature-invariant subspaces, whose corresponding totally geodesic submanifolds $\wh{B} = \exp(\wh{\mk{b}})$ (a boundary component) and $\exp(\R X)$ are the noncompact and flat parts of $\wh{F}$, respectively, and we have $\wh{F} = \wh{B} \times \exp(\R X)$. This shows that $\R X$ is actually determined by $\wh{\mk{f}}$. Consider the Cartan decomposition $\mk{g} = \mk{k} \oplus \mk{p}$, where $\mk{k} = \Lie(K)$, and let $\uptheta$ be the corresponding Cartan involution. By identifying $\mk{p}$ and $T_oM$, we can think of $\R X$ as an abelian subspace $\wh{\mk{a}} \subset \mk{p} \subset \mk{g}$ and thus define $\wh{\mk{l}} = Z_\mk{g}(\wh{\mk{a}})$ and $\wh{\mk{m}} = \wh{\mk{l}} \ominus \wh{\mk{a}}$. These subalgebras are clearly $\uptheta$-stable and have $\wh{\mk{f}}$ and $\wh{\mk{b}}$, respectively, as their $\mk{p}$-parts. It then follows that the subgroup $\wh{L} = Z_G(\wh{\mk{a}})$ has $\wh{F}$ as its $o$-orbit. We can further define $\wh{\mk{g}}$ to be the derived subalgebra of $\wh{\mk{m}}$ (or $\wh{\mk{l}}$) and $\wh{\mk{g}}'$ to be the noncompact part of $\wh{\mk{g}}$. We write $\wh{G}'$ and $\wh{A}$ for the connected Lie subgroups of $G$ corresponding to $\wh{\mk{g}}'$ and $\wh{\mk{a}}$ and let $\wh{K} = K \cap \wh{L}$ and $\wh{M} = \wh{K}\wh{G}'$. Finally, we let $\wh{\mk{n}}$ and $\wh{N}$ stand for the nilradical of $\wh{\mk{q}}$ and its corresponding connected Lie subgroup, respectively. The subalgebra $\wh{\mk{l}} = \wh{\mk{m}} \oplus \wh{\mk{a}}$ normalizes $\wh{\mk{n}}$, and we have the Langlands decomposition $\wh{\mk{q}} = \wh{\mk{m}} \oplus \wh{\mk{a}} \loplus \wh{\mk{n}}$. As an operator on $\wh{\mk{n}}$, $\ad(X)$ is diagonalizable with eigenvalues $\upvarepsilon, \ldots, k\upvarepsilon, \hspace{1pt} (\upvarepsilon > 0, k \ge 1)$, so $\wh{\mk{n}}$ becomes graded by the corresponding eigenspaces: $\wh{\mk{n}} = \bigoplus_{\upnu \ge 1} \wh{\mk{n}}^{\hspace{0.5pt}\upnu}$. The group $\wh{L}$, which also normalizes $\wh{\mk{n}}$, respects this grading. We also have the horospherical decomposition $M \cong \wh{F} \times \wh{N} \cong \wh{B} \times \wh{A} \times \wh{N}$.

Now, since $\wh{L}$ stabilizes $x$ and centralizes $\wh{\mk{a}}$, if we were to carry out these constructions at a different base point $g \cdot o \in \wh{F}, \hspace{1pt} g \in \wh{L}$, we would obtain the same subalgebras $\wh{\mk{a}}, \wh{\mk{l}}, \wh{\mk{m}}, \wh{\mk{g}}, \wh{\mk{g}}' \hspace{1pt}$, the same grading on $\wh{\mk{n}}$, as well as the same submanifold $\wh{F}$. The submanifolds $\wh{B}$ and $\exp(\R X)$ would become $g(\wh{B})$ and $g(\exp(\R X))$, which technically can differ from $\wh{B}$ and $\exp(\R X)$, but these are still the noncompact and flat parts of $\wh{F}$. If one repeats these constructions with the same $x$ but at an arbitrary point of $M$, one can easily see that $M$ gets foliated by extended boundary components all mutually congruent via the action of $\wh{N}$, and the above subalgebras of $\wh{\mk{q}}$ depend only on which leaf of this foliation we fix. Let us say that a maximal proper extended boundary component is \textit{marked} if one has a fixed orientation on its flat factor (clearly, this can be done at just one point and then extended to other points via transvections). For example, $\wh{F}$ is marked because we can orient $\exp(\R X)$ in the direction of $X$. With this additional piece of data, $\wh{F}$ alone fully recovers $x$ and thus $\wh{Q}$.

We will now relate this abstract exposition to the standard theory of parabolic subalgebras laid out in Subsection \ref{sec:preliminaries:parabolic}. Fix a base point $o \in \wh{F}$, take the unit vector $X \in T_oM$ with $\upgamma_X(\infty) = x$, and write $\mk{g} = \mk{k} \oplus \mk{p}$ for the corresponding Cartan decomposition. Identify $\mk{p}$ with $T_oM$, pick a maximal abelian subspace $\mk{a} \subset \mk{p}$ containing $X$, and write $\Upsigma \subset \mk{a}^*$ for the corresponding restricted root system. According to \cite[Prop.\hspace{2pt}2.17.13(1)]{eberlein_nonpositively_curved}, one has 
\begin{equation}\label{parabolic_via_roots}
\wh{\mk{q}} = \mk{g}_0 \oplus \smashoperator{\bigoplus_{\substack{\upalpha \in \Upsigma \mathstrut \\ \upalpha(X) \ge 0}}} \mk{g}_\upalpha.
\end{equation}
Make a choice of simple roots $\Uplambda$ such that $X$ lies in the closure of the positive Weyl chamber; this way, all the positive roots are nonnegative on $X$. If we write $\Upphi = \set{\upalpha \in \Uplambda \mid \bilin{\upalpha}{X} = 0)}$, then, using \eqref{parabolic_via_roots}, one can easily check that $\wh{\mk{q}} = \mk{q}_\Upphi$. But $\wh{\mk{q}}$ is maximal, which means that $\Upphi$ has to contain all but one simple roots. In other words, there exists a unique $\upalpha_j \in \Uplambda$ that is positive on $X$, and we have $\wh{\mk{q}} = \mk{q}_j$. From this, we immediately get $\wh{\mk{l}} = \mk{l}_j, \, \wh{\mk{a}} = \mk{a}_j, \, \wh{\mk{n}} = \mk{n}_j, \, \wh{F} = F_j, \, \wh{M} = M_j$, and so on. 

\begin{rem}\label{rem:parabolic_abstract_vs_standard}
    Having made these identifications, it is now easy to verify all the relevant statements on p.\ \pageref{parabolic_abstract} by translating them into statements about $\mk{q}_j, \mk{n}_j$, etc. For example, $\wh{\mk{f}}$ is indeed curvature-invariant because it gets identified with $\mk{f}_j$, which, as we know, is a Lie triple system.
\end{rem}

Below, we will show that a marked extended boundary component $\wh{F}$ is enough to be able to formulate the nilpotent construction. Before we do that, however, let us investigate which data turns out to be redundant. Consider the following two sets:

\vspace{-3ex}
\begin{align*}
    \wt{\mc{B}} = \{ (o,\mk{a}, \Upsigma^+, \upalpha_j) \mid \, &o \in M, \\ 
    &\mk{a} \subset \mk{p} \hspace{5pt} \text{a maximal abelian subspace}, \\
    &\Upsigma^+ \subset \Upsigma \hspace{5pt} \text{a set of positive roots}, \\
    &\upalpha_j \in \Uplambda \subseteq \Upsigma^+ \hspace{4pt} \text{a simple root} \}, \\[1ex]
    \mc{B} = \{ \wh{F} \subset M \mid \, &\wh{F} \hspace{5pt} \text{a marked maximal proper extended boundary component} \}.
\end{align*}
\vspace{-1.5ex}

Note that $\mk{p}$ is determined by $o$ as the orthogonal complement of $\mk{k} = \Lie(G_o)$ with respect to $B$. We have a map $\mc{F} \colon \wt{\mc{B}} \twoheadrightarrow \mc{B}$ that sends $(o,\mk{a}, \Upsigma^+, \upalpha_j)$ to $F_j$, marked by orienting $\mk{a}_j$ in the direction of $H^j$.

\begin{prop}\label{prop:NC_fiber_Iwasawa}
    The map $\mc{F}$ is well defined. Moreover, for any $\wh{F} \in \mc{B}$, there is a natural one-to-one correspondence between the fiber $\mc{F}^{-1}(\wh{F})$ and the set of pairs $(\wh{B}, \widecheck{K}\widecheck{A}\widecheck{N})$, where $\wh{B} \subset \wh{F}$ is the noncompact part of $\wh{F}$, and $\widecheck{K}\widecheck{A}\widecheck{N}$ is an Iwasawa decomposition of $\wh{G}'$.
\end{prop}

\begin{proof}
    The first part means that $F_j$ is indeed of the form $\wh{F}$ for some $\wh{F}$ constructed as above. To see this, note that $\mk{q}_j = \mk{g}_0 \oplus \bigoplus_{\upalpha(H^j) \ge 0} \mk{g}_\upalpha$. In view of \eqref{parabolic_via_roots}, this coincides with the Lie algebra $\wh{\mk{q}}$ of $G_x$, where $x = \upgamma_{H^j}(\infty)$ and we think of $H^j$ as a tangent vector at $o$. But now both $\mk{f}_j$ and $\wh{\mk{f}}$ can be described as the kernel of $R(-,H^j)$ in $T_oM$, which implies $F_j = \wh{F}$. 
    
    Now we prove the second assertion. Given $(o,\mk{a}, \Upsigma^+, \upalpha_j) \in \mc{F}^{-1}(\upchi)$, we simply assign to it the pair $(B_j, K^j A^j N^j)$. Conversely, suppose we have fixed $\wh{B} \subset \wh{F}$ and an Iwasawa decomposition $\wh{G}' = \widecheck{K}\widecheck{A}\widecheck{N}$. The latter is equivalent to choosing a point $o \in \wh{B}$ (this gives rise to $\mk{g} = \mk{k} \oplus \mk{p}$), a maximal abelian subspace $\widecheck{\mk{a}}$ in $\wh{\mk{b}} \subset \mk{p}$, and a set of simple roots $\wh{\Uplambda}$ in the resulting restricted root system $\wh{\Upsigma} \subseteq \widecheck{\mk{a}}^*$. Observe that $\mk{a} = \widecheck{\mk{a}} \oplus \wh{\mk{a}}$ is a maximal abelian subspace of $\mk{p}$ and denote its corresponding root system by $\Upsigma \subset \mk{a}^*$. Since $\wh{B}$ is a boundary component, we know that $\wh{\Upsigma} = \Upsigma \cap \widecheck{\mk{a}}^*$. Recall that $\wh{F}$ is marked, so we can pick a positively oriented vector $X \in \wh{\mk{a}}$. We claim that there is a unique completion of $\wh{\Uplambda}$ to a set of simple roots for $\Upsigma$ all of whom are nonnegative on $X$. Indeed, it follows from our discussion after \eqref{parabolic_via_roots} that there exists a choice of simple roots $\Uplambda$ for $\Upsigma$ and $\upalpha_j \in \Uplambda$ such that every element of $\Uplambda$ is nonnegative on $X$ and $\wh{\Upsigma} = \Upsigma_j$. Since the Weyl group acts transitively on the set of Weyl chambers, there exists $w \in W(\wh{\Upsigma})$ mapping $\Uplambda_j$ to $\wh{\Uplambda}$. Being a composition of root hyperplane reflections with respect to some roots from $\wh{\Uplambda}$, $w$ extends naturally to an element $\wt{w} \in W(\Upsigma)$. But $\wh{\mk{a}}$ is orthogonal to $\widecheck{\mk{a}}$, which implies that the dual operator of $\wt{w}$ acts as the identity on $\wh{\mk{a}}$. Consequently, the system $\wt{w}(\Uplambda) = \wh{\Uplambda} \sqcup \wt{w}(\upalpha_j)$ of simple roots for $\Upsigma$ is nonnegative on $(\wt{w}^*)^{-1}(X) = X$, so we can take it to be our desired completion of $\wh{\Uplambda}$. The uniqueness of such a completion is straightforward. By invoking the discussion after \eqref{parabolic_via_roots} one more time, we see that $(o, \mk{a}, \wt{w}(\Uplambda), \wt{w}(\upalpha_j)) \in \mc{F}^{-1}(\wh{F})$, hence we can assign this tuple to the pair $(\wh{B}, \widecheck{K}\widecheck{A}\widecheck{N})$. It it easy to see that the constructed maps between $\mc{F}^{-1}(\wh{F})$ and the set of such pairs are the inverses of each other. This completes the proof.
\end{proof}

We are now ready to formulate the nilpotent construction in this new setting. 

\begin{defn}\label{defn:NC_new}
    Let $\wh{F} \in \mc{B}$. A subspace $\mk{w} \subseteq \wh{\mk{n}}^1$ is called

    \begin{enumerate}[(a)]
    \item \textsc{admissible} if $N_{\wh{L}}(\mk{w})$ acts transitively on the extended boundary component $\wh{F}$,
    \item \textsc{protohomogeneous} if some compact subgroup of $N_{\wh{L}}(\mk{w})$ acts transitively on the set of directions in $\wh{\mk{n}}^1/\mk{w}$ (that is, on $\Gr^+(1, \wh{\mk{n}}^1/\mk{w})$).\footnote{Notice how we had to resort to the quotient representation to avoid talking about orthogonal complements.}
\end{enumerate}
\end{defn}

If we are also given a tuple $(o,\mk{a}, \Upsigma^+, \upalpha_j) \in \mc{F}^{-1}(\wh{F})$, we have $\wh{\mk{n}}^1 = \mk{n}_j^1$, so there could potentially be a conflict between this definition and the original one.

\begin{prop}\label{prop:NC_original_new_equivalence}
    Let $\wh{F} \in \mc{B}, \, (o,\mk{a}, \Upsigma^+, \upalpha_j) \in \mc{F}^{-1}(\wh{F}),$ and $\mk{w} \subseteq \wh{\mk{n}}^1 = \mk{n}_j^1$.

    \begin{enumerate}[\normalfont (a)]
        \item\customlabel{prop:NC_original_new_equivalence:a}{a} $\mk{w}$ is admissible with respect to Definition {\normalfont \ref{defn:NC_original}} if and only if it is such with respect to Definition {\normalfont \ref{defn:NC_new}}.
        \item\customlabel{prop:NC_original_new_equivalence:b}{b} If $\mk{w}$ is protohomogeneous with respect to Definition {\normalfont \ref{defn:NC_original}}, it is also such with respect to Definition {\normalfont \ref{defn:NC_new}}. Conversely, if $\mk{w}$ is protohomogeneous with respect to Definition {\normalfont \ref{defn:NC_new}}, there exists $(\overline{o}, \overline{\mk{a}}, \overline{\Upsigma}^+, \overline{\upalpha}_l) \in \mc{F}^{-1}(\wh{F})$ such that $\mk{w}$ is protohomogeneous with respect to Definition {\normalfont \ref{defn:NC_original}} and that tuple.
        \item\customlabel{prop:NC_original_new_equivalence:c}{c} $\mk{w}$ is both admissible and protohomogeneous with respect to Definition {\normalfont \ref{defn:NC_original}} if and only if it is such with respect to Definition {\normalfont \ref{defn:NC_new}}.
    \end{enumerate}
\end{prop}

The asymmetry in part \eqref{prop:NC_original_new_equivalence:b} of the proposition is only a technicality and will not cause any problems because we do not care about the protohomogeneity condition on its own---we always consider it only together with admissibility.

\begin{proof}[Proof of Proposition {\normalfont \ref{prop:NC_original_new_equivalence}}]
    The first assertion follows from the fact that $\wh{L} = L_j$ and $\wh{F} = F_j$. To prove \eqref{prop:NC_original_new_equivalence:b}, notice that the quotient projection $\mk{v} \isoto \mk{n}_j^1 /\mk{w}$ is an isomorphism of $N_{K_j}(\mk{w})$-representations. If $N_{K_j}(\mk{w})$ acts transitively on the unit sphere in $\mk{v}$, it also does so on the set of directions in $\mk{v}$ and thus in $\mk{n}_j^1/\mk{w} = \wh{\mk{n}}^1/\mk{w}$. This proves the first assertion, as $N_{K_j}(\mk{w})$ is a compact subgroup of $N_{L_j}(\mk{w}) = N_{\wh{L}}(\mk{w})$. For the converse, let $K_1$ be a compact subgroup of $N_{\wh{L}}(\mk{w})$ that acts transitively on $\Gr^+(1, \wh{\mk{n}}^1/\mk{w})$. Pick a maximal compact subgroup $\overbar{\wh{K}}$ of $\wh{L}$ containing $K_1$. By the Cartan fixed point theorem applied to $\wh{L} \curvearrowright \wh{F}$, $\overbar{\wh{K}}$ must be the isotropy of some point $\overline{o} \in \wh{F}$. Complete $\overline{o}$ to a tuple $(\overline{o}, \overline{\mk{a}}, \overline{\Upsigma}^+, \overline{\upalpha}_l) \in \mc{F}^{-1}(\wh{F})$. By construction, $N_{\overbar{K}_l}(\mk{w}) = N_{\overbar{\wh{K}}}(\mk{w})$ acts transitively on the set of directions in $\wh{\mk{n}}^1/\mk{w} = \overline{\mk{n}}^1_l/\mk{w}$ and thus in $\overline{\mk{v}} = \overline{\mk{n}}^1_l \ominus \mk{w}$, where the orthogonal complement is taken with respect to the new inner product $B_{\overline{\uptheta}}$ associated with $\overline{o}$. Since $\overbar{K}_l$ preserves this inner product, we deduce that it acts transitively on the unit sphere in $\overline{\mk{v}}$. This completes the proof of \eqref{prop:NC_original_new_equivalence:b} and leaves us only with \eqref{prop:NC_original_new_equivalence:c}. In light of \eqref{prop:NC_original_new_equivalence:a} and \eqref{prop:NC_original_new_equivalence:b}, we need only show the following: if $\mk{w}$ is admissible and protohomogeneous with respect to Definition \ref{defn:NC_new}, then it is also protohomogeneous with respect to Definition \ref{defn:NC_original} (and the same tuple $(o,\mk{a}, \Upsigma^+, \upalpha_j)$). Once again, let $K_1$ be a compact subgroup of $N_{\wh{L}}(\mk{w})$ acting transitively on $\Gr^+(1, \wh{\mk{n}}^1/\mk{w})$. We may assume that $K_1$ is a maximal compact subgroup of $N_{\wh{L}}(\mk{w})$. But so is $N_{K_j}(\mk{w}) = N_{\wh{K}}(\mk{w})$ because the quotient $N_{\wh{L}}(\mk{w})/N_{\wh{K}}(\mk{w})$, which is diffeomorphic to $\wh{F}$ by the admissibility, is contractible (see \cite{maximal_compact_subgroup}). Since every two maximal compact subgroups are conjugate, we deduce that $N_{K_j}(\mk{w})$ acts transitively on the set of directions in $\mk{n}_j^1/\mk{w}$. By the same argument as above, it also acts transitively on the unit sphere in $\mk{v}$. This concludes the proof.
\end{proof}

Fix some $\wh{F} \in \mc{B}$. Given $\mk{w} \subseteq \wh{\mk{n}}^1$, we have a nilpotent subalgebra $\wh{\mk{n}}_\mk{w} = \mk{w} \oplus \bigoplus_{\upnu \ge 2} \wh{\mk{n}}^{\hspace{0.5pt}\upnu}$ and its corresponding connected Lie subgroup $\wh{N}_\mk{w}$. We can also form a semidirect sum $\wh{\mk{h}}_\mk{w} = N_{\wh{\mk{l}}}(\mk{w}) \loplus \wh{\mk{n}}_\mk{w}$ and its corresponding connected Lie subgroup $\wh{H}_\mk{w} = N^0_{\wh{L}}(\mk{w}) \ltimes \wh{N}_\mk{w}$. Combining Theorem \ref{thm:berndt_tamaru_NC} with Proposition \ref{prop:NC_original_new_equivalence} leads to the following

\begin{cor}\label{cor:NC_new}
    Let $M = G/K$ be a symmetric space of noncompact type with a fixed choice of $\wh{F} \in \mc{B}$. Suppose $\mk{w} \subset \wh{\mk{n}}^1$ is an admissible and protohomogeneous subspace of $\codim_{\wh{\mk{n}}^1}(\mk{w}) \ge 2$. Then the subgroup $\wh{H}_\mk{w}$ acts on $M$ with cohomogeneity one and has $\wh{F} \times \wh{N}_\mk{w} \subset \wh{F} \times \wh{N} \cong M$ as its orbit through $o$. That orbit has codimension equal to $\codim_{\wh{\mk{n}}^1}(\mk{w})$ and is thus singular.
\end{cor}

Let us now introduce the following sets of \textit{nilpotent construction data}:

\vspace{-3ex}
\begin{align*}
    \wt{\mc{N}} = \{ (o,\mk{a}, \Upsigma^+, \upalpha_j, \mk{w}) \mid \, &(o,\mk{a}, \Upsigma^+, \upalpha_j) \in \wt{\mc{B}}, \\ 
    &\mk{w} \subset \mk{n}_j^1 \hspace{4.5pt} \text{an admissible protohomogeneous subspace of $\codim \ge 2$} \}, \\[1ex]
    \mc{N} = \{ (\wh{F}, \mk{w}) \mid \, &\wh{F} \in \mc{B}, \\ 
    &\mk{w} \subset \wh{\mk{n}}^1 \hspace{4.5pt} \text{an admissible protohomogeneous subspace of $\codim \ge 2$} \}.
\end{align*}
\vspace{-1em}

Thanks to Propositions \ref{prop:NC_fiber_Iwasawa} and \ref{prop:NC_original_new_equivalence}, we have a map $\mc{G} \colon \wt{\mc{N}} \twoheadrightarrow \mc{N}$ that sends $(o,\mk{a}, \Upsigma^+, \upalpha_j, \mk{w})$ to $(F_j, \mk{w})$, where $F_j$ is marked by orienting $\mk{a}_j$ in the direction of $H^j$. This map is surjective and its fibers admit a simple description thanks to Proposition \ref{prop:NC_fiber_Iwasawa}. Let $\mc{C}$ stand for the set of closed connected subgroups of $G$ acting on $M$ with cohomogeneity one. We can then define another map $\mc{H} \colon \mc{N} \to \mc{C}$ that sends $(\wh{F}, \mk{w})$ to $\wh{H}_\mk{w}$. The original nilpotent construction is then simply the composition $\mc{H} \circ \mc{G}$. The group $G$ (or even $I(M)$) acts naturally on each of these three sets. For instance, given $g \in G$, we have
$$
g \cdot (o,\mk{a}, \Upsigma^+, \upalpha_j, \mk{w}) = (g \cdot o, \Ad(g)(\mk{a}), \Ad(g)_\mk{a}(\Upsigma^+), \Ad(g)_\mk{a}(\upalpha_j), \Ad(g)(\mk{w})),
$$
where $\Ad(g)_\mk{a} = (\restr{\Ad(g)}{\mk{a}}^*)^{-1} \colon \mk{a}^* \isoto \Ad(g)(\mk{a})^*$. It is a matter of routine checks that this map is well-defined and that $\Ad(g)(\mk{w})$ is admissible and protohomogeneous. Essentially, this boils down to the fact that all the relevant subalgebras and representations arising from the 5-tuple on the right are identified with those arising on the left via $\Ad(g)$. Similarly, $G$ acts on $\mc{N}$ by $g \cdot (\wh{F}, \mk{w}) = (g(\wh{F}), \Ad(g)(\mk{w}))$, where the orientation on the flat factor of $g(\wh{F})$ is carried over from $\wh{F}$ via $g$. Finally, $G$ acts on $\mc{C}$ via subgroup conjugation. One can readily see that the maps $\mc{G}$ and $\mc{H}$ are $G$-equivariant (see \cite[Lem.\hspace{2pt}4.1.5]{solonenko_thesis} for a similar statement and its proof). 

\begin{rem}
    If we already have some nilpotent construction data $(\wh{F}, \mk{w}) \in \mc{N}$ fixed, it is useful to consider the action of $\wh{L}$ on $\mc{N}$ and its orbit through $(\wh{F}, \mk{w})$: since $\wh{L}$ preserves $\wh{F}$ (as a marked boundary component) and $\wh{\mk{n}}^1$, this action reduces to moving subspaces of $\wh{\mk{n}}^1$ around by means of the representation $\wh{L} \curvearrowright \wh{\mk{n}}^1$. By the equivariance of $\mc{H}$, for every $g \in \wh{L}$, the action of $\mc{H}(\wh{F}, \Ad(g)(\mk{w}))$ is orbit-equivalent to that of $\mc{H}(\wh{F}, \mk{w})$. This allows us to strengthen part \eqref{thm:berndt_tamaru_NC:b} of Theorem \ref{thm:berndt_tamaru_NC}, where only the congruence with respect to $K_j = \wh{K}$ was considered. We will make use of this in Corollary \ref{cor:moving_w_instead} below.
\end{rem}

\begin{obs}\label{obs:NC_acting_in_the_fiber}
    Let us consider all possible C1-actions that can be obtained by performing the nilpotent construction over a given $\wh{F} \in \mc{B}$. Fix some $(o,\mk{a}, \Upsigma^+, \upalpha_j) \in \mc{F}^{-1}(\wh{F})$. We know three things: $\wh{L} = \wh{M} \times \wh{A}$, $\wh{A}$ acts transitively on the set of noncompact parts $\wh{B} \subset \wh{F}$ (this action looks like shifting the noncompact parts of $\wh{F}$ by transvections along an orthonormal geodesic), and $\wh{G}' \subseteq \wh{M}$ acts transitively on the set of its own Iwasawa decompositions. In view of Proposition \ref{prop:NC_fiber_Iwasawa}, we see that every element $(\overline{o},\overline{\mk{a}}, \overline{\Upsigma}^+, \overline{\upalpha}_l, \overline{\mk{w}}) \in \wt{\mc{N}}$ with $(\overline{o},\overline{\mk{a}}, \overline{\Upsigma}^+, \overline{\upalpha}_l) \in \mc{F}^{-1}(\wh{F})$ can be moved by $\wh{L}$ to a tuple of the form $(o,\mk{a}, \Upsigma^+, \upalpha_j, \mk{w})$. Essentially, this means that by fixing $o, \mk{a}, \Upsigma^+,$ and $\upalpha_j$ and going through all possible admissible and protohomogeneous subspaces $\mk{w} \subset \mk{n}_j^1$ of $\codim \ge 2$, we exhaust all C1-actions that can be obtained from $\wh{F}$ via nilpotent construction.
\end{obs}

Having all these preparations finally ready, we can prove the main result of this subsection:

\begin{prop}\label{prop:adjusting_the_normalizer}
    Let $M = G/K$ be a symmetric space of noncompact type, and let $\upchi = (\wh{F}, \mk{w}) \in \mc{N}$ be some nilpotent construction data. Then there exist $\wt{\upchi} = (o,\mk{a}, \Upsigma^+, \upalpha_j, \mk{w}) \in \wt{\mc{N}}$ with $\mc{G}(\wt{\upchi}) = \upchi$ and a maximal abelian subspace $\mk{t}_0 \subseteq \mk{k}_0$ such that the normalizer $N_{\mk{m}_j}(\mk{w})$ contains a solvable subalgebra $V \oplus \mk{n}^j$, where $V \subseteq \mk{t}_0 \oplus \mk{a}^j$ projects surjectively onto $\mk{a}^j$.
\end{prop}

In a nutshell, this proposition says that in the nilpotent construction, we can always assume that $N_{\mk{m}_j}(\mk{w})$ contains the whole nilpotent part $\mk{n}^j$ of the Iwasawa decomposition of $\mk{g}'_j$ and ``almost contains" its abelian part $\mk{a}^j$. For the sake of brevity, if $(o,\mk{a}, \Upsigma^+, \upalpha_j, \mk{w}) \in \wt{\mc{N}}$ satisfies this property---that is, if $N_{\mk{m}_j}(\mk{w})$ has a subalgebra of the form $V \oplus \mk{n}^j$ as above---we will call it \textbf{positive}.

\begin{proof}[Proof of Proposition {\normalfont \ref{prop:adjusting_the_normalizer}}]
    Let $(\wh{F}, \mk{w})$ be as in the proposition. By admissibility, the subgroup $N_{\wh{L}}(\mk{w}) \subseteq \wh{L}$ acts transitively on the extended boundary component $\wh{F}$. Pick a noncompact part $\wh{B}$ of $\wh{F}$. Since we are dealing with maximal boundary components, we can use reformulation (iv) of the admissibility condition (see p.\ \pageref{admissibility_M}), which in our redundancy-free setting reads: the subgroup $N_{\wh{M}}(\mk{w}) \subseteq \wh{M}$ acts transitively on the boundary component $\wh{B}$. Recall that $\wh{B}$ can be represented by the Riemannian symmetric pair $(\wh{M}^0, \wh{K}^0)$. The ineffectiveness kernel $\mk{i}$ of $\wh{\mk{m}}$ is the sum of its center $\wh{\mk{z}}$ and the compact semisimple part of $\wh{\mk{g}}$. Consequently, $\mk{i}$ is compact, so we can apply Proposition \ref{prop:transitive_noncompact} to the normalizer $N^0_{\wh{M}}(\mk{w})$. For starters, it provides us with an Iwasawa decomposition of $\wh{\mk{g}}'$. Together with $\wh{B}$, it specifies a tuple $\wt{\upchi} = (o,\mk{a}, \Upsigma^+, \upalpha_j, \mk{w}) \in \wt{\mc{N}}$ with $\mc{G}(\wt{\upchi}) = \upchi$---as follows from Proposition \ref{prop:NC_fiber_Iwasawa}. By applying Proposition \ref{prop:transitive_noncompact} further, we also get a maximal abelian subspace $\mk{t}_0^j$ in the $\mk{k}_0$-part $\mk{k}_0^j$ of $\wh{\mk{g}}' = \mk{g}'_j$ and a maximal abelian subspace $\mk{t}$ in $\mk{i} = Z_{\mk{k}_0}(\mk{g}'_j)$. In the end, $N_{\wh{\mk{m}}}(\mk{w}) = N_{\mk{m}_j}(\mk{w})$ contains a solvable subalgebra of the form $V \oplus \mk{n}^j$, where 
    $$
    V \subseteq \mk{t} \oplus \mk{t}_0^j \oplus \mk{a}^j \subseteq Z_{\mk{k}_0}(\mk{g}'_j) \oplus \mk{k}_0^j \oplus \mk{a}^j = \mk{k}_0 \oplus \mk{a}^j
    $$
    projects surjectively onto $\mk{a}^j$ along $\mk{t} \oplus \mk{t}_0^j$. It simply remains to notice that $\mk{t}_0 = \mk{t} \oplus \mk{t}_0^j$ is a maximal abelian subspace in $\mk{k}_0$. This concludes the proof.
\end{proof}

We can also look at Proposition \ref{prop:adjusting_the_normalizer} from a slightly different angle. Suppose we start with a tuple $\wt{\upchi} = (o,\mk{a}, \Upsigma^+, \upalpha_j, \mk{w}) \in \wt{\mc{N}}$, and write $\upchi = \mc{G}(\wt{\upchi}) = (F_j, \mk{w})$. Proposition \ref{prop:adjusting_the_normalizer} asserts that there exists a \textit{positive} tuple $\overline{\wt{\upchi}} = (\overline{o},\overline{\mk{a}}, \overline{\Upsigma}^+, \overline{\upalpha}_l, \overline{\mk{w}})$ in the fiber of $\mc{G}$ over $\upchi$. But thanks to Observation \ref{obs:NC_acting_in_the_fiber}, we can move $\overline{\wt{\upchi}}$ to a tuple of the form $\wt{\upchi}' = (o,\mk{a}, \Upsigma^+, \upalpha_j, \mk{w}')$ by an element of $\wh{L} = L_j$ (or in our case even of $\wh{G}' = G'_j$). Now, whether or not an element of $\wt{\mc{N}}$ is positive is clearly preserved by the action of $G$ on $\wt{\mc{N}}$. Since both $\mc{G}$ and $\mc{H}$ are $G$-equivariant, we arrive at the following alternative formulation of Proposition \ref{prop:adjusting_the_normalizer}:

\begin{cor}\label{cor:moving_w_instead}
    Let $M = G/K$ be a symmetric space of noncompact type and $\wt{\upchi} = (o,\mk{a}, \Upsigma^+, \upalpha_j, \mk{w}) \in \wt{\mc{N}}$ some nilpotent construction data. Then there exists a subspace $\mk{w}' \subset \mk{n}_j^1$ congruent to $\mk{w}$ via $G'_j$ {\normalfont(}and thus also admissible and protohomogeneous{\normalfont)} such that $N_{\mk{m}_j}(\mk{w}')$ contains a solvable subalgebra of the form $V \oplus \mk{n}^j$, where $V \subseteq \mk{t}_0 \oplus \mk{a}^j$ projects surjectively onto $\mk{a}^j$, and $\mk{t}_0$ is a maximal abelian subspace of $\mk{k}_0$. The C1-actions arising from $\mk{w}$ and $\mk{w}'$ are orbit-equivalent.
\end{cor}

\subsection{Root vector elimination}\label{sec:admissibility:root_vector_elimination}

In this final part of the section, we will see how the condition of positivity on nilpotent construction data $(o,\mk{a}, \Upsigma^+, \upalpha_j, \mk{w})$ puts strong restrictions on the subspace $\mk{w}$ and forces it to be positioned well with respect to the restricted root space decomposition. Given a subspace $\mk{w} \subseteq \mk{n}_j^1$ and $\uplambda \in \Updelta_j^1$, we denote $\boldsymbol{\mk{w}_\uplambda} = \mk{w} \cap \mk{g}_\uplambda$.

\begin{lem}\label{lem:NC_containing_a_vs_RRSD_adjusted}
    Let $M = G/K$ be a symmetric space of noncompact type and consider some nilpotent construction data $(o,\mk{a}, \Upsigma^+, \upalpha_j, \mk{w}) \in \wt{\mc{N}}$. The following conditions are equivalent:
\begin{enumerate}[\normalfont (i)]
        \item\customlabel{lem:NC_containing_a_vs_RRSD_adjusted:a}{i} $\mk{w} = \bigoplus_{\uplambda \in \Updelta_j^1} \mk{w}_\uplambda$.
        \item\customlabel{lem:NC_containing_a_vs_RRSD_adjusted:b}{ii} $\mk{a}^j \subseteq N_{\mk{m}_j}(\mk{w})$.
    \end{enumerate}
\end{lem}

\begin{proof}
    One direction is obvious: for every $H \in \mk{a}^j$ and $\uplambda \in \Upsigma$, $\ad(H)$ acts on $\mk{g}_\uplambda$ as the multiplication by $\bilin{\uplambda}{H}$, hence $\ad(H)$ preserves each $\mk{w}_\uplambda$. If $\mk{w}$ decomposes as the sum of $\mk{w}_\uplambda$'s, then $\ad(H)$ preserves it and so $H$ lies in the normalizer. Conversely, assume $\mk{a}^j \subseteq N_{\mk{m}_j}(\mk{w})$. Since $H^j$ always acts on $\mk{n}_j^1$ as the identity, we have that the whole $\mk{a}$ normalizes $\mk{w}$. Take any $X \in \mk{w}$ and write it as $X = \sum_{\uplambda \in \Updelta_j^1} X_\uplambda$, where $X_\uplambda \in \mk{g}_\uplambda$. It suffices to show that $X_\uplambda \in \mk{w}$ for an arbitrary $\uplambda \in \Updelta_j^1$. Take any other $\upgamma \in \Updelta_j^1$. Since $\uplambda$ and $\upgamma$ cannot be proportional, there exists a vector $H \in \mk{a}$ such that $\bilin{\upgamma}{H} = 0$ but $\bilin{\uplambda}{H} = 1$. (Note that we might have $2\uplambda \in \Upsigma$ if the root system is not reduced, but that root would then lie in $\Updelta_j^2$, not $\Updelta_j^1$, so we could not take it as $\upgamma$.) We can replace $X$ with $[H,X] \in \mk{w}$: its $\upgamma$-component is now zero but its $\uplambda$-component is still $X_\uplambda$. By repeating this trick with the other roots in $\Updelta_j^1 \mysetminus \set{\uplambda}$, we obtain $X_\uplambda \in \mk{w}$. This completes the proof.
\end{proof}

Note that we did not actually use the fact that $\mk{w}$ is admissible or protohomogeneous. We will now prove the main result of this subsection.

\begin{prop}\label{prop:w_vs_RRSD}
    Let $M = G/K$ be a symmetric space of noncompact type and $\wt{\upchi} = (o,\mk{a}, \Upsigma^+, \upalpha_j, \mk{w}) \in \wt{\mc{N}}$ some nilpotent construction data. If $\wt{\upchi}$ is positive, then it satisfies the conditions of Lemma {\normalfont \ref{lem:NC_containing_a_vs_RRSD_adjusted}}.
\end{prop}

\begin{proof}
    We will show that $\wt{\upchi}$ satisfies condition \eqref{lem:NC_containing_a_vs_RRSD_adjusted:a} of Lemma \ref{lem:NC_containing_a_vs_RRSD_adjusted}. According to the positivity assumption on $\wt{\upchi}$, there exist a maximal abelian subspace $\mk{t}_0 \subseteq \mk{k}_0$ and a subspace $V \subseteq \mk{t}_0 \oplus \mk{a}^j$ projecting surjectively onto $\mk{a}^j$ such that $V \oplus \mk{n}^j \subseteq N_{\mk{m}_j}(\mk{w})$. Recall that the adjoint representation of $\mk{k}_0$ on $\mk{g}$ preserves each root space and is orthogonal with respect to the inner product $B_\uptheta$. Since $\mk{t}_0$ is abelian, each root space $\mk{g}_\uplambda$ can be decomposed orthogonally as $\mk{g}_\uplambda = \mk{g}_\uplambda^0 \oplus \mk{g}_\uplambda^1 \oplus \cdots \oplus \mk{g}_\uplambda^{k_\uplambda}$ so that the following holds:

    \begin{enumerate}[\normalfont (a)]
        \item For each $1 \le i \le k_\uplambda$, $\mk{g}_\uplambda^i$ is 2-dimensional\footnote{One should be able to show that $\mk{g}_\uplambda^0$ has dimension 1 if the multiplicity of $\uplambda$ is odd and 0 otherwise, but we do not need that here.}.
        \item For every $T \in \mk{t}_0$, $\ad(T)$ preserves each summand $\mk{g}_\uplambda^i$ and is zero on $\mk{g}_\uplambda^0$.
        \item Fix an orthogonal complex structure $J_\uplambda$ on $\mk{g}_\uplambda^1 \oplus \cdots \oplus \mk{g}_\uplambda^{k_\uplambda}$ such that each $\mk{g}_\uplambda^i$ becomes a complex subspace, and write $J_\uplambda^i = \restr{J_\uplambda}{\mk{g}_\uplambda^i}$. Then for every $T \in \mk{t}_0$, $\restr{\ad(T)}{\mk{g}_\uplambda} = c_1 J_\uplambda^1 + \cdots + c_{k_\uplambda} J_\uplambda^{k_\uplambda}$ for some $c_i \in \R$.
    \end{enumerate}

    Given a vector $X \in \mk{g}$, we write $X_\uplambda$ and $X_\uplambda^i$ for its orthogonal projections in $\mk{g}_\uplambda$ and $\mk{g}_\uplambda^i$, respectively. Observe that for every $A \in \mk{t}_0 \oplus \mk{a}$, $\ad(A)$ preserves not only each root space $\mk{g}_\uplambda$ but in fact each subspace $\mk{g}_\uplambda^i$. In particular, if $X \in \mk{g}$ has $X_\uplambda^i = 0$, then so does $[A,X]$.
    
    Given $\uplambda \in \Updelta_j^1$, we write $W_\uplambda$ for the orthogonal projection of $\mk{w}$ in $\mk{g}_\uplambda$. We have $\mk{w}_\uplambda \subseteq W_\uplambda$, and our goal is to show that $\mk{w}_\uplambda = W_\uplambda$ for an arbitrary $\uplambda \in \Updelta_j^1$ (which we fix for the rest of the proof). Assume the converse and pick a vector $X \in \mk{w}$ with $X_\uplambda \in W_\uplambda \mysetminus \mk{w}_\uplambda$. By subtracting a suitable element of $\mk{w}_\uplambda$, we may assume that $X_\uplambda \perp \mk{w}_\uplambda$ (and $X_\uplambda \ne 0$). We will construct a vector in $W_\uplambda \mysetminus \mk{w}_\uplambda$ that also lies in $\mk{w}$, thus arriving at a contradiction. Take a root $\upgamma \in \Updelta_j^1$ different from $\uplambda$ such that $X_\upgamma \ne 0$. Since we can proceed inductively, it suffices to produce another vector $Y \in \mk{w}$ satisfying the following properties:

    \begin{enumerate}[\normalfont (a)]
        \item\customlabel{w_vs_RRSD:Y:a}{a} $Y_\uplambda \in W_\uplambda \mysetminus \mk{w}_\uplambda$.
        \item\customlabel{w_vs_RRSD:Y:b}{b} $Y_\upgamma = 0$.
        \item\customlabel{w_vs_RRSD:Y:c}{c} For every $\updelta \in \Updelta_j^1 \mysetminus \set{\uplambda}$ and every $0 \le i \le k_\updelta$, if $X_\updelta^i = 0$, then $Y_\updelta^i = 0$.
    \end{enumerate}

    To begin with, we introduce an auxiliary subspace $U_\uplambda = W_\uplambda \ominus (\mk{w}_\uplambda \oplus \R X_\uplambda)$; we thus have an orthogonal decomposition $W_\uplambda = \mk{w}_\uplambda \oplus U_\uplambda \oplus \R X_\uplambda$. There exists a vector $H \in \mk{a}$ such that $\bilin{\upgamma}{H} = 0$ but $\bilin{\uplambda}{H} = 1$. We have the subspace $V \oplus \R H^j \subseteq N_{\mk{l}_j}(\mk{w})$ that projects surjectively onto $\mk{a}$. Consequently, there exists $T \in \mk{t}_0$ such that $A = T + H \in N_{\mk{l}_j}(\mk{w})$. We are going to use the operator $\ad(A)$ to eliminate undesired components of $X$. By construction, $\restr{\ad(A)}{\mk{g}_\upgamma} = \restr{\ad(T)}{\mk{g}_\upgamma} = \sum_{i=1}^{k_\upgamma} c_i J_\upgamma^i$ for some $c_i \in \R$. Consider the vector $[A,X] \in \mk{w}$. We have $[A,X]_\upgamma^0 = 0$. On the other hand, since $\ad(T)$ is a skew-symmetric operator, the orthogonal projection of $[A,X]$ in $\R X_\uplambda$ is given by $\pr_{\R X_\uplambda}([H,X_\uplambda]) = \pr_{\R X_\uplambda}(X_\uplambda) = X_\uplambda$. This means that $[A,X]_\uplambda \in W_\uplambda \mysetminus \mk{w}_\uplambda$. Therefore, we may simply assume from the outset that $X_\upgamma^0 = 0$. Suppose that $X_\upgamma^l \ne 0$ for some $1 \le l \le k_\upgamma$. We make yet another reduction: arguing inductively, it suffices to find $Z \in \mk{w}$ that satisfies conditions \eqref{w_vs_RRSD:Y:a} and \eqref{w_vs_RRSD:Y:c} on $Y$, but instead of \eqref{w_vs_RRSD:Y:b} we just require $Z_\upgamma^l = 0$.

    We again consider the vector $[A,X] \in \mk{w}$. We have:
    $$
    \begin{cases}
        [A,X]_\uplambda = [T,X_\uplambda] + X_\uplambda, \\
        [A,X]_\upgamma = \sum_{i=1}^{k_\upgamma} c_i J_\upgamma^i X_\upgamma^i.
    \end{cases}
    $$
    As we already noted earlier, $[T,X_\uplambda] \perp X_\uplambda$ and hence $[A,X]_\uplambda \in W_\uplambda \mysetminus \mk{w}_\uplambda$. If $c_l = 0$, we simply let $Z = [A,X]$, which does the trick. Otherwise, assuming $c_l \ne 0$, we introduce $X' = [A,X] - \pr_{\mk{w}_\uplambda}([A,X]_\uplambda) = [A,X] - \pr_{\mk{w}_\uplambda}([T,X_\uplambda]) \in \mk{w}$. This way, we have $X'_\uplambda \perp \mk{w}_\uplambda$ and thus $X'_\uplambda = \pr_{U_\uplambda}(X'_\uplambda) + X_\uplambda$. We split the argument into two cases.

    {\scshape Case 1:} $[T,X_\uplambda] \in \mk{w}_\uplambda \Leftrightarrow \pr_{U_\uplambda}(X'_\uplambda) = 0 \Leftrightarrow X'_\uplambda = X_\uplambda$. Consider $[A,X'] \in \mk{w}$:
    $$
    \begin{cases}
        [A,X']_\uplambda = [T,X_\uplambda] + X_\uplambda, \\
        [A,X']_\upgamma = -\sum_{i=1}^{k_\upgamma} c_i^2 X_\upgamma^i.
    \end{cases}
    $$
    We let $Z = [A,X'] + c_l^2 X$. By construction, $Z_\upgamma^l = 0$ but the orthogonal projection of $Z$ in $\R X_\uplambda$ equals $(1+c_l^2)X_\uplambda \ne 0$. This means $Z_\uplambda \in W_\uplambda \mysetminus \mk{w}_\uplambda$ and so $Z$ satisfies the required conditions.
    
    {\scshape Case 2:} $[T,X_\uplambda] \not\in \mk{w}_\uplambda \Leftrightarrow \pr_{U_\uplambda}(X'_\uplambda) \ne 0 \Leftrightarrow X'_\uplambda \ne X_\uplambda$. Let us then denote $V = \pr_{U_\uplambda}(X'_\uplambda)$. In this case, we also start by considering the vector $[A,X'] \in \mk{w}$:
    $$
    \begin{cases}
        [A,X']_\uplambda = [T,V] + [T,X_\uplambda] + V + X_\uplambda, \\
        [A,X']_\upgamma = -\sum_{i=1}^{k_\upgamma} c_i^2 X_\upgamma^i.
    \end{cases}
    $$
    Let us adjust $[A,X']$ to make it orthogonal to $\mk{w}_\uplambda$. Consider the vector $X'' = [A,X'] - \pr_{\mk{w}_\uplambda}([A,X']) \in \mk{w}$. We have:
    \begin{align*}
    X''_\uplambda &= [A,X']_\uplambda - \pr_{\mk{w}_\uplambda}([A,X']) \\
    &= ([T,V] - \pr_{\mk{w}_\uplambda}([T,V])) + ([T,X_\uplambda] - \pr_{\mk{w}_\uplambda}([T,X_\uplambda])) + V + X_\uplambda \\
    &= ([T,V] - \pr_{\mk{w}_\uplambda}([T,V])) + 2V + X_\uplambda.
    \end{align*}
    In the last line, the summand in the parentheses is orthogonal to $V$, and so is the last summand $X_\uplambda$. Consequently, the orthogonal projection of $X''$ in $\R V$ equals $2V \ne 0$. We then let $Z = X'' + c_l^2 X$; this vector satisfies $Z_\upgamma^l = 0$ but has the same nonzero projection in $\R V$. The latter means that $Z_\uplambda \in W_\uplambda \mysetminus \mk{w}_\uplambda$, so $Z$ meets the desired requirements. This completes the proof.
\end{proof}

Combining Corollary \ref{cor:moving_w_instead}, Proposition \ref{prop:w_vs_RRSD} and Lemma \ref{lem:NC_containing_a_vs_RRSD_adjusted}, we obtain the main result of this section:

\begin{thm}\label{thm:NC_simplification}
    Let $M = G/K$ be a symmetric space of noncompact type and $\wt{\upchi} = (o,\mk{a}, \Upsigma^+, \upalpha_j, \mk{w}) \in \wt{\mc{N}}$ some nilpotent construction data. Then there exists a subspace $\mk{w}' \subset \mk{n}_j^1$ congruent to $\mk{w}$ via $G'_j$ {\normalfont(}and thus also admissible and protohomogeneous{\normalfont)} such that $N_{\mk{m}_j}(\mk{w}')$ contains $\mk{a}^j \oplus \mk{n}^j$. The subspace $\mk{w}'$ splits as $\mk{w}' = \bigoplus_{\uplambda \in \Updelta_j^1} \mk{w}'_\uplambda$. The C1-actions arising from $\mk{w}$ and $\mk{w}'$ are orbit-equivalent.
\end{thm}

\section{Totally geodesic submanifolds and the singular orbit}\label{sec:t.g._and_singular_orbit}
In this section, we make a slight detour and prove Theorem \ref{thm:C}---a structural result that relates singular orbits of cohomogeneity-one actions on a symmetric space $M = G/K$ of noncompact type, as well as totally geodesic submanifolds therein, to an Iwasawa decomposition of $G$ and allows to represent these types of submanifolds neatly via the solvable model of $M$. We begin with totally geodesic submanifolds.

Let $M = G/K$ be a symmetric space of noncompact type and $S \subseteq M$ a complete connected totally geodesic submanifold. Fix a point $o \in S$ and write $S_{nc}$ and $S_0$ for the noncompact and flat parts of $S$ passing through $o$, respectively. Having fixed the base point, we have the induced Cartan decomposition $\mk{g} = \mk{k} \oplus \mk{p}$ and an identification $\mk{p} \cong T_oM$, so we can introduce the Lie triple systems $\wt{\mk{p}}_{nc} = T_oS_{nc}, \, \wt{\mk{p}}_0 = T_o S_0,$ and $\wt{\mk{p}} = T_oS = \wt{\mk{p}}_{nc} \oplus \wt{\mk{p}}_0$ in $\mk{p}$. It follows from the standard theory of symmetric spaces that $\wt{\mk{g}}_{nc} = [\wt{\mk{p}}_{nc},\wt{\mk{p}}_{nc}] \oplus \wt{\mk{p}}_{nc}$ is a semisimple subalgebra of $\mk{g}$, $\wt{\mk{g}}_0 = \wt{\mk{p}}_0$ is an abelian subalgebra commuting with $\wt{\mk{g}}_{nc}$, and thus $\wt{\mk{g}} = \wt{\mk{g}}_{nc} \oplus \wt{\mk{g}}_0 = [\wt{\mk{p}},\wt{\mk{p}}] \oplus \wt{\mk{p}}$ is a reductive subalgebra. The corresponding connected Lie subgroups $\wt{G}_{nc}, \wt{G}_0,$ and $\wt{G} = \wt{G}_{nc} \times \wt{G}_0$ of $G$ have $S_{nc}, S_0,$ and $S$ as their $o$-orbits, respectively. One can show that these subgroups are closed. (For $\wt{G}_{nc}$, see \cite[Sect.\hspace{2pt}6]{mostow_semisimple_subgroups}, the rest then follows from the global Cartan decomposition, see \cite[Th.\hspace{2pt}6.31(c)]{knapp}.) Pick a maximal abelian subspace $\wt{\mk{a}}$ in $\wt{\mk{p}}$ and then enlarge it to a maximal abelian subspace $\mk{a}$ in $\mk{p}$. Note that we necessarily have $\wt{\mk{a}} = \wt{\mk{a}}_{nc} \oplus \wt{\mk{a}}_0$, where $\wt{\mk{a}}_{nc}$ is maximal abelian in $\wt{\mk{p}}_{nc}$ and $\wt{\mk{a}}_0 = \wt{\mk{p}}_0 = \wt{\mk{g}}_0$. We have an orthogonal decomposition $\mk{a} = \wt{\mk{a}}_{nc} \oplus \wt{\mk{a}}_0 \oplus \wt{\mk{a}}^\perp = \wt{\mk{a}}_{nc} \oplus \wt{\mk{a}}_{nc}^\perp$. (As always, we fix $B_\uptheta$ as a default inner product on $\mk{g}$.) If we identify $\mk{a}$ and $\mk{a}^*$ by means of $B_\uptheta$, we have an induced orthogonal decomposition $\mk{a}^* = \wt{\mk{a}}_{nc}^* \oplus (\wt{\mk{a}}_{nc}^\perp)^*$. Let us denote the restricted root systems of $\wt{\mk{g}}_{nc}$ and $\mk{g}$ by $\wt{\Upsigma} \subseteq \wt{\mk{a}}_{nc}^*$ and $\Upsigma \subset \mk{a}^*$, respectively.

Our first step is to obtain a relation between the restricted root space decompositions of $\wt{\mk{g}}_{nc}$ and $\mk{g}$. Pick a root $\wt{\upalpha} \in \wt{\Upsigma}$. Take any vectors $H \in \wt{\mk{a}}_{nc}$ and $X \in (\wt{\mk{g}}_{nc})_{\wt{\upalpha}}$. By decomposing $X = \sum_{\upalpha \in \Upsigma_0} X_\upalpha$ with respect to the restricted root space decomposition of $\mk{g}$, we can compute:
$$
\bilin{\wt{\upalpha}}{H}\sum_{\upalpha \in \Upsigma_0} X_\upalpha = [H,X] = \sum_{\upalpha \in \Upsigma_0} \bilin{\upalpha}{H} X_\upalpha.
$$
Since this holds for every $H \in \wt{\mk{a}}_{nc}$, we deduce that $X_\upalpha = 0$ whenever the restriction of $\upalpha$ to $\wt{\mk{a}}_{nc}$ differs from $\wt{\upalpha}$. In particular, $X_0 = 0$. If we introduce $\Upsigma(\wt{\upalpha}) = \big\{\upalpha \in \Upsigma \mid \restr{\upalpha}{\wt{\mk{a}}_{nc}} = \wt{\upalpha}\big\}$, then our finding can be reformulated as
\begin{equation}\label{RRSD_relation_tg}
(\wt{\mk{g}}_{nc})_{\wt{\upalpha}} \subseteq \bigoplus_{\upalpha \in \Upsigma(\wt{\upalpha})} \mk{g}_\upalpha.
\end{equation}
For the sake of convenience, we introduce another piece of notation:
$$
\restr{\Upsigma}{\wt{\mk{a}}_{nc}^*} = \big\{\restr{\upalpha}{\wt{\mk{a}}_{nc}^*} \mid \upalpha \in \Upsigma \big\} \mysetminus \set{0}.
$$
This can also be thought of as the orthogonal projection of $\Upsigma$ in $\wt{\mk{a}}_{nc}^*$ (with zero removed if necessary). The inclusion \eqref{RRSD_relation_tg} then implies $\wt{\Upsigma} \subseteq \restr{\Upsigma}{\wt{\mk{a}}_{nc}^*}$. Next, we are going to show that one can make a choice of positive roots for these two root systems in a compatible manner. We can choose a vector $H' \in \wt{\mk{a}}_{nc}$ such that its annulator in $\wt{\mk{a}}_{nc}^*$ does not intersect $\restr{\Upsigma}{\wt{\mk{a}}_{nc}^*}$. Unfortunately, $H'$ can still vanish on some roots in $\Upsigma$, namely on $\Upsigma \cap (\wt{\mk{a}}_{nc}^\perp)^*$. To remedy this, we pick another vector $H'' \in \wt{\mk{a}}_{nc}^\perp$ satisfying the following two properties:

\begin{enumerate}[(a)]
    \item\customlabel{t.g._compatible_choices_of_positivity:a}{a} The annulator of $H''$ in $(\wt{\mk{a}}_{nc}^\perp)^*$ does not intersect $\Upsigma \cap (\wt{\mk{a}}_{nc}^\perp)^*$.
    \item\customlabel{t.g._compatible_choices_of_positivity:b}{b} For every $\upalpha \in \Upsigma \mysetminus (\wt{\mk{a}}_{nc}^\perp)^*, |\bilin{H''}{\upalpha}| < |\bilin{H'}{\upalpha}|$.
\end{enumerate}

Finally, we let $H = H' + H''$. We can now make the following choices of positive roots:
\begin{align*}
    \wt{\Upsigma}^+ &\defeq \{\wt{\upalpha} \in \wt{\Upsigma} \mid \bilin{H'}{\wt{\upalpha}} > 0\}, \\
    \Upsigma^+ &\defeq \set{\upalpha \in \Upsigma \mid \bilin{H}{\upalpha} > 0}.
\end{align*}
By construction, $H'$ does not vanish on any roots in $\wt{\Upsigma}$, and the same is true for $H$ and $\Upsigma$, so these are indeed valid choices of positive roots for the two root systems. Given $\wt{\upalpha} \in \wt{\Upsigma}^+$ and any $\upalpha \in \Upsigma(\wt{\upalpha})$, we have $\bilin{H}{\upalpha} = \bilin{H'}{\upalpha} + \bilin{H''}{\upalpha}$, which has the same sign as $\bilin{H'}{\upalpha} = \bilin{H'}{\wt{\upalpha}} > 0$ thanks to \eqref{t.g._compatible_choices_of_positivity:b} above. This implies $\upalpha \in \Upsigma^+$ and thus $\Upsigma(\wt{\upalpha}) \subseteq \Upsigma^+$. In view of \eqref{RRSD_relation_tg}, we arrive at:
\begin{equation}\label{t.g._nilpotent_parts}
\wt{\mk{n}} = \bigoplus_{\wt{\upalpha} \in \wt{\Upsigma}^+} (\wt{\mk{g}}_{nc})_{\wt{\upalpha}} \subseteq \bigoplus_{\upalpha \in \Upsigma^+} \mk{g}_\upalpha = \mk{n}.
\end{equation}
Having made all these choices, we have the induced Iwasawa decompositions $G = KAN$ and $\wt{G}_{nc} = \wt{K}\wt{A}_{nc}\wt{N}$. By virtue of \eqref{t.g._nilpotent_parts}, we have an inclusion between the solvable parts $\wt{A}_{nc}\wt{N} \subseteq AN$, and the orbit of the former through $o$ is still the whole $S_{nc}$.  We can also consider the subgroup $\wt{A}\wt{N} = \wt{A}_{nc}\wt{A}_0\wt{N}$ of $AN$, whose orbit through $o$ is the original submanifold $S$. We arrive at the following

\begin{prop}\label{prop:t.g._from_AN}
    Let $M = G/K$ be a symmetric space of noncompact type and $S \subseteq M$ a complete connected totally geodesic submanifold. Then there exist an Iwasawa decomposition $G = KAN$ and a connected Lie subgroup $H \subseteq AN$, $H = (H \cap A) \ltimes (H \cap N)$, such that $S = H \cdot o$.
\end{prop}

We are now in a position to prove a result similar to Proposition \ref{prop:t.g._from_AN} but pertaining to singular orbits of C1-actions. We will formulate it in a slightly more general setting to include actions without singular orbits.

\begin{prop}\label{prop:singular_orbit_from_AN}
    Let $M = G/K$ be a symmetric space of noncompact type and $H$ a connected Lie group acting on $M$ properly, isometrically, and with cohomogeneity one. Then there exist an Iwasawa decomposition $G = KAN$ and a connected Lie subgroup $H' \subset AN$ of the form $H' = (H' \cap A) \ltimes (H' \cap N)$ such that the following is true:
    \begin{enumerate}[\normalfont (a)]
        \item\customlabel{prop:singular_orbit_from_AN:a}{a} If $H$ has no singular orbits, then its orbit foliation coincides with that of $H'$.
        \item\customlabel{prop:singular_orbit_from_AN:b}{b} If $H$ has a singular orbit, then that orbit coincides with the orbit of $H'$ through the base point $o$.
    \end{enumerate}
\end{prop}

\begin{proof}
    Part \eqref{prop:singular_orbit_from_AN:a} follows directly from \cite{berndt_tamaru_foliations}. Suppose $H$ has a singular orbit $S$. (As we know from Subsection \ref{sec:c1-actions:types}, it must then be unique.) If $S$ is totally geodesic, the assertion follows immediately from Proposition \ref{prop:t.g._from_AN}. Otherwise, in view of Theorem \ref{thm:classification_of_c1_actions}, the action of $H$ has to come from either the canonical extension or the nilpotent construction up to orbit-equivalence. First, consider the latter. It follows from Theorem \ref{thm:NC_simplification} that there exist nilpotent construction data $(o,\mk{a}, \Upsigma^+, \upalpha_j, \mk{w})$ with $A^j N^j \subseteq N^0_{M_j}(\mk{w})$ such that the orbits of $H$ coincide with those of $H_{j,\mk{w}} = (N^0_{M_j}(\mk{w}) \times A_j) \ltimes N_{j,\mk{w}}$. Note that the orbit of $A^j N^j$ through $o$ is the same as that of $N^0_{M_j}(\mk{w})$---i.e., the whole boundary component $B_j$. But then we can consider the Lie subgroup $H' = (A^j N^j \times A_j) \ltimes N_{j,\mk{w}} = A \ltimes (N^j \ltimes N_{j,\mk{w}}) \subseteq H_{j,\mk{w}}$. With respect to the horospherical decomposition, the orbit of $H'$ through $o$ is given by $B_j \times A_j \times N_{j,\mk{w}}$, which is nothing but the $o$-orbit of $H_{j,\mk{w}}$ and thus the singular orbit of $H$. Plainly, $H'$ satisfies the desired requirements.

    Finally, assume that the action of $H$ arises from the canonical extension. This means that there exist a choice of $o \in M, \, \mk{a} \subset \mk{p}, \, \Upsigma^+ \subset \Upsigma$, and $\Upphi \subset \Uplambda$ and a closed connected subgroup $L \subset G'_\Upphi$ such that the canonical extension of $L \curvearrowright B_\Upphi$ has the same orbits in $M$ as $H$. (We are assuming that the singular orbit of $H$ passes through $o$.) The action of $L$ on $B_\Upphi$ also has cohomogeneity one, and with respect to the horospherical decomposition, the singular orbit of $H$ is given by $(L \cdot o) \times A_\Upphi \times N_\Upphi$. Since the canonical extension procedure is transitive, we may assume the action $L \curvearrowright B_\Upphi$ does not come from any further canonical extension from a proper boundary component of $B_\Upphi$. Thanks to Theorem \ref{thm:classification_of_c1_actions}, this means that this action must either come from the nilpotent construction or have a totally geodesic singular orbit. Either way, we already know by now that it must then satisfy the assertion of the proposition: there exists an Iwasawa decomposition $G'_\Upphi = \widecheck{K} \widecheck{A} \widecheck{N}$ (this entails a new base point $p \in B_\Upphi$) and a connected Lie subgroup $L' \subset \widecheck{A} \widecheck{N}$ of the form $L' = (L' \cap \widecheck{A}) \ltimes (L' \cap \widecheck{N})$ such that the singular orbit $L \cdot o$ of $L$ in $B_\Upphi$ coincides with $L' \cdot p$. This last condition implies that $p$ lies on $L \cdot o$, hence there exists an element $\upvarphi \in G'_\Upphi$ that preserves $L \cdot o$ and maps $o$ to $p$ ($\upvarphi$ can be taken to lie in $L$ or $L'$). By conjugating the new Iwasawa decomposition of $G'_\Upphi$ by $\upvarphi$, we can additionally assume that $p = o$. (We keep the rest of the notation the same.) In particular, we have $\widecheck{K} = K^\Upphi$, as both are now the stabilizers of $o$ in $G'_\Upphi$.
    
    We claim that $G = K(\widecheck{A}A_\Upphi)(\widecheck{N}N_\Upphi)$ is also an Iwasawa decomposition. Indeed, since every two Iwasawa decompositions of $G'_\Upphi$ are inner-conjugate, there exists an element $\uppsi \in G'_\Upphi$ that conjugates $K^\Upphi, A^\Upphi,$ and $N^\Upphi$ to $\widecheck{K}, \widecheck{A},$ and $ \widecheck{N}$, respectively. But that element also conjugates the ambient decomposition $G = KAN$ to some new Iwasawa decomposition of $G$. Since $G'_\Upphi$ centralizes $A_\Upphi$ and normalizes $N_\Upphi$, we easily compute:
    $$
    \uppsi A \uppsi^{-1} =  \uppsi A^\Upphi A_\Upphi \uppsi^{-1} =  \uppsi A^\Upphi \uppsi^{-1} A_\Upphi = \widecheck{A} A_\Upphi,
    $$
    and the same for $N = N^\Upphi N_\Upphi$. Finally, note that $\uppsi$ conjugates $K^\Upphi = \widecheck{K}$ to itself, which means that it preserves $o$ and thus $\uppsi K \uppsi^{-1} = K$. This proves our claim. Now, $G = K(\widecheck{A}A_\Upphi)(\widecheck{N}N_\Upphi)$ is going to be our desired Iwasawa decomposition. As for the group $H'$, we take it to be
    $$
    H' = [(L' \cap \widecheck{A}) \times A_\Upphi] \ltimes [(L' \cap \widecheck{N}) \ltimes N_\Upphi].
    $$
    We need to show that $H' \cdot o = H \cdot o = (L \cdot o) \times A_\Upphi \times N_\Upphi$, where the latter equality is taken with respect to the horospherical decomposition $M \cong B_\Upphi \times A_\Upphi \times N_\Upphi$. Note that with respect to the Langlands decomposition $Q_\Upphi = M_\Upphi \times A_\Upphi \ltimes N_\Upphi$, we can rewrite $H'$ simply as $L' \times A_\Upphi \ltimes N_\Upphi$. But thanks to \eqref{horospherical_decomposition_action}, the orbit of this through $o$ is $(L' \cdot o) \times A_\Upphi \times N_\Upphi$. By construction, $L' \cdot o = L \cdot o$. This completes our proof.
\end{proof}

When combined, Propositions \ref{prop:t.g._from_AN} and \ref{prop:singular_orbit_from_AN} imply Theorem \ref{thm:C}. We close this section with an open question. Note that submanifolds of both types \eqref{thm:C:a} and \eqref{thm:C:b} in Theorem \ref{thm:C} are examples of homogeneous CPC submanifolds. (Here by a \textit{homogeneous submanifold} we mean an orbit of an isometric Lie group action.) This class is known to contain many other types of submanifolds (see \cite{berndt_SL_CPC}), but all of them share one thing in common: they can be realized by a Lie subgroup in the solvable model of $M$. This leads to a natural question:

\begin{openq}
    Let $M = G/K$ be a symmetric space of noncompact type. Can every connected homogeneous CPC submanifold $S \subset M$ be realized as $H \cdot o$ for some Iwasawa decomposition $G = KAN$ and connected Lie subgroup $H \subseteq AN$?
\end{openq}

To our knowledge, there are even no known examples of homogeneous \textit{minimal} submanifolds that cannot be realized in this way.

\section{Solving the nilpotent construction problem}\label{sec:solving_the_NC}
In this final section, we fully resolve the nilpotent construction problem for all symmetric spaces of noncompact type and rank greater than one. In Section \ref{sec:admissibility}, we squeezed everything we could out of the admissibility condition---which culminated in Theorem \ref{thm:NC_simplification}. So in this section, our focus will be on the remaining protohomogeneity condition. Using the conclusions of Theorem \ref{thm:NC_simplification}, we are going to investigate this condition closely within the framework of the restricted root space decomposition and obtain strong restrictions on the subspace $\mk{w}$, the root system $\Upsigma$, and the root multiplicities. Ultimately, this will allow us to determine all possible choices of $\mk{w}$ for all the aforementioned spaces. We will conclude the section by proving the \hyperlink{thm:main}{Main theorem}.

From now on and until the end of the section, we fix a symmetric space $M = G/K$ of noncompact type and let $(o,\mk{a}, \Upsigma^+, \upalpha_j, \mk{w})$ be some \textit{positive} nilpotent construction data (unless otherwise mentioned). We stress that most of the results of this section do not apply to nonpositive nilpotent construction data. We assume that the metric on $M$ is Killing (Remark \ref{rem:choice_of_metric}). Since the nilpotent construction problem has been fully resolved in rank one, we also assume that $\rk(M) > 1$. Finally, we assume that $M$ is irreducible, which is justified by the comment right after Theorem \ref{thm:classification_of_c1_actions}. By virtue of Theorem \ref{thm:NC_simplification}, the subspace $\mk{w}$ splits as $\mk{w} = \bigoplus_{\uplambda \in \Updelta_j^1} \mk{w}_\uplambda$, and hence so does the normal space $\mk{v} = \mk{n}_j^1 \ominus \mk{w} = \bigoplus_{\uplambda \in \Updelta_j^1} \mk{v}_\uplambda$. (Recall that we denote $\mk{w}_\uplambda = \mk{w} \cap \mk{g}_\uplambda$ and $\mk{v}_\uplambda = \mk{v} \cap \mk{g}_\uplambda$.) Moreover, the normalizer $N_{\mk{m}_j}(\mk{w})$ contains $\mk{a}^j \oplus \mk{n}^j$. To start off, we show that this normalizer also decomposes nicely with respect to the restricted root space decomposition.

\begin{lem}\label{lem:normalizer_RRSD}
    We have:\vspace{-2.5ex}
    
    \begin{equation*}
        N_{\mk{m}_j}(\mk{w}) = N_{\mk{k}_0}(\mk{w}) \oplus \mk{a}^j \oplus \mk{n}^j \oplus \bigoplus_{\upalpha \in \Upsigma_j^-} N_{\mk{g}_\upalpha}(\mk{w}) \quad \text{and} \quad N_{\mk{k}_0}(\mk{w}) = \bigcap_{\uplambda \in \Updelta_j^1} N_{\mk{k}_0}(\mk{w}_\uplambda).
    \end{equation*}
\end{lem}\vspace{-2.5ex}

\begin{proof}
    This lemma can be proven in a fashion similar to how we proved Lemma \ref{lem:NC_containing_a_vs_RRSD_adjusted}. However, we will take a slightly different approach. Take any vector $X = X_0 + H + \sum_{\upalpha \in \Upsigma_j} X_\upalpha$ in $N_{\mk{m}_j}(\mk{w})$, where $X_0 \in \mk{k}_0, \, H \in \mk{a}^j,$ and $X_\upalpha \in \mk{g}_\upalpha$. We want to show that each of these summands lies in the normalizer by itself. Since we already know that $\mk{a}^j \oplus \mk{n}^j \subseteq N_{\mk{m}_j}(\mk{w})$, we may assume that $H = 0$ and $X_\upalpha = 0$ whenever $\upalpha$ is positive. Now take $\upalpha \in \Upsigma_j^-$. In order to normalize $\mk{w}$, $X_\upalpha$ should send $\mk{w}_\uplambda$ to $\mk{w}$ for each $\uplambda \in \Updelta_j^1$. Take any such $\uplambda$ and fix an arbitrary $Y \in \mk{w}_\uplambda$. We have $[X,Y] = [X_0,Y] + \sum_{\upalpha \in \Upsigma_j^-} [X_\upalpha,Y]$. Each summand on the right-hand side lies in a single root space. What is more, the only summand that lies in $\mk{g}_{\uplambda+\upalpha}$ and is potentially nonzero is $[X_\upalpha,Y]$. Since $[X,Y] \in \mk{w}$, we see that $[X_\upalpha,Y]$ lies in $\mk{w}_{\uplambda+\upalpha}$. This shows that $X_\upalpha$ sends $\mk{w}_\uplambda$ to $\mk{w}$ and hence $X_\upalpha \in N_{\mk{g}_\upalpha}(\mk{w}) = N_{\mk{m}_j}(\mk{w}) \cap \mk{g}_\upalpha$. As this applies to every $\upalpha \in \Upsigma_j^-$, we are left with $X_0 \in N_{\mk{k}_0}(\mk{w})$. The second equality in the lemma is straightforward, as $\mk{k}_0$ preserves every root space.
\end{proof}

There are three normalizers at play in the nilpotent construction: $N_{\mk{m}_j}(\mk{w}), N_{\mk{m}_j}(\mk{v}),$ and $N_{\mk{k}_j}(\mk{v})$. Our next goal is to relate them with each other.

\begin{lem}\label{lem:three_normalizers}
    The following relations hold:
    \begin{enumerate}[\normalfont (a)]
        \item\customlabel{lem:three_normalizers:a}{a} $N_{\mk{m}_j}(\mk{v}) = \uptheta N_{\mk{m}_j}(\mk{w})$.
        \item\customlabel{lem:three_normalizers:b}{b} For any $\upalpha \in \Upsigma_j^+$, $N_{\mk{g}_\upalpha}(\mk{v}) = \uptheta N_{\mk{g}_{-\upalpha}}(\mk{w})$. We denote this subspace of $\mk{g}_\upalpha$ by $\boldsymbol{\mk{g}'_\upalpha}$.
        \item\customlabel{lem:three_normalizers:c}{c} For any $\upbeta \in \Upsigma_j^-$, $N_{\mk{g}_\upbeta}(\mk{w}) = \uptheta N_{\mk{g}_{-\upbeta}}(\mk{v})$. We denote this subspace of $\mk{g}_\upbeta$ by $\boldsymbol{\mk{g}'_\upbeta}$.
        \item\customlabel{lem:three_normalizers:d}{d} For any $\upalpha \in \Upsigma_j$, $N_{\mk{k}_\upalpha}(\mk{v}) = N_{\mk{k}_\upalpha}(\mk{w}) = \pr_{\mk{k}_\upalpha}(\mk{g}'_\upalpha) = \pr_{\mk{k}_\upalpha}(\mk{g}'_{-\upalpha})$. We denote this subspace of $\mk{k}_\upalpha$ by $\boldsymbol{\mk{k}'_\upalpha}$.
    \end{enumerate}
\end{lem}

The projections in part \eqref{lem:three_normalizers:d} are assumed to be---as always---orthogonal with respect to $B_\uptheta$.

\textit{Proof of Lemma {\normalfont \ref{lem:three_normalizers}}.}
    Part \eqref{lem:three_normalizers:a} is a special case of Lemma \ref{lem:NC_normalizers_vs_Uptheta}\eqref{lem:NC_normalizers_vs_Uptheta:b}. Part \eqref{lem:three_normalizers:b} follows from \eqref{lem:three_normalizers:a} because $N_{\mk{g}_\upalpha}(\mk{v}) = N_{\mk{m}_j}(\mk{v}) \hspace{1pt} \cap \hspace{1pt} \mk{g}_\upalpha$ and $\uptheta$ interchanges $\mk{g}_\upalpha$ and $\mk{g}_{-\upalpha}$. Part \eqref{lem:three_normalizers:c} follows from \eqref{lem:three_normalizers:b}, since $\uptheta$ is an involution. Note that, by definition, $\uptheta$ interchanges $\mk{g}'_\upalpha$ and $\mk{g}'_{-\upalpha}$ for any $\upalpha \in \Upsigma_j$. This immediately implies the last equality in \eqref{lem:three_normalizers:d}. The first equality therein follows from Lemma~\ref{lem:NC_normalizers_vs_Uptheta}\eqref{lem:NC_normalizers_vs_Uptheta:c}. Finally, assuming $\upalpha \in \Upsigma_j^+$ without loss of generality, we calculate:
    \begin{align*}\pushQED{\qed}
    N_{\mk{k}_\upalpha}(\mk{w}) &= (N_{\mk{g}_\upalpha}(\mk{w}) \oplus N_{\mk{g}_{-\upalpha}}(\mk{w})) \cap \mk{k} \\
    &= (\mk{g}_\upalpha \oplus \mk{g}'_{-\upalpha}) \cap \mk{k} && (\text{Lemma \ref{lem:normalizer_RRSD}}) \\
    &= (\mk{g}'_\upalpha \oplus \mk{g}'_{-\upalpha}) \cap \mk{k} = \pr_{\mk{k}_\upalpha}(\mk{g}'_\upalpha) && (\uptheta \mk{g}'_\upalpha = \mk{g}'_{-\upalpha}). \\[-0.82\baselineskip] & && \qedhere \popQED
    \end{align*}

Combining Lemmas \ref{lem:normalizer_RRSD} and \ref{lem:three_normalizers} yields the following neat description of the three normalizers:

\begin{cor}\label{cor:three_normalizers_RRSD}
    We have:
    \begin{enumerate}[\normalfont (a)]
        \item\customlabel{cor:three_normalizers_RRSD:a}{a} $N_{\mk{m}_j}(\mk{w}) = N_{\mk{k}_0}(\mk{w}) \oplus \mk{a}^j \oplus \mk{n}^j \oplus \bigoplus_{\upalpha \in \Upsigma_j^-} \mk{g}'_\upalpha$.
        \item\customlabel{cor:three_normalizers_RRSD:b}{b} $N_{\mk{m}_j}(\mk{v}) = N_{\mk{k}_0}(\mk{v}) \oplus \mk{a}^j \oplus \bigoplus_{\upalpha \in \Upsigma_j^+} \mk{g}'_\upalpha \oplus \uptheta\mk{n}^j$.
        \item\customlabel{cor:three_normalizers_RRSD:c}{c} $N_{\mk{k}_j}(\mk{v}) = N_{\mk{k}_j}(\mk{w}) = N_{\mk{k}_0}(\mk{v}) \oplus \bigoplus_{\upalpha \in \Upsigma_j^+} \mk{k}'_\upalpha$.
        \item\customlabel{cor:three_normalizers_RRSD:d}{d} $N_{\mk{k}_0}(\mk{v}) = N_{\mk{k}_0}(\mk{w}) = \bigcap_{\hspace{1pt} \uplambda \in \Updelta_j^1} N_{\mk{k}_0}(\mk{v}_\uplambda) = \bigcap_{\hspace{1pt} \uplambda \in \Updelta_j^1} N_{\mk{k}_0}(\mk{w}_\uplambda)$.
        \end{enumerate}
\end{cor}

Now we come to the crux of our argument, in which we will render the protohomogeneity condition considerably more accessible and relate it to the roots. Recall that this condition asks for the representation $N_{K_j}^0(\mk{v}) \curvearrowright \mk{v}$ to be of cohomogeneity one, or equivalently, to be transitive on the spheres centered at the origin. In general, the problem of determining whether an orthogonal representation is transitive on spheres can be easily restated on the level of Lie algebras. Given a Lie algebra representation $\uprho \colon \mk{g} \to \mk{gl}(V)$ and a vector $v \in V$, we write $\mk{g} \cdot v = \set{\uprho(X)v \mid X \in \mk{g}}$.

\begin{lem}\label{lem:C1_representation}
    Let $G$ be a compact Lie group equipped with a finite-dimensional representation $\uprho \colon G \to \GL(V)$, and consider the induced representation $\uprho_*$ of $\mk{g} = \Lie(G)$ on $V$. The following conditions are equivalent:
    \begin{enumerate}[\normalfont (i)]
        \item\customlabel{lem:C1_representation:i}{i} The action $G \curvearrowright V$ is of cohomogeneity one.
        \item\customlabel{lem:C1_representation:ii}{ii} For some {\normalfont(}hence any{\normalfont)} nonzero vector $v \in V$, $\mk{g} \cdot v$ is a hyperplane in $V$.
    \end{enumerate}
    If these conditions are satisfied and we have a fixed $G$-invariant inner product on $V$, then $\mk{g} \cdot v = v^\perp$.
\end{lem}

\begin{proof}
    First, we claim that the orbit of a nonzero vector $v \in V$ is a hypersurface if and only if $\mk{g} \cdot v$ is a hyperplane. Indeed, that orbit is a hypersurface precisely when $T_v(G \cdot v)$ is a hyperplane, so it suffices to show that, in general, $T_v(G \cdot v) = \mk{g} \cdot v$. Given $X \in \mk{g}$, let us write $\wh{X}$ for its corresponding fundamental vector field on $V$. The tangent space in question can be described as
    $$
    T_v(G \cdot v) = \bigl\{\wh{X}(v) \mid X \in \mk{g}\bigr\}.
    $$
    So the claim will follow if we show that, under the identification $T_v V \cong V$, one has $\wh{X}(v) = \uprho_*(X)v$. This is a matter of a standard computation:
    \begin{align*}
        \wh{X}(v) & = \restr{\frac{d}{dt}}{t=0} \uprho(\exp(tX)) v = \restr{\frac{d}{dt}}{t=0} e^{t\uprho_*(X)} v \\[2pt]
        &= \restr{\frac{d}{dt}}{t=0} \left( v + t\uprho_*(X)v + \frac{t^2\uprho_*(X)^2}{2!} v + \cdots \right) = \uprho_*(X)v.
    \end{align*}
    The action $G \curvearrowright V$ has cohomogeneity one precisely when it has at least one codimension-one orbit, so the above claim implies that \eqref{lem:C1_representation:i} is equivalent to \eqref{lem:C1_representation:ii} being true for some $v$. To show that it is then true for every nonzero vector, equip $V$ with a $G$-invariant inner product. The orbit of $v$ is bound to be the sphere of radius $||v||$ centered at the origin. But then every nonzero orbit of $G$ is also a sphere, so condition \eqref{lem:C1_representation:ii} holds for every nonzero vector. Note that, with respect to an invariant inner product, the representation $\uprho_*$ is by skew-symmetric operators, hence $\mk{g} \cdot v \perp v$, which implies the last assertion.
\end{proof}

This lemma yields the following reformulation of the protohomogeneity condition:

\begin{cor}\label{cor:protohomogeneity_simplified}
    Let $M = G/K$ be a symmetric space of noncompact type with a fixed choice of $o \in M, \, \mk{a} \subset \mk{p}, \, \Upsigma^+ \subset \Upsigma \subset \mk{a}^*,$ and $\upalpha_j \in \Uplambda$. Then a subspace $\mk{w} \subset \mk{n}_j^1$ is protohomogeneous if and only if for some {\normalfont(}and hence any{\normalfont)} nonzero $v \in \mk{v} = \mk{n}_j^1 \ominus \mk{w}$, $[N_{\mk{k}_j}(\mk{v}),v] = \mk{v} \ominus \R v$.
\end{cor}

Note that in this corollary, we do not require $\mk{w}$ to be admissible, let alone the 5-tuple $(o,\mk{a}, \Upsigma^+, \upalpha_j, \mk{w})$ to be positive. Let us now come back to our problem. In general, Corollary~\ref{cor:protohomogeneity_simplified} is just a linearized reformulation of the protohomogeneity condition and is not of much use. Everything changes when we restrict to positive nilpotent construction data and put it in the context of the results we have accumulated so far.

\begin{prop}\label{prop:protohomogeneity_RRSD_crux}
    The set $\Updelta_j^1$ and the subspace $\mk{v}$ are related via the following:
    \begin{enumerate}[\normalfont (a)]
        \item\customlabel{prop:protohomogeneity_RRSD_crux:1}{a} Whenever $\uplambda, \hspace{1pt} \upgamma \in \Updelta_j^1 \,$ {\normalfont ($\uplambda \ne \upgamma$)} are such that both $\mk{v}_\uplambda$ and $\mk{v}_\upgamma$ are nonzero, we must have $\upgamma - \uplambda \in \Upsigma_j$ and $\mk{g}'_{\upgamma-\uplambda} \ne \set{0}$. Moreover, for every nonzero $v \in \mk{v}_\uplambda$, we have $[\mk{g}'_{\upgamma - \uplambda}, v] = \mk{v}_\upgamma$.
        \item\customlabel{prop:protohomogeneity_RRSD_crux:2}{b} For each $k \in \Z_{> 0}$, there is at most one root $\uplambda \in \Updelta_j^1$ of height $k$ such that $\mk{v}_\uplambda \ne \set{0}$.
    \end{enumerate}
\end{prop}

\begin{proof}
    We begin with part \eqref{prop:protohomogeneity_RRSD_crux:1}. Take any nonzero vector $v \in \mk{v}_\uplambda$. Thanks to Corollaries \ref{cor:protohomogeneity_simplified} and \ref{cor:three_normalizers_RRSD}\eqref{cor:three_normalizers_RRSD:c}, we have
    \begin{equation}\label{k_stabilizer_bracket}
    (\mk{v}_\uplambda \ominus \R v) \oplus \smashoperator{\bigoplus_{\updelta \hspace{0.5pt} \in \hspace{0.5pt} \Updelta_j^1 \mysetminus \{\uplambda\}}} \hspace{1.5pt} \mk{v}_\updelta = [N_{\mk{k}_j}(\mk{v}), v] = [N_{\mk{k}_0}(\mk{v}),v] \oplus \bigoplus_{\upalpha \in \Upsigma_j^+} [\mk{k}'_\upalpha,v].
    \end{equation}
    Assume for a moment that $\hgt(\upgamma) \ge \hgt(\uplambda)$. The only summand on the right-hand side of \eqref{k_stabilizer_bracket} that can have a nonzero projection in $\mk{v}_\upgamma \ne \{0\}$ is $[\mk{k}'_{\upgamma-\uplambda},v]$---but this is only possible when $\upgamma-\uplambda \in \Upsigma_j^+$ and $\mk{k}'_{\upgamma-\uplambda} \ne \{0\}$. If $\hgt(\upgamma) < \hgt(\uplambda)$, the same is true with $\upgamma-\uplambda$ replaced by $\uplambda-\upgamma$. (Note that $\mk{k}'_{\upgamma-\uplambda} = \mk{k}'_{\uplambda-\upgamma}$.) In any case, we see that $\upgamma-\uplambda$ must be a root. This root is either positive or negative, which means that $\uplambda$ and $\upgamma$ cannot have the same height. This proves part \eqref{prop:protohomogeneity_RRSD_crux:2}, while we continue with part \eqref{prop:protohomogeneity_RRSD_crux:1}. Since $\mk{k}'_{\upgamma-\uplambda} \ne \set{0}$, we also have $\mk{g}'_{\upgamma-\uplambda} \ne \set{0}$ (and thus $\mk{g}'_{\uplambda-\upgamma} = \uptheta \mk{g}'_{\upgamma-\uplambda} \ne \set{0}$), which follows from Lemma \ref{lem:three_normalizers}\eqref{lem:three_normalizers:d}. We are left to show that $[\mk{g}'_{\upgamma - \uplambda}, v] = \mk{v}_\upgamma$. The left-to-right inclusion follows from the definition of $\mk{g}'_{\upgamma - \uplambda}$ (Lemma \ref{lem:three_normalizers}\eqref{lem:three_normalizers:b}). For the other direction, take any vector $w \in \mk{v}_\upgamma$. By \eqref{k_stabilizer_bracket}, there must exist an element $X+\uptheta X \in \mk{k}'_{\upgamma-\uplambda}$ (here $X \in \mk{g}'_{\upgamma-\uplambda}$ in accordance with Lemma \ref{lem:three_normalizers}\eqref{lem:three_normalizers:d}) such that $[X + \uptheta X, v] = w$. The fact that this lies in $\mk{g}_\upgamma$ forces $[\uptheta X, v]$ to be zero, which leaves us with $[X,v] = w$. This completes the proof.
\end{proof}

The upshot of the above proposition is that, whenever the normal space $\mk{v}$ has nonzero summands in different root spaces, that has direct implications for the root system $\Upsigma_j$ and how it interacts with $\Updelta_j^1$. As explained at the end of Subsection \ref{sec:preliminaries:parabolic}, $\Updelta_j^1$ can be regarded as the $\Uplambda_j$-string of $\upalpha_j$. Note that $\upalpha_j$ is a root of minimum height in this string, and any other root therein has height at least two. By invoking \cite[Prop.\hspace{2pt}3.3]{sanmartin-strings}, we obtain:

\begin{cor}\label{cor:root_lemma_applied}
For every $\uplambda \in \Updelta_j^1$ of height $\ge 2$, there exists a simple root $\upalpha_i \in \Uplambda_j$ such that $\uplambda-\upalpha_i \in \Updelta_j^1$.
\end{cor}


Equipped with this corollary---which will prove extremely useful---we can establish the following important result:

\begin{prop}\label{prop:snake_continuous}
    If there exists $k \in \Z_{> 0}$ such that, for every root $\uplambda \in \Updelta_j^1$ of height $k$, we have $\mk{v}_\uplambda = \set{0}$, then the same is true for every height $m > k$.
\end{prop}

\begin{proof}
    We can argue inductively, so it suffices to show that for every $\upgamma \in \Updelta_j^1$ of height $k+1$, $\mk{v}_\upgamma = \set{0}$. By Corollary \ref{cor:root_lemma_applied}, there exists $\upalpha_i \in \Uplambda_j$ such that $\upgamma - \upalpha_i \in \Updelta_j^1$. We have $\mk{g}_{\upalpha_i} \subset N_{\mk{m}_j}(\mk{w})$ (by Corollary \ref{cor:three_normalizers_RRSD}\eqref{cor:three_normalizers_RRSD:a}) and $\mk{w}_{\upgamma - \upalpha_i} = \mk{g}_{\upgamma - \upalpha_i}$ (by assumption). Therefore,
    $$
    \mk{w}_\upgamma \supseteq [\mk{g}_{\upalpha_i}, \mk{w}_{\upgamma - \upalpha_i}] = [\mk{g}_{\upalpha_i}, \mk{g}_{\upgamma - \upalpha_i}] = \mk{g}_\upgamma,
    $$
    which means $\mk{w}_\upgamma = \mk{g}_\upgamma$, or in other words, $\mk{v}_\upgamma = \set{0}$.
\end{proof}

Since $\mk{v} \ne \set{0}$ and the only root of height 1 in $\Updelta_j^1$ is $\upalpha_j$, this proposition implies in particular that $\mk{v}_{\upalpha_j} \ne \set{0}$. Combining Propositions \ref{prop:protohomogeneity_RRSD_crux}\eqref{prop:protohomogeneity_RRSD_crux:2} and \ref{prop:snake_continuous} leads to the following conclusion:

\begin{cor}\label{cor:snake}
    There exist $m \ge 1$ and roots $\uplambda_1, \ldots, \uplambda_m \in \Updelta_j^1$ with $\hgt(\uplambda_k) = k$ such that $\mk{v} = \bigoplus_{k=1}^m \mk{v}_{\uplambda_k}$ and $\mk{v}_{\uplambda_k} \ne \set{0}$ for each $k=1,\ldots,m$.
\end{cor}

From here on, and for the rest of the section, we fix $m$ and $\uplambda_1, \ldots, \uplambda_m$ as in the corollary. As is becoming evident by now, applying Lemma \ref{lem:C1_representation} to our problem has led to a number of powerful conclusions. We are now going to introduce the second key ingredient of this section, which will help immensely in our arguments and yield strong restrictions on $\mk{v}$. The main idea is simple: we may assume that the C1-action resulting from the nilpotent construction cannot be obtained via canonical extension---otherwise, we could disregard it. Our goal is the following proposition, which, loosely speaking, says that $\mk{v}$ has to be large enough for the resulting action to not come from the canonical extension.

\begin{prop}\label{prop:CE_trick}
     Suppose there exists $\upalpha_i \in \Uplambda_j$ such that $\bilin{H^i}{\uplambda_k} = 0$ for each $k = 1,\ldots,m$. Then the resulting C1-action on $M$ is the canonical extension of a C1-action on the boundary component $B_i$.
\end{prop}

Before we proceed with the proof, we need the following technical

\begin{lem}\label{lem:k_0_of_boundary_component}
    Let $M = G/K$ be a symmetric space of noncompact type with $o \in M, \mk{a} \subset \mk{p},$ and $\Upsigma^+ \subset \Upsigma \subset \mk{a}^*$ fixed. Pick a subset $\Upphi \subseteq \Uplambda$ and consider the subalgebra $\mk{g}'_\Upphi$. The adjoint representation of $\mk{k}_0$ on $\mk{g}$ preserves $\mk{g}'_\Upphi$ and factors through the adjoint representation $\mk{k}^\Upphi_0 \to \mk{gl}(\mk{g}'_\Upphi)$. To be more precise, there exists a surjective Lie algebra homomorphism $\mk{k}_0 \twoheadrightarrow \mk{k}^\Upphi_0$ that makes the following diagram commutative:
    \begin{center}
    \begin{tikzcd}[every label/.append style = {font = \small}]
    \mk{k}_0 \ar[dr, "\ad"'] \ar[rr, dashed, two heads] && \mk{k}^\Upphi_0 \ar[dl, "\ad"] \\
    & \mk{gl}(\mk{g}'_\Upphi)
    \end{tikzcd}
    \end{center}
\end{lem}

Roughly speaking, this lemma says the following. We have the isometry Lie algebra $\mk{g}'_\Upphi$ of the boundary component $B_\Upphi$, as well as its $\mk{k}_0$-part $\mk{k}^\Upphi_0$. The representation of $\mk{k}_0$ on $\mk{g}$ (and of $\mk{k}^\Upphi_0$ on $\mk{g}'_\Upphi$) is an important part of the whole adjoint representation, as it preserves the restricted root space decomposition and acts on each root space individually. We can restrict the representation of $\mk{k}_0$ to $\mk{g}'_\Upphi$ and see if it brings any new operators. Lemma \ref{lem:k_0_of_boundary_component} states that the answer is ``no": all such operators already come from $\mk{k}^\Upphi_0$.

\begin{proof}[Proof of Lemma {\normalfont \ref{lem:k_0_of_boundary_component}}]
    Recall that the adjoint representation of $\mk{k}_0$ on $\mk{g}$ preserves each root space. Since $\mk{g}'_\Upphi$ is generated by $\mk{g}_\upalpha, \, \upalpha \in \Upsigma_\Upphi$, it is also preserved by $\mk{k}_0$. The representation of $\mk{k}_0$ on $\mk{g}'_\Upphi$ is by derivations. But $\mk{g}'_\Upphi$ is semisimple, so we have an isomorphism $\ad \colon \mk{g}'_\Upphi \isoto \Der(\mk{g}'_\Upphi)$. Therefore, the representation $\mk{k}_0 \to \Der(\mk{g}'_\Upphi)$ lifts to a Lie algebra homomorphism $\upkappa \colon \mk{k}_0 \to \mk{g}'_\Upphi$. Now, $\mk{k}_0$ preserves the Cartan decomposition of $\mk{g}$ and hence of $\mk{g}'_\Upphi$, so $\upkappa$ takes values in $\mk{k}^\Upphi$. What is more, $\mk{k}_0$ commutes with $\mk{a}$ and thus with $\mk{a}^\Upphi$. Consequently, the image of $\upkappa$ lies in $(\mk{g}'_\Upphi)_0$. We conclude that $\upkappa(\mk{k}_0) \subseteq \mk{k}^\Upphi \cap (\mk{g}'_\Upphi)_0 = \mk{k}^\Upphi_0$. On the other hand, $\mk{k}^\Upphi_0$ is a subalgebra of $\mk{k}_0$, and $\upkappa$ is the identity map on $\mk{k}^\Upphi_0$ by construction. Consequently, $\upkappa \colon \mk{k}_0 \twoheadrightarrow \mk{k}^\Upphi_0$ is surjective. Recall from Subsection \ref{sec:preliminaries:parabolic} that we have a decomposition $\mk{k}_0 = Z_{\mk{k}_0}(\mk{g}'_\Upphi) \oplus \mk{k}^\Upphi_0$. One can check that $\upkappa \colon \mk{k}_0 \to \mk{k}^\Upphi_0$ is actually the projection along the ideal $Z_{\mk{k}_0}(\mk{g}'_\Upphi)$---although we do not need that, as the proof is already complete.
\end{proof}

We are now ready to prove Proposition \ref{prop:CE_trick}. The assumption in the proposition ensures that the normal space $\mk{v}$ of the singular orbit of our C1-action is tangent to the boundary component $B_i$. Note that any C1-action canonically extended from $B_i$ also shares this property. So it is indeed reasonable to expect that our C1-action arises via canonical extension from $B_i$.

\begin{proof}[Proof of Proposition {\normalfont \ref{prop:CE_trick}}]
    Our goal is to provide a C1-action on $B_i$ whose canonical extension has the same orbits in $M$ as $H_{j,\mk{w}}$. We are going to construct the desired action via the nilpotent construction\footnote{This is a special case of the following general fact: the canonical extension of an action arising via the nilpotent construction itself comes from the nilpotent construction performed on the ambient space. See \cite[Lemma\hspace{2pt}4.3]{DR_DV_Otero_C1}.}. To that end, we need to provide nilpotent construction data for $B_i$. We choose the base point $o \in B_i$, the maximal abelian subspace $\mk{a}^i \subset \mk{b}_i$, the set of positive roots $\Upsigma_i^+ \subset \Upsigma_i$, and the simple root $\upalpha_j \in \Uplambda_i$. In order to avoid notational confusion, we are going to denote the corresponding maximal parabolic subalgebra of $\mk{g}'_i$ and its various related subalgebras by $\mk{q}(i)_j, \mk{m}(i)_j, \mk{a}(i)^j, \mk{n}(i)^j,\mk{n}(i)_j$, and so on. These can also be expressed in terms of the subalgebras $\mk{q}_j, \mk{m}_j$, etc.: $\mk{q}(i)_j = \mk{g}'_i \cap \mk{q}_j, \mk{m}(i)_j = \mk{g}'_i \cap \mk{m}_j$, and so on. We will be particularly interested in $\mk{m}(i)_j$ and how it compares to $\mk{m}_j$. We have:
    \begin{equation}\label{CE_trick_eq1}
    \mk{m}_j = \mk{k}_0 \oplus \mk{a}^j \oplus \bigoplus_{\upalpha \in \Upsigma_j} \mk{g}_\upalpha, \quad \mk{m}(i)_j = \mk{k}^i_0 \oplus \mk{a}(i)^j \oplus \bigoplus_{\upalpha \in \Upsigma_\Upphi} \mk{g}_\upalpha,
    \end{equation}
    where $\Upphi = \Uplambda \mysetminus \set{\upalpha_i,\upalpha_j}$. We can also use $\Upphi$ to give an alternative description of $\mk{a}(i)^j$ and $\mk{n}(i)^j$:
    $$
    \mk{a}(i)^j = \bigoplus_{\upalpha_l \in \Upphi} \R H_{\upalpha_l}, \quad \mk{n}(i)^j = \bigoplus_{\upalpha \in \Upsigma_\Upphi^+} \mk{g}_\upalpha.
    $$
    In compliance with our new notation, we write $\Updelta(i)_j = \Upsigma_i^+ \mysetminus \Upsigma_\Upphi^+ = \Upsigma_i \cap \Updelta_j$. We can decompose $\Updelta_j^1$ as $\Updelta(i)_j^1 \sqcup (\Updelta_j^1 \mysetminus \Updelta(i)_j^1)$. The merit of this decomposition is that $\Upsigma_\Upphi$ ``preserves" each of its two parts: by this we mean that adding a root from $\Upsigma_\Upphi$ to one in $\Updelta(i)_j^1$ results in either another root in $\Updelta(i)_j^1$ or not a root at all, and the same goes for $\Updelta_j^1 \mysetminus \Updelta(i)_j^1$. The corresponding decomposition of $\mk{n}_j^1$ is 
    \begin{equation}\label{CE_trick_eq2}
    \mk{n}_j^1 = \mk{n}(i)_j^1 \oplus (\mk{n}_j^1 \ominus \mk{n}(i)_j^1),
    \end{equation}
    and $\mk{m}(i)_j$ preserves each of the two summands. By assumption, for every $k=1,\ldots,m$, the root $\uplambda_k$ does not have $\upalpha_i$ in its expression via the simple roots, hence $\uplambda_k = \upalpha_j + \sum_{\upalpha_l \in \Upphi} n_l \upalpha_l$. This means that it lies in $\Updelta(i)_j^1$. In view of Corollary \ref{cor:snake}, we have $\mk{v} \subseteq \mk{n}(i)_j^1$, which means that we can write $\mk{w}$ as 
    \begin{equation}\label{CE_trick_eq3}
    \mk{w} = \mk{w}' \oplus (\mk{n}_j^1 \ominus \mk{n}(i)_j^1),
    \end{equation}
    where $\mk{w}' \subset \mk{n}(i)_j^1$. We consider the nilpotent construction on $B_i$ with respect to $\mk{w}'$. First, we would like to show that the normal space $\mk{v}' = \mk{n}(i)_j^1 \ominus \mk{w}'$, taken with respect to the inner product $B^i_\uptheta$ coming from the Killing form $B^i$ of $\mk{g}'_i$, coincides with $\mk{v}$. To see this, first note that $\mk{w} = \bigoplus_{\uplambda \in \Updelta_j^1} \mk{w}_\uplambda$ implies $\mk{w}' = \bigoplus_{\uplambda \in \Updelta(i)_j^1} \mk{w}_\uplambda$. This latter space has two a priori different orthogonal complements in $\mk{n}(i)_j^1$, namely $\mk{v}$ and $\mk{v}'$ (with respect to $B_\uptheta$ and $B^i_\uptheta$, respectively), and the former decomposes as $\mk{v} = \bigoplus_{k=1}^m \mk{v}_{\uplambda_k}$ by Corollary \ref{cor:snake}. But since the root spaces of $\mk{g}'_i$ are mutually orthogonal with respect to both $B_\uptheta$ and $B^i_\uptheta$, we have $\mk{v}' = \bigoplus_{k=1}^m \mk{v}'_{\uplambda_k}$, so it suffices to show that $\mk{v}_{\uplambda_k} = \mk{v}'_{\uplambda_k}$ for each $k=1,\ldots,m$. Consider the compact connected subgroup $(K_0^i)^0$ of $G'_i$ corresponding to $\mk{k}_0^i$. Its representation on $\mk{g}'_i$ has each root space of $\mk{g}'_i$ as an irreducible subrepresentation (see, e.g., \cite[Lem.\hspace{2pt}4.2.5]{solonenko_thesis}). The restrictions of $B_\uptheta$ and $B^i_\uptheta$ to each such root space are $(K_0^i)^0$-invariant, so they have to be proportional (see \cite[Lem.\hspace{2pt}6.15]{ziller_notes} or \cite[Cor.\hspace{2pt}2.1.108]{solonenko_thesis}). Since $\mk{v}_{\uplambda_k}$ and $\mk{v}'_{\uplambda_k}$ are the orthogonal complements of $\mk{w}_{\uplambda_k} = \mk{w}'_{\uplambda_k}$ in $\mk{g}_{\uplambda_k}$ with respect to $B_\uptheta$ and $B^i_\uptheta$, respectively, they have to coincide. Altogether, we deduce $\mk{v} = \mk{v}'$.
    
    In order to carry out the nilpotent construction on $B_i$, we need to show that $\mk{w}'$ is admissible and protohomogeneous. We start by deriving an explicit description of the normalizer $N_{\mk{m}(i)_j}(\mk{w}')$. Due to \eqref{CE_trick_eq3} and the fact that $\mk{m}(i)_j$ preserves each summand in \eqref{CE_trick_eq2}, an element of $\mk{m}(i)_j$ preserves $\mk{w}'$ if and only if it preserves $\mk{w}$. Consequently, $N_{\mk{m}(i)_j}(\mk{w}') = N_{\mk{m}(i)_j}(\mk{w}) = N_{\mk{m}_j}(\mk{w}) \cap \mk{m}(i)_j$. By combining this with Corollary \ref{cor:three_normalizers_RRSD}\eqref{cor:three_normalizers_RRSD:a} and \eqref{CE_trick_eq1}, we get:
    \begin{equation}\label{CE_trick_eq4}
    N_{\mk{m}(i)_j}(\mk{w}') = N_{\mk{k}^i_0}(\mk{w}') \oplus \mk{a}(i)^j \oplus \mk{n}(i)^j \oplus \bigoplus_{\upalpha \in \Upsigma_\Upphi^-} \mk{g}'_\upalpha.
    \end{equation}
    In particular, we see that $\mk{a}(i)^j \oplus \mk{n}(i)^j \subseteq N_{\mk{m}(i)_j}(\mk{w}')$, which implies that $\mk{w}'$ is admissible. We are left to prove that it is also protohomogeneous. We would like to use \eqref{CE_trick_eq4} to describe $N_{\mk{k}(i)_j}(\mk{v})$. Thanks to Lemma \ref{lem:NC_normalizers_vs_Uptheta}\eqref{lem:NC_normalizers_vs_Uptheta:c}, we have $N_{\mk{k}(i)_j}(\mk{v}) = N_{\mk{k}(i)_j}(\mk{w}')$, and the latter can be written as $N_{\mk{k}(i)_j}(\mk{w}') = N_{\mk{m}(i)_j}(\mk{w}') \cap \mk{k}(i)_j$. By plugging \eqref{CE_trick_eq4} into this and making use of Lemma \ref{lem:three_normalizers}\eqref{lem:three_normalizers:d}, we arrive at:
    \begin{equation}\label{CE_trick_eq5}
        N_{\mk{k}(i)_j}(\mk{v}) = N_{\mk{k}^i_0}(\mk{v}) \oplus \bigoplus_{\upalpha \in \Upsigma_\Upphi^+} \mk{k}'_\upalpha.
    \end{equation} 
    Take any nonzero vector $v \in \mk{v}_{\uplambda_1}$. Thanks to Corollary \ref{cor:protohomogeneity_simplified}, we know that $[N_{\mk{k}_j}(\mk{v}), v] = \mk{v} \ominus \R v = (\mk{v}_{\uplambda_1} \ominus \R v) \oplus \bigoplus_{k=2}^m \mk{v}_{\uplambda_k}$. Let us look at this through the lens of the decomposition of $N_{\mk{k}_j}(\mk{v})$ in Corollary \ref{cor:three_normalizers_RRSD}\eqref{cor:three_normalizers_RRSD:c}. First of all, we must have $[N_{\mk{k}_0}(\mk{v}),v] = \mk{v}_{\uplambda_1} \ominus \R v$. By virtue of Lemma \ref{lem:k_0_of_boundary_component}, the representation of $\mk{k}_0$ on $\mk{g}'_i$ has the same image in $\mk{gl}(\mk{g}'_i)$ as $\mk{k}^i_0$. This implies that $[N_{\mk{k}^i_0}(\mk{v}),v] = \mk{v}_{\uplambda_1} \ominus \R v$. Next, we also must have $[\mk{k}'_{\uplambda_k - \uplambda_1},v] = \mk{v}_{\uplambda_k}$ for every $k=2,\ldots,m$. Note that the root $\uplambda_k - \uplambda_1$ lies in $\Upsigma_\Upphi^+$. Finally, if $\upalpha \in \Upsigma_j^+$ differs from $\uplambda_k - \uplambda_1$ for every $k=2,\ldots,m$, then $[\mk{k}'_\upalpha,v] = \set{0}$. Let us now consider $[N_{\mk{k}(i)_j}(\mk{v}),v]$. Putting all of the above together with \eqref{CE_trick_eq5}, we obtain:
    \begin{align*}
        [N_{\mk{k}(i)_j}(\mk{v}),v] &= \bigl[ N_{\mk{k}^i_0}(\mk{v}) \oplus \bigoplus_{\upalpha \in \Upsigma_\Upphi^+} \mk{k}'_\upalpha, v \bigr] = \bigl[ N_{\mk{k}^i_0}(\mk{v}),v \bigr] \oplus \bigoplus_{k=2}^m \bigl[ \mk{k}'_{\uplambda_k - \uplambda_1}, v \bigr] \\
        &= (\mk{v}_{\uplambda_1} \ominus \R v) \oplus \bigoplus_{k=2}^m \mk{v}_{\uplambda_k} = \mk{v} \ominus \R v.
    \end{align*}
    This shows that the representation of $N_{\mk{k}(i)_j}(\mk{v})$ on $\mk{v}$ satisfies condition \eqref{lem:C1_representation:ii} of Lemma \ref{lem:C1_representation} and thus $\mk{w}'$ is protohomogeneous. Altogether, we have shown that $(o, \mk{a}^i, \Upsigma_i^+, \upalpha_j, \mk{w}')$ is nilpotent construction data for $B_i$, hence the subalgebra $\mk{h}(i)_{j,\mk{w}'}$ of $\mk{g}'_i$ induces a C1-action on $B_i$. Its canonical extension is given by
    $$
    (\mk{h}(i)_{j,\mk{w}'})^\Uplambda = \mk{h}(i)_{j,\mk{w}'} \oplus \mk{a}_i \loplus \mk{n}_i = (N_{\mk{m}(i)_j}(\mk{w}') \oplus \mk{a}(i)_j \loplus \mk{n}(i)_{j,\mk{w}'}) \oplus \mk{a}_i \loplus \mk{n}_i.
    $$
    In view of \eqref{CE_trick_eq4}, this can be simplified to
    $$
    (\mk{h}(i)_{j,\mk{w}'})^\Uplambda = N_{\mk{k}^i_0}(\mk{w}') \oplus \mk{a} \oplus (\mk{n} \ominus \mk{v}) \oplus \bigoplus_{\upalpha \in \Upsigma_\Upphi^-} \mk{g}'_\upalpha.
    $$
    If we compare this to $\mk{h}_{j,\mk{w}}$ and take into account Corollary \ref{cor:three_normalizers_RRSD}\eqref{cor:three_normalizers_RRSD:a} and the fact that $N_{\mk{k}^i_0}(\mk{w}') = N_{\mk{k}^i_0}(\mk{w})$, we see that $(\mk{h}(i)_{j,\mk{w}'})^\Uplambda \subseteq \mk{h}_{j,\mk{w}}$. As either subalgebra induces a C1-action on $M$, these actions must have the same orbits. This completes the proof.
\end{proof}

From now on, \textit{we assume that the C1-action coming from $\mk{w}$ does not arise via nontrivial canonical extension}. In our further arguments, Proposition \ref{prop:CE_trick} will manifest itself in the form of the following

\begin{cor}\label{cor:simple_g'_nonzero}
    For each $\upalpha_i \in \Uplambda_j$, the subspace $\mk{g}'_{\upalpha_i}$ is nonzero.
\end{cor}

\begin{proof}
    Fix a simple root $\upalpha_i \in \Uplambda_j$. As $\bilin{H^i}{\upalpha_j} = 0$, Proposition \ref{prop:CE_trick} guarantees that there exists some $2 \le l \le m$ such that $\bilin{H^i}{\uplambda_l} \ne 0$. By reducing $l$ if necessary, we may additionally assume that $\bilin{H^i}{\uplambda_{l-1}} = 0$ (as $\uplambda_1 = \upalpha_j$). Recall that $\mk{v}_{\uplambda_k}$ is nonzero for every $k = 1, \ldots, m$ by Corollary \ref{cor:snake}. By virtue of Proposition \ref{prop:protohomogeneity_RRSD_crux}\eqref{prop:protohomogeneity_RRSD_crux:1}, the difference $\uplambda_l - \uplambda_{l-1}$ must lie in $\Upsigma_j^+$. Since this is a height-1 root, it has to be $\upalpha_i$. Invoking Proposition \ref{prop:protohomogeneity_RRSD_crux}\eqref{prop:protohomogeneity_RRSD_crux:1} one more time, we get $\mk{g}'_{\upalpha_i} \ne \set{0}$.
\end{proof}

We are now in a position to prove the following crucial

\begin{prop}\label{proposition:v:neq:0}
For every root $\uplambda \in \Updelta_j^1$, we have $\mk{v}_\uplambda \neq \{ 0 \}$.
\end{prop}
\begin{proof}
We will proceed by induction on the height of $\uplambda$. If $\hgt(\uplambda) = 1$, then $\uplambda = \upalpha_j = \uplambda_1$ and $\mk{v}_\uplambda \ne \{0\}$ by Corollary \ref{cor:snake}. Assume that $\mk{v}_{\upgamma} \neq \{0\}$ for every $\upgamma \in \Updelta_j^1$ of height $\le k$ and take $\uplambda \in \Updelta_j^1$ of height $k+1$. From Corollary \ref{cor:root_lemma_applied} we get the existence of a root $\upalpha_i \in \Uplambda_j$ such that $\uplambda - \upalpha_i \in \Updelta_j^1$. Put $\upvarepsilon$ to be the largest positive integer such that $\uplambda - \upvarepsilon \upalpha_i$ is a root. By the induction hypothesis, we have $\mk{v}_{\uplambda - \upvarepsilon \upalpha_i} \neq \{0\}$. Now, take a nonzero vector $X \in \mk{g}'_{\upalpha_i}$---this is possible thanks to Corollary \ref{cor:simple_g'_nonzero}. Lemma \ref{lem:string_injective} ensures that $\ad(X)^\upvarepsilon (\mk{v}_{\uplambda - \upvarepsilon \upalpha_i})$ is a nonzero subspace of $\mk{g}_\uplambda$. As $\mk{g}'_{\upalpha_i}$ normalizes $\mk{v}$, this subspace actually lies in $\mk{g}_{\uplambda} \cap \mk{v} = \mk{v}_\uplambda$. This concludes the induction and the proof.
\end{proof}

When combined with Corollary \ref{cor:snake} and Proposition \ref{prop:protohomogeneity_RRSD_crux}\eqref{prop:protohomogeneity_RRSD_crux:1}, this proposition yields very strong restrictions on the root system $\Upsigma$ and the choice of $\upalpha_j$:

\begin{cor}\label{corollary:bifurcation}
We have $\Updelta_j^1 = \set{\uplambda_1, \ldots, \uplambda_m}$ with $\hgt(\uplambda_k) = k$ and $\uplambda_1 = \upalpha_j$. Moreover, $\mk{v} = \bigoplus_{k=1}^m \mk{v}_{\uplambda_k}$, and $\mk{v}_{\uplambda_k}$ is nonzero for each $k = 1,\ldots,m$. Finally, for every $1 \le k, l \le m$ with $k \ne l$, we have $\uplambda_k - \uplambda_l \in \Upsigma_j$.
\end{cor}

This immediately leads to the following powerful

\begin{cor}\label{corollary:alphaj:corner}
The root $\upalpha_j$ must be connected to precisely one simple root in the Dynkin diagram of $\Upsigma$.
\end{cor}

\begin{proof}
Since $M$ is irreducible and of rank at least two, the Dynkin diagram $\DD$ of $\Upsigma$ is connected and contains at least two vertices. Therefore, $\upalpha_j$ must be connected to at least one other simple root in $\DD$---say, to $\upalpha_i \in \Uplambda_j$. Then, the Cartan integer $n_{\upalpha_j \upalpha_i}$ is negative and thus $\upalpha_j + \upalpha_i$ is a root, as follows from~\cite[Prop.\hspace{2pt}2.48(e)]{knapp}. This makes $\upalpha_j + \upalpha_i$ a root of height 2 in $\Updelta_j^1$---and there can be at most one such root according to Corollary \ref{corollary:bifurcation}. Consequently, there can be only one such $\upalpha_i$ in $\Uplambda_j$.
\end{proof}

This corollary whittles down the list of potential candidates for $\upalpha_j$ to only a few options. But for many spaces, we can actually carry this further and rule out all choices of $\upalpha_j$ altogether. Namely, we will now use Corollary \ref{corollary:bifurcation} to rule out the spaces whose root systems are of type $\mm{E}_6$, $\mm{E}_7$, $\mm{E}_8$, $\mm{F}_4$, or $\mm{D}_r$, for any $r \geq 4$. We start with the types $\mm{D}$ and $\mm{E}$. According to Corollary \ref{corollary:alphaj:corner}, we may assume that $\upalpha_j$ is connected to only one simple root in the Dynkin diagram $\DD$ of $\Upsigma$. But for the types $\mm{D}$ and $\mm{E}$, this means that $\DD$ has a connected subgraph of the form\vspace{-0.5ex}
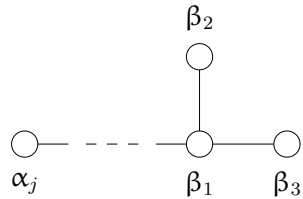
\begin{center}
\resizebox{0.27\columnwidth}{!}{
\begin{tikzpicture}[scale=1.3]
\draw (0, 0) circle (0.15);
\draw (2, 0) circle (0.15);
\draw (3, 0) circle (0.15);
\draw (2, 1) circle (0.15);
\draw (0.15, 0.) -- (0.5, 0.);
\draw (0.7, 0.) -- (0.82, 0.);
\draw (0.94, 0.) -- (1.06, 0.);
\draw (1.18, 0.) -- (1.3, 0.);
\draw (1.5, 0) -- (1.85, 0);
\draw (2., -0.45) node {$\upbeta_1$};
\draw (2.15, 0.) -- (2.85, 0.);
\draw (3., -0.45) node {$\upbeta_3$};
\draw (2., 0.15) -- (2., 0.85);
\draw (2., 1.45) node {$\upbeta_2$};
\draw (0, -0.45) node {$\upalpha_j$};
\end{tikzpicture}
}
\end{center}
(here $\upalpha_j$ may be connected directly to $\upbeta_1$). Since the sum of the simple roots in any connected subgraph of a Dynkin diagram is still a root (\cite[Lem.\hspace{2pt}2.5]{sanmartin-strings}), we see that
\begin{equation}\label{equation:roots:bifurcation:e:d}
\upalpha_j + \dots + \upbeta_1 +\upbeta_2 \quad \text{and} \quad  \upalpha_j + \dots + \upbeta_1 +\upbeta_3
\end{equation}
are both roots in $\Updelta_j^1$. What is more, they have the same height, which contradicts Corollary \ref{corollary:bifurcation}. Let us now focus on the spaces with root systems of type $\mm{F}_4$. Consider the corresponding Dynkin diagram:
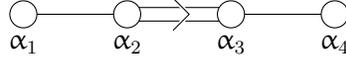
\begin{center}
\begin{dynkinDiagram}[arrow shape/.style={-{Bourbaki[length=8pt]}}, labels={\upalpha_1,\upalpha_2,\upalpha_3,\upalpha_4}, text style/.style={scale=0.9}, scale=3, root radius=.06cm]F4
\end{dynkinDiagram}
\end{center}
First, put $j = 1$. We know from \cite[Prop.\hspace{2pt}5.6]{sanmartin-strings} that
\begin{equation*}
\upalpha_1 + \upalpha_2 + 2\upalpha_3 \quad \text{and} \quad  \upalpha_1 + \upalpha_2 + \upalpha_3 + \upalpha_4
\end{equation*}
are both roots, and they are then clearly elements of $\Updelta_1^1$ of the same height, which again contradicts Corollary \ref{corollary:bifurcation}. For the sake of completeness, we include the diagram of $\Updelta_1^1$ in Figure \ref{figure:f4:j:1} with these two roots corresponding to the blue nodes.
\begin{figure}[h]
\setlength{\abovecaptionskip}{5pt plus 3pt minus 2pt}
\begin{tikzpicture}[scale=1.8]
\draw (0, 0) circle (0.1);
\draw (1, 0) circle (0.1);
\draw (2, 0) circle (0.1);
\filldraw[color = blue] (3, 0) circle (0.1);
\draw (4, 0) circle (0.1);
\filldraw[color = blue] (2, 1) circle (0.1);
\draw (3, 1) circle (0.1);
\draw (4, 1) circle (0.1);
\draw (5, 1) circle (0.1);
\draw (3, 2) circle (0.1);
\draw (4, 2) circle (0.1);
\draw (5, 2) circle (0.1);
\draw (6, 2) circle (0.1);
\draw (7, 2) circle (0.1);
\draw[->] (0.1, 0.) -- (0.5, 0.);
\draw (0.5, 0.) -- (0.9, 0.);
\draw[->] (0.1, 0.) -- (0.5, 0.);
\draw (0.5, 0.) -- (0.9, 0.);
\draw (0.5, -0.15) node {$\upalpha _ 2$};
\draw[->] (1.1, 0.) -- (1.5, 0.);
\draw (1.5, 0.) -- (1.9, 0.);
\draw (1.5, -0.15) node {$\upalpha _ 3$};
\draw[->] (2.1, 0.) -- (2.5, 0.);
\draw (2.5, 0.) -- (2.9, 0.);
\draw (2.5, -0.15) node {$\upalpha _ 3$};
\draw[->] (3.1, 0.) -- (3.5, 0.);
\draw (3.5, 0.) -- (3.9, 0.);
\draw (3.5, -0.15) node {$\upalpha _ 2$};
\draw[->] (3.1, 2.) -- (3.5, 2.);
\draw (3.5, 2.) -- (3.9, 2.);
\draw (3.5, 1.85) node {$\upalpha _ 2$};
\draw[->] (4.1, 2.) -- (4.5, 2.);
\draw (4.5, 2.) -- (4.9, 2.);
\draw (4.5, 1.85) node {$\upalpha _ 3$};
\draw[->] (5.1, 2.) -- (5.5, 2.);
\draw (5.5, 2.) -- (5.9, 2.);
\draw (5.5, 1.85) node {$\upalpha _ 3$};
\draw[->] (6.1, 2.) -- (6.5, 2.);
\draw (6.5, 2.) -- (6.9, 2.);
\draw (6.5, 1.85) node {$\upalpha _ 2$};
\draw[->] (2.1, 1.) -- (2.5, 1.);
\draw (2.5, 1.) -- (2.9, 1.);
\draw (2.5, 0.85) node {$\upalpha _ 3$};
\draw[->] (3.1, 1.) -- (3.5, 1.);
\draw (3.5, 1.) -- (3.9, 1.);
\draw (3.5, 0.85) node {$\upalpha _ 2$};
\draw[->] (4.1, 1.) -- (4.5, 1.);
\draw (4.5, 1.) -- (4.9, 1.);
\draw (4.5, 0.85) node {$\upalpha _ 3$};
\draw[->] (2., 0.1) -- (2., 0.5);
\draw (2., 0.5) -- (2., 0.9);
\draw (1.85, 0.5) node {$\upalpha _ 4$};
\draw[->] (3., 0.1) -- (3., 0.5);
\draw (3., 0.5) -- (3., 0.9);
\draw (2.85, 0.5) node {$\upalpha _ 4$};
\draw[->] (4., 0.1) -- (4., 0.5);
\draw (4., 0.5) -- (4., 0.9);
\draw (3.85, 0.5) node {$\upalpha _ 4$};
\draw[->] (3., 1.1) -- (3., 1.5);
\draw (3., 1.5) -- (3., 1.9);
\draw (2.85, 1.5) node {$\upalpha _ 4$};
\draw[->] (4., 1.1) -- (4., 1.5);
\draw (4., 1.5) -- (4., 1.9);
\draw (3.85, 1.5) node {$\upalpha _ 4$};
\draw[->] (5., 1.1) -- (5., 1.5);
\draw (5., 1.5) -- (5., 1.9);
\draw (4.85, 1.5) node {$\upalpha _ 4$};
\draw (0, -0.25) node {$\upalpha _ 1$};
\end{tikzpicture}
\caption{$\Updelta_j^1$ for $\Upsigma \cong \mm{F}_4$ and $j  = 1$.}
\label{figure:f4:j:1}
\end{figure}
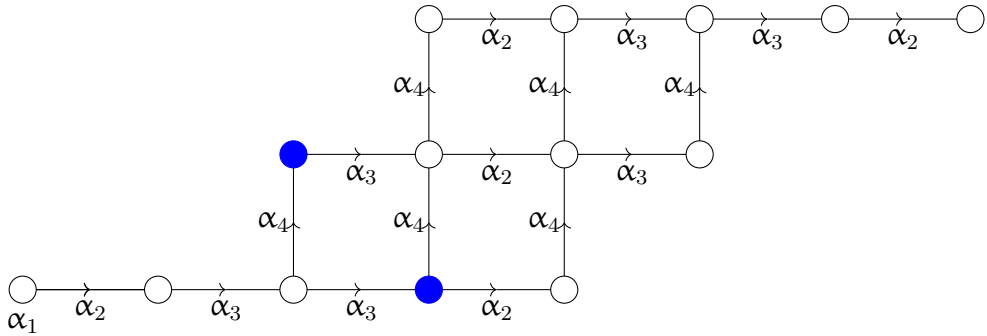

Now suppose that $j = 4$. According to \cite[Prop.\hspace{2pt}5.5]{sanmartin-strings}, both combinations
\begin{equation*}
\upalpha_2 + 2\upalpha_3 + \upalpha_4 \quad \text{and} \quad \upalpha_1 + \upalpha_2 + \upalpha_3 + \upalpha_4
\end{equation*}
are roots, hence they are elements of $\Updelta_4^1$ of the same height, and we once again arrive at a contradiction. Again, for the sake of completeness, we include the diagram of $\Updelta_4^1$ in Figure \ref{figure:f4:j:r}, where these roots correspond to the blue nodes.
\begin{figure}[h]
\begin{tikzpicture}[scale=1.8]
\draw (0, 0) circle (0.1);
\draw (1, 0) circle (0.1);
\draw (2, 0) circle (0.1);
\draw (3, 0) circle (0.1);
\filldraw[color = blue] (2.5, 0.5) circle (0.1);
\filldraw[color = blue] (2.5, -0.5) circle (0.1);
\draw (4, 0) circle (0.1);
\draw (4, 0) circle (0.1);
\draw (5, 0) circle (0.1);
\draw[->] (0.1, 0.) -- (0.5, 0.);
\draw (0.5, 0.) -- (0.9, 0.);
\draw (0.5, -0.15) node {$\upalpha _3$};
\draw[->] (1.1, 0.) -- (1.5, 0.);
\draw (1.5, 0.) -- (1.9, 0.);
\draw (1.5, -0.15) node {$\upalpha _2$};
\draw[->] (2.57071, -0.429289) -- (2.75, -0.25);
\draw (2.75, -0.25) -- (2.92929, -0.0707107);
\draw (2.9, -0.33) node {$\upalpha _3$};
\draw[->] (2.07071, 0.0707107) -- (2.25, 0.25);
\draw (2.25, 0.25) -- (2.42929, 0.429289);
\draw (2.1, 0.35) node {$\upalpha _3$};
\draw[->] (3.1, 0.) -- (3.5, 0.);
\draw (3.5, 0.) -- (3.9, 0.);
\draw (3.5, -0.15) node {$\upalpha _2$};
\draw[->] (4.1, 0.) -- (4.5, 0.);
\draw (4.5, 0.) -- (4.9, 0.);
\draw (4.5, -0.15) node {$\upalpha _3$};
\draw[->] (2.07071, -0.0707107) -- (2.25, -0.25);
\draw (2.25, -0.25) -- (2.42929, -0.429289);
\draw (2.1, -0.33) node {$\upalpha _1$};
\draw[->] (2.57071, 0.429289) -- (2.75, 0.25);
\draw (2.75, 0.25) -- (2.92929, 0.0707107);
\draw (2.9, 0.35) node {$\upalpha _1$};
\draw (0, -0.25) node {$\upalpha _4$};
\end{tikzpicture}
\caption{$\Updelta_j^1$ for $\Upsigma \cong \mm{F}_4$ and $j  = 4$.}
\label{figure:f4:j:r}
\end{figure}
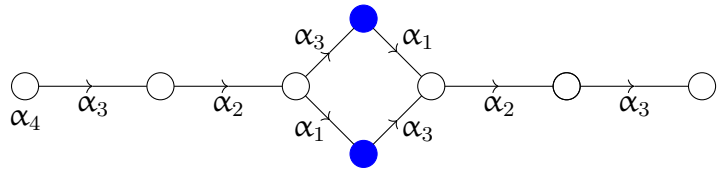

The above discussion reveals that, for the spaces of types $\mm{D}, \mm{E},$ and $\mm{F}$, the assumption that the action of $H_{j,\mk{w}}$ does not arise via nontrivial canonical extension is too restrictive and it rules out all choices of $\mk{w}$. We conclude:

\begin{prop}\label{prop:NC_DEF}
Any nilpotent construction action on a symmetric space $M$ of noncompact type with root system of type $\mm{D}_r \, (r \geq 4), \mm{E}_6, \mm{E}_7, \mm{E}_8,$ or $\mm{F}_4$ can be obtained via canonical extension from a proper boundary component of $M$.
\end{prop}

This leaves us with the spaces whose root systems are of type $\mm{A}_r$, $\mm{B}_r$, $\mm{C}_{r+1}$, or $\mm{BC}_r$, with $r \geq 2$. For such spaces, arguing on the level of roots is not going to be sufficient, and we have to ascend to the level of root spaces. We are going to investigate the shape operators of the singular orbit $S = H_{j,\mk{w}} \cdot o$. By virtue of Corollary \ref{cor:three_normalizers_RRSD}\eqref{cor:three_normalizers_RRSD:a}, the Lie algebra $\mk{h}_{j,\mk{w}} = N_{\mk{m}_j}(\mk{w}) \oplus \mk{a}_j \oplus \mk{n}_{j,\mk{w}}$ of our action has a subalgebra $\mk{h}$ of the form
\begin{equation}\label{equation:tangent:orbit:an}
\mk{h} = \mk{a} \oplus \mk{n}^j \oplus \mk{w} \oplus \Bigl( \bigoplus_{\upnu \geq 2} \mk{n}_j^\upnu \Bigr) = \mk{a} \oplus \Bigl( \bigoplus_{\upalpha \in \Upsigma_j^+} \mk{g}_\upalpha \Bigr) \oplus \Bigl(  \bigoplus_{k=1}^m \mk{w}_{\uplambda_k} \Bigr)\oplus \Bigl(  \bigoplus_{\upgamma \in \bigcup_{\upnu\geq 2} \Updelta_j^\upnu} \mk{g}_\upgamma \Bigr).
\end{equation}
This subalgebra has the same projection in $\mk{p}$ as $\mk{h}_{j,\mk{w}}$, which means that its corresponding connected Lie subgroup $H$ also has $S$ as its orbit through $o$. But $H$ is a subgroup of $AN$, which gives us access to the tools introduced in Subsection \ref{sec:preliminaries:extrinsic_geometry_cpc} to study the geometry of~$S$. In what follows, we will embrace the notation established in Subsection \ref{sec:preliminaries:extrinsic_geometry_cpc}. We also extend the operator $\top$ to the whole $\mk{g}$ by defining it to be the orthogonal projection onto $\mk{h}$ with respect to $B_\uptheta$. Thanks to \eqref{equation:inner:relation:b:theta:an} and \eqref{equation:tangent:orbit:an}, this definition agrees with the old one (which was the orthogonal projection onto $\mk{h}$ with respect to $\cross{-}{-}_{AN}$) on $\mk{a} \oplus \mk{n}$. Our first goal is to show that the general formula \eqref{AN_shape_operators_general} for the shape operators becomes much simpler in the context of the nilpotent construction.

\begin{prop}\label{proposition:shape:orbit:an}
Take $\uplambda \in \Updelta_j^1$ and $\upalpha \in \Upsigma^+$ and pick any vectors $\upxi \in \mk{v}_\uplambda$ and $X \in \mk{g}_\upalpha \cap \mk{h}$. Then the shape operator $\mc{A}_\upxi$ satisfies the following:
\begin{enumerate}[\normalfont (a)]
\item\customlabel{proposition:shape:orbit:an:a}{a} $\mathcal{A}_\upxi H = 0$ for any $H \in \mk{a}$.
\item\customlabel{proposition:shape:orbit:an:b}{b} $\mathcal{A}_\upxi X = \dfrac{1}{2} ([\upxi, X] - [\uptheta \upxi, X])^\top$.
\item\customlabel{proposition:shape:orbit:an:c}{c} $\mathcal{A}_\upxi X = \dfrac{1}{2} [\upxi, X]^\top$ if $\upalpha \in \Upsigma_j^+$.
\item\customlabel{proposition:shape:orbit:an:d}{d} $\mathcal{A}_\upxi X = 0$ if $\upalpha \in \Upsigma_j^+$ and $\uplambda = \uplambda_m$.
\end{enumerate}
\end{prop}

\begin{proof}
We begin with \eqref{proposition:shape:orbit:an:a}. From \eqref{AN_shape_operators_general} and the fact that $\uptheta(\mk{g}_\uplambda) = \mk{g}_{-\uplambda}$, we have, for any $Y \in \mk{h}$:
\begin{equation}\label{shape_operator_a}
\cross{\mc{A}_\upxi H}{Y}_{AN} = \frac{1}{4}\cross{[\upxi,H] - [\uptheta\upxi,H]}{Y}_{B_\uptheta} = -\frac{1}{4}\uplambda(H)\cross{\upxi}{Y}_{B_\uptheta} - \frac{1}{4}\uplambda(H)\cross{\uptheta \upxi}{Y}_{B_\uptheta}.
\end{equation}
The second term on the right vanishes because $\uptheta \upxi$ lies in $\mk{g}_{-\uplambda}$, which is orthogonal to $\mk{a} \oplus \mk{n}$ and thus to $Y$. Since $\upxi \in \mk{n}$, we can use the comment after \eqref{equation:inner:relation:b:theta:an} to rewrite the first term on the right-hand side of \eqref{shape_operator_a} as $-\frac{1}{2}\uplambda(H)\cross{\upxi}{Y}_{AN}$, which vanishes because $\mk{v}$ and $\mk{h}$ are orthogonal with respect to $\cross{-}{-}_{AN}$. All in all, this shows that $\mc{A}_\upxi H$ must be zero. To show \eqref{proposition:shape:orbit:an:b}, we first claim that either $[\uptheta \upxi, X]$ lies in $\mk{n}$, or else it is orthogonal to $Y$ with respect to $B_\uptheta$. Since $[\uptheta \upxi, X] \in \mk{g}_{\upalpha-\uplambda}$, we need only show that $[\uptheta \upxi, X] \perp \mk{a}$ whenever $\upalpha = \uplambda$. This follows directly from Lemma \ref{lemma:a:k0}. Since $[\upxi,X] \in \mk{n}$, we can use the same discussion after \eqref{equation:inner:relation:b:theta:an} to rewrite \eqref{AN_shape_operators_general} in our case as:
$$
\cross{\mc{A}_\upxi X}{Y}_{AN} = \frac{1}{2}\cross{[\upxi,X] - [\uptheta\upxi,X]}{Y}_{AN}
$$
As $Y$ is arbitrary, this implies \eqref{proposition:shape:orbit:an:b}. For \eqref{proposition:shape:orbit:an:c}, just note that $\upalpha-\uplambda = -\upalpha_j + \sum_{i \ne j} n_i \upalpha_i$ with $n_i \in \Z$---this means that $\upalpha-\uplambda$ is either a negative root or not a root at all. In either case, $[\uptheta \upxi, X]$ has trivial projection in $\mk{h}$. The rest follows from \eqref{proposition:shape:orbit:an:b}. Finally, we prove \eqref{proposition:shape:orbit:an:d}. If $\uplambda_m + \upalpha$ were a root, it would be an element of $\Updelta_j^1$ of height greater than $m$, which would contradict Corollary \ref{corollary:bifurcation}. The assertion then follows from \eqref{proposition:shape:orbit:an:c}.
\end{proof}

Our next step is to utilize the protohomogeneity condition to show that part \eqref{proposition:shape:orbit:an:d} of the above proposition must in fact hold for any normal vector.

\begin{prop}\label{proposition:all:shapes:nj}
For any $\upxi \in \mk{v}$ and $X \in \mk{n}^j$, we have $\mathcal{A}_\xi X = 0$.
\end{prop}

\begin{proof}
This proposition is one of the few places where it is actually more convenient to bring our perspective from $\mk{a} \oplus \mk{n}$ back to $\mk{p} \cong T_oM$. Recall that the identification $\upvarphi \colon AN \isoto M$ is by definition an isometry with respect to the metric $\cross{-}{-}_{AN}$, and its differential at the identity can be identified with the projection of $\mk{a} \oplus \mk{n}$ onto $\mk{p}$ along $\mk{k}$, i.e., $d\upvarphi_e = \restr{(\Id_\mk{g}\hspace{-1pt} -\hspace{1.5pt} \uptheta)/2}{\mk{a} \oplus \mk{n}}$. Given a vector $X \in \mk{g}$ (resp., a subspace $V \subseteq \mk{g}$), we are going to write $X_\mk{p}$ (resp., $V_\mk{p}$) for its projection in $\mk{p}$ along $\mk{k}$. Since $\upvarphi(H) = S$, we have $\mk{h}_\mk{p} = T_oS$ and $\mk{v}_\mk{p} = N_oS$. Note that $\mk{v}_\mk{p}$ splits as $\bigoplus_{k=1}^m \mk{p}_{\uplambda_k}$. We claim that the sum $\mk{a} \oplus \mk{n}^j_\mk{p}$ is invariant under the representation of $K_j$ on $\mk{p}$. Indeed, note that $\mk{a} \oplus \mk{n}^j_\mk{p} = (\mk{a}^j \oplus \mk{n}^j_\mk{p}) \oplus \mk{a}_j = \mk{b}_j \oplus \mk{a}_j = \mk{f}_j$. The group $K_j$ is the isotropy of the point $o$ under the action of $L_j$, hence $K_j$ preserves $T_o(L_j \cdot o) = T_o F_j \cong \mk{f}_j$. Given a normal vector $v \in \mk{v}_\mk{p}$, let us denote the corresponding shape operator of $S$ at $o$ by $\mc{A}^\mk{p}_v$. Since $\upvarphi$ is an isometry that takes $H$ to $S$, it identifies the shape operators of $H$ at $e$ and of $S$ at $o$: for any $\upxi \in \mk{v}$ and $X \in \mk{h}$, we have $(\mc{A}_\upxi X)_\mk{p} = \mc{A}^\mk{p}_{\upxi_\mk{p}} X_\mk{p}$. Finally, note that the representation of $N_{K_j}(\mk{v})$ on $\mk{v}_\mk{p}$ is the slice representation of $H_{j,\mk{w}}$ at $o$, hence it is transitive on the unit sphere in $\mk{v}_\mk{p}$. This representation is isomorphic to $N_{K_j}(\mk{v}) \curvearrowright \mk{v}$ by means of the projection $\mk{v} \isoto \mk{v}_\mk{p}$. Indeed, the projection $(\Id_\mk{g}\hspace{-1pt} -\hspace{1.5pt} \uptheta)/2 \colon \mk{g} \twoheadrightarrow \mk{p}$ is $K$-equivariant, hence $\mk{v} \isoto \mk{v}_\mk{p}$ is $N_{K_j}(\mk{v})$-equivariant.

Now we just put everything together. Take $\upxi \in \mk{v}$ and $X \in \mk{n}^j$. We need to show that $\mc{A}_\upxi X = 0$, or equivalently, that $\mc{A}^\mk{p}_{\upxi_\mk{p}} X_\mk{p} = 0$. According to the above discussion, there exists $k \in N_{K_j}(\mk{v})$ such that $\Ad(k)\upxi_\mk{p} = (\Ad(k)\upxi)_\mk{p} \in \mk{p}_{\uplambda_m}$. The vector $\Ad(k)\upxi$ then lies in $\mk{g}_{\uplambda_m}$, hence the shape operator $\mc{A}_{\Ad(k)\upxi}$ vanishes on the whole $\mk{a} \oplus \mk{n}^j$---thanks to parts \eqref{proposition:shape:orbit:an:a} and \eqref{proposition:shape:orbit:an:d} of Proposition \ref{proposition:shape:orbit:an}. In other words, the operator $\mc{A}^\mk{p}_{(\Ad(k)\upxi)_\mk{p}} = \mc{A}^\mk{p}_{\Ad(k)\upxi_\mk{p}}$ vanishes on the whole $\mk{a} \oplus \mk{n}^j_\mk{p}$. In particular, it vanishes on $\Ad(k)X_\mk{p}$, which belongs to $\mk{a} \oplus \mk{n}^j_\mk{p}$ because $X_\mk{p} \in \mk{n}^j_\mk{p}$ and $\Ad(k)$ preserves $\mk{a} \oplus \mk{n}^j_\mk{p}$. But now simply notice that $0 = \mc{A}^\mk{p}_{\Ad(k)\upxi_\mk{p}} (\Ad(k)X_\mk{p}) = \mc{A}^\mk{p}_{\upxi_\mk{p}} X_\mk{p}$, as $\Ad(k)$ is an element of the isotropy of the action at $o$ and thus preserves the shape operators of $S$ at $o$. This shows that $\mc{A}_\upxi X = 0$, which completes the proof.
\end{proof}

This result allows us to improve drastically on Corollary \ref{cor:three_normalizers_RRSD}:

\begin{cor}\label{corollary:new:normalizers}
For every $\upalpha \in \Upsigma_j$, we have $\mk{g}'_\upalpha = \mk{g}_\upalpha$. Moreover,
$$
N_{\mk{m}_j}(\mk{w}) = N_{\mk{m}_j}(\mk{v}) = N_{\mk{k}_0}(\mk{w}) \oplus \mk{a}^j \oplus \bigoplus_{\upalpha \in \Upsigma_j} \mk{g}_{\upalpha}.
$$
In particular, this normalizer contains the isometry Lie algebra $\mk{g}'_j$ of $B_j$.
\end{cor}

\begin{proof}
Pick some $\upalpha \in \Upsigma_j^+$ and take any $X \in \mk{g}_\upalpha$ and $\upxi \in \mk{v}$. We know from Propositions \ref{proposition:shape:orbit:an}\eqref{proposition:shape:orbit:an:c} and \ref{proposition:all:shapes:nj} that $[X,\upxi]$ has trivial projection in $\mk{h}$, hence it belongs to $\mk{v}$. This shows that $X$ normalizes $\mk{v}$ and thus lies in $\mk{g}'_\upalpha$. If $\upbeta \in \Upsigma_j^-$, we have $\mk{g}'_\upbeta = \uptheta(\mk{g}'_{-\upbeta}) = \uptheta(\mk{g}_{-\upbeta}) = \mk{g}_\upbeta$. The formula for the normalizers then follows directly from Corollary \ref{cor:three_normalizers_RRSD}. The last assertion follows from the fact that $\mk{g}'_j$ is generated by $\mk{g}_{\upalpha}$ as $\upalpha$ runs through $\Upsigma_j$.
\end{proof}

For our next step, we deal with the extreme case of $\mk{w}$ being zero. As it will turn out, this is actually the only case leading to C1-actions that cannot be obtained via other methods.

\begin{prop}\label{proposition:whole:n1:totally:gedesic}
Suppose that for some $1 \le k \le m$, we have $\mk{w}_{\uplambda_k} = \set{0}$. Then $\mk{w} = \set{0}$. Moreover, in this situation, the singular orbit of $H_{j,\mk{w}}$ is totally geodesic, except for the case when $\Upsigma \cong \mm{G}_2$ and $\upalpha_j$ is the short simple root.
\end{prop}

Note that we are not saying that $\mk{w} = \set{0}$ always leads to a C1-action: it is always trivially admissible but may not be protohomogeneous, as we saw in the case of $\mm{DEF}$-type spaces. What we are saying is, when it is protohomogeneous, the singular orbit of the corresponding C1-action must be totally geodesic, barring one exception. 

\begin{proof}[Proof of Proposition {\normalfont \ref{proposition:whole:n1:totally:gedesic}}]
We begin with the first assertion. Suppose that $\mk{w}_{\uplambda_k} = \set{0}$ for some $k$ as above. This means that $\mk{v}_{\uplambda_k} = \mk{g}_{\uplambda_k}$. Pick any $1 \le l \le m$ other than $k$. We must have $\uplambda_l - \uplambda_k \in \Upsigma_j$ by virtue of Corollary \ref{corollary:bifurcation}. But now Corollary \ref{corollary:new:normalizers} tells us that $\mk{g}_{\uplambda_l - \uplambda_k}$ normalizes $\mk{v}$. Altogether, we have
$$
\mk{g}_{\uplambda_l} = [\mk{g}_{\uplambda_l - \uplambda_k},\mk{g}_{\uplambda_k}] = [\mk{g}_{\uplambda_l - \uplambda_k},\mk{v}_{\uplambda_k}] \subseteq \mk{v}_{\uplambda_l},
$$
hence $\mk{v}_{\uplambda_l} = \mk{g}_{\uplambda_l}$. As $l$ was arbitrary, we obtain $\mk{v} = \mk{n}_j^1, \, \mk{w} = \set{0}$. Now we come to the second part. Let $\updelta = \sum_{i=1}^r n_i \upalpha_i \in \Upsigma^+$ be the highest root. By looking at \cite[pp.\hspace{2pt}336--340]{submanifolds_&_holonomy}, we see that $n_i \le 3$ whenever $\upalpha_i$ has only one neighbor in the Dynkin diagram (which $\upalpha_j$ does by Corollary \ref{corollary:alphaj:corner}), and the equality is only attained when $\Upsigma \simeq \mm{G}_2$ and $\upalpha_i$ is the short simple root, or when $\Upsigma \simeq \mm{E}_8$ and $\upalpha_i$ is one particular boundary root (among the three). In regard to $\upalpha_j$, assume first that $n_j \le 2$. Then $\Updelta_j^\upnu$ is empty whenever $\upnu \ge 3$ and thus $\mk{n}_j = \mk{n}_j^1 \oplus \mk{n}_j^2$. Since $\mk{w} = \set{0}$, decomposition \eqref{equation:tangent:orbit:an} now reads as $\mk{h} = \mk{a} \oplus \mk{n}^j \oplus \mk{n}_j^2$. Take a normal vector $\upxi \in \mk{v}_{\uplambda_k} = \mk{g}_{\uplambda_k}$ for some $1 \le k \le m$. We want to show that the shape operator $\mc{A}_\upxi$ is identically zero. We already know that it vanishes on $\mk{a} \oplus \mk{n}^j$---thanks to Propositions \ref{proposition:shape:orbit:an}\eqref{proposition:shape:orbit:an:a} and \ref{proposition:all:shapes:nj}---so we only need to show that it vanishes on $\mk{n}_j^2$ as well. Pick any $\upgamma \in \Updelta_j^2$ and take some $X \in \mk{g}_\upgamma$. By definition, $\uplambda_k = \upalpha_j + \sum_{i \ne j} n_i \upalpha_i$ and $\upgamma = 2\upalpha_j + \sum_{i \ne j} m_i \upalpha_i$. Therefore, $\upgamma + \uplambda_k$ is not a root, for otherwise $\Updelta_j^3$ would be nonempty. What is more, $\upgamma - \uplambda_k$ either lies in $\Updelta_j^1$ or is not a root at all. We deduce that $[\upxi, X] = 0$ and $[\uptheta \upxi, X] \in \mk{n}_j^1 = \mk{v}$, and the latter implies $[\uptheta \upxi, X]^\top = 0$. By plugging this into Proposition \ref{proposition:shape:orbit:an}\eqref{proposition:shape:orbit:an:b}, we get $\mathcal{A}_\upxi X = 0$. We conclude that $\mc{A}_\upxi$ vanishes on $\mk{n}_j^2$ and thus on the whole $\mk{h}$. According to our discussion at the end of Subsection \ref{sec:preliminaries:extrinsic_geometry_cpc}, singular orbits of C1-actions are CPC, which implies that all of the shape operators of $H_{j,\mk{w}} \cdot o$ are zero and thus this orbit is totally geodesic.

Finally, let us look at the case when $n_j = 3$. If $\Upsigma \simeq E_8$, then $\mk{w} = \{0\}$ is not protohomogeneous, as we witnessed in the proof of Proposition \ref{prop:NC_DEF}. So we may assume $\Upsigma \simeq \mm{G}_2$ and $\upalpha_j = \upalpha_2$ is the short simple root. There are two spaces with this root system: $\mm{G}_2^2/\SO(4)$ and $\mm{G}_2(\C)/\mm{G}_2$. As was shown in \cite{berndt_tamaru_cohomogeneity_one}, for either of them, the subspace $\mk{w} = \set{0}$ of $\mk{n}_2^1$ is admissible and protohomogeneous, and the corresponding C1-action of $H_{2,0}$ has a non-totally-geodesic singular orbit---because it does not figure in the classification in \cite{berndt_tamaru_totally_geodesic_singular_orbit}. One can also verify that this orbit is not totally geodesic directly. Indeed, we have $\upalpha_1 + 2\upalpha_2 \in \Updelta_2^2$ and $\upalpha_1 + 3\upalpha_2 \in \Updelta_2^3$. For any nonzero $\upxi \in \mk{v}_{\upalpha_2} = \mk{g}_{\upalpha_2}$, the restriction of $\mc{A}_\upxi$ to $\mk{g}_{\upalpha_1 + 3\upalpha_2}$ is given simply by $-\frac{1}{2}\ad(\uptheta \upxi) \colon \mk{g}_{\upalpha_1 + 3\upalpha_2} \to \mk{g}_{\upalpha_1 + 2\upalpha_2}$, as follows from Proposition \ref{proposition:shape:orbit:an}\eqref{proposition:shape:orbit:an:b}. Since $\upalpha_1 + 3\upalpha_2$ is the first root in its $(-\upalpha_2)$-string, this map in nonzero by virtue of Lemma \ref{lem:string_injective}.
\end{proof}

Since the classification of C1-actions with a totally geodesic singular orbit on irreducible symmetric spaces of noncompact type is complete, we can disregard the case $\mk{w} = \set{0}$. Thanks to the above proposition, from now on, \textit{we can and will assume that $\mk{w}_{\uplambda_k} \ne \set{0}$ for every $k=1,\ldots,m$}. This enables us to obtain the following result:

\begin{lem}\label{lemma:dimension:root:Lambda:Delta1}
For any $\upalpha_i \in \Uplambda_j$, there exists a root $\uplambda_k \in \Updelta_j^1$ such that $\dim \mk{g}_{\uplambda_k} \ge 2 \dim \mk{g}_{\upalpha_i}$.
\end{lem}

\begin{proof}
Take any $\upalpha = \upalpha_i \in \Uplambda_j$. Put $l$ to be the largest number among $\{1, \dots, m\}$ such that $\bilin{H^i}{\uplambda_l} = 0$, and write $\uplambda = \uplambda_l$ for simplicity. As $\uplambda_1 = \upalpha_j$ and $\bilin{H^i}{\upalpha_j} = 0$, $l$ is well-defined. Moreover, according to Proposition \ref{prop:CE_trick}, $l$ cannot be equal to $m$. We must have $\bilin{H^i}{\uplambda_{l+1}} = 1$, which means $\uplambda_{l+1} = \uplambda + \upalpha$. At the same time, $\uplambda - \upalpha$ is of the form $\upalpha_j - \upalpha_i + \sum_{k \ne i,j} n_k \upalpha_k$, hence it cannot be a root. But then $\upalpha - \uplambda$ cannot be a root either. Therefore, $\upalpha$ is the root of minimum height in its $\uplambda$-string, and that string also contains $\upalpha + \uplambda$. According to Lemma \ref{lem:string_injective}, for any nonzero $X \in \mk{g}_{\uplambda}$, the map $\restr{\ad(X)}{\mk{g}_{\upalpha}} \colon \mk{g}_{\upalpha} \to \mk{g}_{\upalpha + \uplambda}$ is injective. It follows from Corollary \ref{corollary:bifurcation} and our assumption above that each of the two summands in $\mk{g}_{\uplambda} = \mk{w}_{\uplambda} \oplus \mk{v}_{\uplambda}$ is nontrivial. Take any nonzero vectors $X \in \mk{w}_{\uplambda}$ and $\upxi \in \mk{v}_{\uplambda}$. In view of Corollary \ref{corollary:new:normalizers}, $\mk{g}_{\upalpha}$ normalizes both $\mk{w}$ and $\mk{v}$, which leads to
$$
\ad(X) \mk{g}_{\upalpha} \subseteq \mk{w}_{\upalpha + \uplambda} \quad \text{and} \quad  \ad(\upxi) \mk{g}_{\upalpha} \subseteq \mk{v}_{\upalpha + \uplambda}.
$$
We deduce that $\dim \mk{w}_{\upalpha + \uplambda} \geq \dim \mk{g}_{\upalpha}$ and $\dim \mk{v}_{\upalpha + \uplambda} \geq \dim \mk{g}_{\upalpha}$, which implies $\dim \mk{g}_{\upalpha + \uplambda} \ge 2 \dim \mk{g}_{\upalpha}$, so we can take $\upalpha + \uplambda = \uplambda_{l+1}$ as our desired root.
\end{proof}

The power of this result lies in the fact that it will allow us to exclude many candidates for $\upalpha_j$ based solely on root multiplicities. To begin with:

\begin{prop}\label{prop:NC_A}
Any nilpotent construction action on a symmetric space $M$ of noncompact type with root system of type $\mm{A}_r, \, r \ge 2,$ either has a totally geodesic singular orbit or else can be obtained via canonical extension from a proper boundary component of $M$.
\end{prop}

Before we proceed, let us make a quick overview. In the proof of this proposition (as well as of Propositions \ref{prop:NC_C} and \ref{prop:NC_B_BC}), we will assume that $\mk{w}$ is nonzero (justified by Proposition~\ref{proposition:whole:n1:totally:gedesic}) and the action of $H_{j,\mk{w}}$ does not come from nontrivial canonical extension. These assumptions (either one of them or both) are required by some of the results we have established in this section (notably, by Corollaries \ref{corollary:bifurcation} and \ref{corollary:alphaj:corner} and Lemma \ref{lemma:dimension:root:Lambda:Delta1}). We will then use those results to show that either no such $\mk{w}$ exists or the singular orbit of $H_{j,\mk{w}}$ must be totally geodesic.

\begin{proof}[Proof of Proposition {\normalfont \ref{prop:NC_A}}]
Consider the Dynkin diagram of $\Upsigma \simeq \mm{A}_r$:
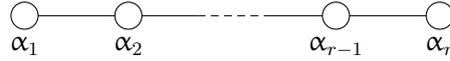
\begin{center}
\begin{dynkinDiagram}[arrow shape/.style={-{Bourbaki[length=8pt]}}, labels={\upalpha_1,\upalpha_2,\upalpha_{r-1},\upalpha_r}, text style/.style={scale=0.9}, scale=3, root radius=.06cm]A{}
\end{dynkinDiagram}
\end{center}
Thanks to Corollary \ref{corollary:alphaj:corner}, we need only deal with $j = 1$ or $r$. The two cases are essentially identical due to the symmetry of the Dynkin diagram, so we assume $j = 1$. According to Lemma \ref{lemma:dimension:root:Lambda:Delta1}, there must exist a root in $\Updelta_1^1$ whose multiplicity is at least twice that of $\upalpha_2$. But this is impossible since all roots in $\Upsigma$ have the same length and thus multiplicity (see, e.g., \cite[Cor.\hspace{2pt}6.3.2]{solonenko_thesis}).
\end{proof}

\begin{rem}
    Even though we do not need this, a similar argument can be applied to any space of rank $\ge 2$ all of whose restricted roots have the same multiplicity (and any choices of $\upalpha_j$ and nonzero $\mk{w}$ therein). This includes all spaces of types $\mm{A}, \mm{D}, \mm{E},$ and $\mm{G}_2$, all spaces whose isometry Lie algebra is a split real form, as well as the noncompact duals of all compact simple Lie groups.
\end{rem}

Using similar arguments and a little more effort, we can also dispatch the case of $\mm{C}_r$:

\begin{prop}\label{prop:NC_C}
Any nilpotent construction action on a symmetric space $M$ of noncompact type with root system of type $\mm{C}_r, \, r \ge 3,$ either has a totally geodesic singular orbit or else can be obtained via canonical extension from a proper boundary component of $M$.
\end{prop}

\textit{Proof.} 
Consider the Dynkin diagram of $\Upsigma \simeq \mm{C}_r$:
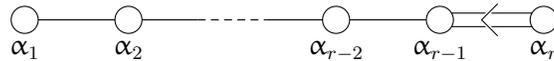
\begin{center}
\begin{dynkinDiagram}[arrow shape/.style={-{Bourbaki[length=8pt]}}, labels={\upalpha_1,\upalpha_2,\upalpha_{r-2},\upalpha_{r-1},\upalpha_r}, text style/.style={scale=0.9}, scale=3, root radius=.06cm]C{}
\end{dynkinDiagram}
\end{center}
Once again, according to Corollary \ref{corollary:alphaj:corner}, we only need to analyze the cases $j = 1$ and $j = r$, except this time we have to handle them separately as the Dynkin diagram has no nontrivial symmetries. First, put $j = 1$. According to \cite[Prop\hspace{2pt}4.6]{sanmartin-strings}, we have 
$$
\Updelta_1^1 = \{ \upalpha_1\} \cup \Bigl\{ \upalpha_1 + \sum_{i=2}^k \upalpha_i, \, \upalpha_1 + \sum_{i=2}^r \upalpha_i + \sum_{j=1}^l \upalpha_{r-j} \, \mid \,  \substack{2 \leq k \leq r, \\[1pt] 1 \leq l \leq  r-2} \Bigr\}.
$$
Moreover, all roots in $\Updelta_1^1$ have the same length and thus multiplicity. In Figure \ref{figure:string:c:c} we include the diagram of $\Updelta_1^1$ with $r = 5$ for illustrative purposes.\vspace{2pt}

\begin{figure}[h!]
\setlength{\abovecaptionskip}{5pt plus 0pt minus 2pt}
\begin{tikzpicture}[scale=1.8]
\draw (0, 0) circle (0.1);
\draw (1, 0) circle (0.1);
\draw (2, 0) circle (0.1);
\draw (3, 0) circle (0.1);
\draw (4, 0) circle (0.1);
\draw (5, 0) circle (0.1);
\draw (6, 0) circle (0.1);
\draw (7, 0) circle (0.1);
\draw[->] (0.1, 0.) -- (0.5, 0.);
\draw (0.5, 0.) -- (0.9, 0.);
\draw (0.5, -0.15) node {$\upalpha _2$};
\draw[->] (1.1, 0.) -- (1.5, 0.);
\draw (1.5, 0.) -- (1.9, 0.);
\draw (1.5, -0.15) node {$\upalpha _3$};
\draw[->] (2.1, 0.) -- (2.5, 0.);
\draw (2.5, 0.) -- (2.9, 0.);
\draw (2.5, -0.15) node {$\upalpha _4$};
\draw[->] (3.1, 0.) -- (3.5, 0.);
\draw (3.5, 0.) -- (3.9, 0.);
\draw (3.5, -0.15) node {$\upalpha _5$};
\draw[->] (4.1, 0.) -- (4.5, 0.);
\draw (4.5, 0.) -- (4.9, 0.);
\draw (4.5, -0.15) node {$\upalpha _4$};
\draw[->] (5.1, 0.) -- (5.5, 0.);
\draw (5.5, 0.) -- (5.9, 0.);
\draw (5.5, -0.15) node {$\upalpha _3$};
\draw[->] (6.1, 0.) -- (6.5, 0.);
\draw (6.5, 0.) -- (6.9, 0.);
\draw (6.5, -0.15) node {$\upalpha _2$};
\draw (0, -0.25) node {$\upalpha _1$};
\end{tikzpicture}
\caption{$\Updelta_j^1$ for $\Upsigma \cong \mm{C}_5$ and $j  = 1$.}
\label{figure:string:c:c}
\end{figure}

As $r \ge 3$, every $\uplambda \in \Updelta_1^1$ has the same length as $\upalpha_2$---and hence the same multiplicity, which contradicts Lemma \ref{lemma:dimension:root:Lambda:Delta1}. Now let us consider the case $j = r$. Using \cite[Prop.\hspace{2pt}4.2]{sanmartin-strings}, we have 
\begin{equation}
\Updelta_r^1 = \{ \upalpha_r \} \cup \Bigl\{ \upalpha_r + \sum_{i = 1}^k \upalpha_{r-i}, \, \upalpha_r + \sum_{i = 1}^k \upalpha_{r-i}  + \sum_{j = 1}^l \upalpha_{r-j} \, \mid \,  \substack{  1 \leq k \leq r-1, \\[1pt] \, 1 \leq l \leq k} \Bigr\}.
\end{equation}
Since $r \geq 3$, we see that
\begin{equation*}
\upalpha_r + 2 \upalpha_{r-1} \quad \text{and} \quad \upalpha_r +  \upalpha_{r-1} +\upalpha_{r-2}
\end{equation*}
are both roots in $\Updelta_r^1$. Moreover, they have the same height, which violates Corollary \ref{corollary:bifurcation}. In Figure \ref{figure:string:a:c}, we draw the diagram of $\Updelta_r^1$ for $r = 5$ with these two roots corresponding to the blue nodes. \qed

\begin{figure}[h]
\setlength{\abovecaptionskip}{5pt plus 3pt minus 2pt}
\begin{tikzpicture}[scale=1.8]
\draw (0, 0) circle (0.1);
\draw (1, 0) circle (0.1);
\filldraw[color = blue] (2, 0) circle (0.1);
\draw (3, 0) circle (0.1);
\filldraw[color = blue] (1, 1) circle (0.1);
\draw (2, 1) circle (0.1);
\draw (3, 1) circle (0.1);
\draw (2, 2) circle (0.1);
\draw (3, 2) circle (0.1);
\draw (3, 3) circle (0.1);
\draw (4, 0) circle (0.1);
\draw (4, 1) circle (0.1);
\draw (4, 2) circle (0.1);
\draw (4, 3) circle (0.1);
\draw (4, 4) circle (0.1);
\draw[->] (0.1, 0.) -- (0.5, 0.);
\draw (0.5, 0.) -- (0.9, 0.);
\draw (0.5, -0.15) node {$\upalpha _4$};
\draw[->] (1.1, 0.) -- (1.5, 0.);
\draw (1.5, 0.) -- (1.9, 0.);
\draw (1.5, -0.15) node {$\upalpha _3$};
\draw[->] (2.1, 0.) -- (2.5, 0.);
\draw (2.5, 0.) -- (2.9, 0.);
\draw (2.5, -0.15) node {$\upalpha _2$};
\draw[->] (3.1, 0.) -- (3.5, 0.);
\draw (3.5, 0.) -- (3.9, 0.);
\draw (3.5, -0.15) node {$\upalpha _1$};
\draw[->] (1., 0.1) -- (1., 0.5);
\draw (1., 0.5) -- (1., 0.9);
\draw (0.85, 0.5) node {$\upalpha _4$};
\draw[->] (2., 0.1) -- (2., 0.5);
\draw (2., 0.5) -- (2., 0.9);
\draw (1.85, 0.5) node {$\upalpha _4$};
\draw[->] (3., 0.1) -- (3., 0.5);
\draw (3., 0.5) -- (3., 0.9);
\draw (2.85, 0.5) node {$\upalpha _4$};
\draw[->] (4., 0.1) -- (4., 0.5);
\draw (4., 0.5) -- (4., 0.9);
\draw (3.85, 0.5) node {$\upalpha _4$};
\draw[->] (1.1, 1.) -- (1.5, 1.);
\draw (1.5, 1.) -- (1.9, 1.);
\draw (1.5, 0.85) node {$\upalpha _3$};
\draw[->] (2.1, 1.) -- (2.5, 1.);
\draw (2.5, 1.) -- (2.9, 1.);
\draw (2.5, 0.85) node {$\upalpha _2$};
\draw[->] (3.1, 1.) -- (3.5, 1.);
\draw (3.5, 1.) -- (3.9, 1.);
\draw (3.5, 0.85) node {$\upalpha _1$};
\draw[->] (2., 1.1) -- (2., 1.5);
\draw (2., 1.5) -- (2., 1.9);
\draw (1.85, 1.5) node {$\upalpha _3$};
\draw[->] (3., 1.1) -- (3., 1.5);
\draw (3., 1.5) -- (3., 1.9);
\draw (2.85, 1.5) node {$\upalpha _3$};
\draw[->] (2.1, 2.) -- (2.5, 2.);
\draw (2.5, 2.) -- (2.9, 2.);
\draw (2.5, 1.85) node {$\upalpha _2$};
\draw[->] (3., 2.1) -- (3., 2.5);
\draw (3., 2.5) -- (3., 2.9);
\draw (2.85, 2.5) node {$\upalpha _2$};
\draw[->] (3.1, 2.) -- (3.5, 2.);
\draw (3.5, 2.) -- (3.9, 2.);
\draw (3.5, 1.85) node {$\upalpha _1$};
\draw[->] (4., 1.1) -- (4., 1.5);
\draw (4., 1.5) -- (4., 1.9);
\draw (3.85, 1.5) node {$\upalpha _3$};
\draw[->] (4., 2.1) -- (4., 2.5);
\draw (4., 2.5) -- (4., 2.9);
\draw (3.85, 2.5) node {$\upalpha _2$};
\draw[->] (4., 3.1) -- (4., 3.5);
\draw (4., 3.5) -- (4., 3.9);
\draw (3.85, 3.5) node {$\upalpha _1$};
\draw[->] (3.1, 3.) -- (3.5, 3.);
\draw (3.5, 3.) -- (3.9, 3.);
\draw (3.5, 2.85) node {$\upalpha _1$};
\draw (0, -0.25) node {$\upalpha _5$};
\end{tikzpicture}
\caption{$\Updelta_j^1$ for $\Upsigma \cong \mm{C}_5$ and $j  = 5$.}
\label{figure:string:a:c}
\end{figure}
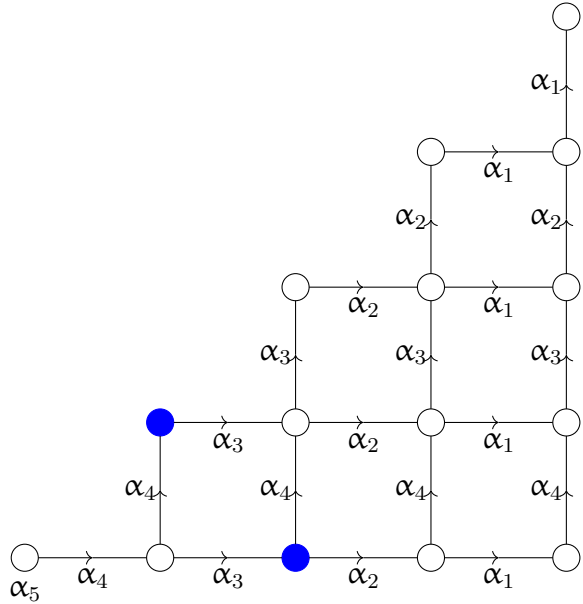

We are left to deal with the spaces of types $\mm{B}_r$ and $\mm{BC}_r$ with $r \ge 2$. As we are about to see, these cases are harder, as arguing merely by looking at the roots and their multiplicities is not going to cut it.

\begin{prop}\label{prop:NC_B_BC}
Any nilpotent construction action on a symmetric space $M$ of noncompact type with root system of type $\mm{B}_r$ or $\mm{BC}_r, \, r \ge 2$, either has a totally geodesic singular orbit or can be obtained via canonical extension from a proper boundary component of~$M$.
\end{prop}

\begin{proof}
We will analyze both cases simultaneously. Consider the Dynkin diagram of $\Upsigma$: 
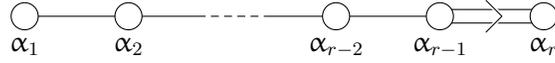
\begin{center}
\begin{dynkinDiagram}[arrow shape/.style={-{Bourbaki[length=8pt]}}, labels={\upalpha_1,\upalpha_2,\upalpha_{r-2},\upalpha_{r-1},\upalpha_r}, text style/.style={scale=0.9}, scale=3, root radius=.06cm]B{}
\end{dynkinDiagram}
\end{center}
Note that if $\Upsigma$ is of type $\mm{BC}_r$, then $2 \upalpha_r$ is also a root\footnote{For this reason, to distinguish the Dynkin diagrams of $\mm{B}_r$ and $\mm{BC}_r$, the latter is sometimes drawn with a double circle for $\upalpha_r$: \begin{dynkinDiagram}[edge length=0.77cm,arrow shape/.style={-{Bourbaki[length=4pt]}}, root radius=.06cm]B{oo..ooo}\dynkinRootMark{O}5\end{dynkinDiagram}\,. We are not using this notation here.}. Thanks to Corollary \ref{corollary:alphaj:corner}, we need only deal with $j = 1$ or $r$. We start with the case $j = 1$. From \cite[Prop.\hspace{2pt}4.5]{sanmartin-strings}, we have
$$
\Updelta_1^1 = \{ \upalpha_1 \} \cup \Bigl\{ \upalpha_1+ \sum_{i=2}^k \upalpha_i, \, \upalpha_1 +\sum_{i=2}^r \upalpha_i  + \sum_{j=0}^l \upalpha_{r-j} \, \mid \,  \substack{  2 \leq k \leq r, \\[1pt] \, 0 \leq l \leq r-2}  \Bigr\}.
$$
For illustrative purposes, we include the diagram of $\Updelta_1^1$ for $r = 5$ in Figure \ref{figure:b:b:string}.
\begin{figure}[h]
\setlength{\abovecaptionskip}{5pt plus 3pt minus 2pt}
\begin{tikzpicture}[scale=1.8]
\draw (0, 0) circle (0.1);
\draw (1, 0) circle (0.1);
\draw (2, 0) circle (0.1);
\draw (3, 0) circle (0.1);
\draw (4, 0) circle (0.1);
\draw (5, 0) circle (0.1);
\draw (6, 0) circle (0.1);
\draw (7, 0) circle (0.1);
\draw (8, 0) circle (0.1);
\draw[->] (0.1, 0.) -- (0.5, 0.);
\draw (0.5, 0.) -- (0.9, 0.);
\draw (0.5, -0.15) node {$\upalpha _2$};
\draw[->] (1.1, 0.) -- (1.5, 0.);
\draw (1.5, 0.) -- (1.9, 0.);
\draw (1.5, -0.15) node {$\upalpha _3$};
\draw[->] (2.1, 0.) -- (2.5, 0.);
\draw (2.5, 0.) -- (2.9, 0.);
\draw (2.5, -0.15) node {$\upalpha _4$};
\draw[->] (3.1, 0.) -- (3.5, 0.);
\draw (3.5, 0.) -- (3.9, 0.);
\draw (3.5, -0.15) node {$\upalpha _5$};
\draw[->] (4.1, 0.) -- (4.5, 0.);
\draw (4.5, 0.) -- (4.9, 0.);
\draw (4.5, -0.15) node {$\upalpha _5$};
\draw[->] (5.1, 0.) -- (5.5, 0.);
\draw (5.5, 0.) -- (5.9, 0.);
\draw (5.5, -0.15) node {$\upalpha _4$};
\draw[->] (6.1, 0.) -- (6.5, 0.);
\draw (6.5, 0.) -- (6.9, 0.);
\draw (6.5, -0.15) node {$\upalpha _3$};
\draw[->] (7.1, 0.) -- (7.5, 0.);
\draw (7.5, 0.) -- (7.9, 0.);
\draw (7.5, -0.15) node {$\upalpha _2$};
\draw (0, -0.25) node {$\upalpha _1$};
\end{tikzpicture}
\caption{$\Updelta_j^1$ when $\Upsigma \cong \mm{B}_5$ or $\mm{BC}_5$ and $j = 1$.}
\label{figure:b:b:string}
\end{figure}\vspace{-2ex}

Take the simple root $\upalpha_r \in \Uplambda_1$, which is the shortest in $\Uplambda$. As can be seen from \cite[pp.\hspace{2pt}336--340]{submanifolds_&_holonomy}, $\upalpha_r$ has the largest possible multiplicity among all the roots in $\Upsigma$. In particular, $\dim \mk{g}_{\upalpha_r} \geq \dim \mk{g}_{\uplambda}$ for any $\uplambda \in \Updelta_1^1$, which contradicts Lemma \ref{lemma:dimension:root:Lambda:Delta1}. 

We are left to consider the case $j = r$. In this situation, according to~\cite[Prop.\hspace{2pt}4.1]{sanmartin-strings}, we have
\begin{equation}\label{equation:b:bcr:j:r_roots}
\Updelta_r^1 = \Bigl\{ \sum_{i=0}^k \upalpha_{r-i} \, \mid \,  0 \leq k \leq r-1   \Bigl\},
\end{equation}
and all roots in $\Updelta_r^1$ have the same length and multiplicity. In Figure \ref{figure:a:b:string}, we draw the diagram of $\Updelta_r^1$ for $r = 5$.

\begin{figure}[h]
\setlength{\abovecaptionskip}{5pt plus 3pt minus 2pt}
\begin{tikzpicture}[scale=1.8]
\draw (0, 0) circle (0.1);
\draw (1, 0) circle (0.1);
\draw (2, 0) circle (0.1);
\draw (3, 0) circle (0.1);
\draw (4, 0) circle (0.1);
\draw[->] (0.1, 0.) -- (0.5, 0.);
\draw (0.5, 0.) -- (0.9, 0.);
\draw (0.5, -0.15) node {$\upalpha _4$};
\draw[->] (1.1, 0.) -- (1.5, 0.);
\draw (1.5, 0.) -- (1.9, 0.);
\draw (1.5, -0.15) node {$\upalpha _3$};
\draw[->] (2.1, 0.) -- (2.5, 0.);
\draw (2.5, 0.) -- (2.9, 0.);
\draw (2.5, -0.15) node {$\upalpha _2$};
\draw[->] (3.1, 0.) -- (3.5, 0.);
\draw (3.5, 0.) -- (3.9, 0.);
\draw (3.5, -0.15) node {$\upalpha _1$};
\draw (0, -0.25) node {$\upalpha _5$};
\end{tikzpicture}
\caption{$\Updelta_j^1$ when $\Upsigma \cong \mm{B}_5$ or $\mm{BC}_5$ and $j  = 5$.}
\label{figure:a:b:string}
\end{figure}

Recall that for every $\uplambda \in \Updelta_r^1$, both summands in the decomposition $\mk{g}_\uplambda = \mk{w}_\uplambda \oplus \mk{v}_\uplambda$ are nontrivial. In particular, we can take a normal vector $\upxi \in \mk{v}_{\upalpha_r}$ with $||\upxi||_{AN} = 1$. We will prove that the shape operator $\mc{A}_\upxi$ is zero. Since the orbit $S \cdot o$ is CPC, this will imply that it is totally geodesic.

Referring to \cite[pp.\hspace{2pt}336--340]{submanifolds_&_holonomy} one more time, we see that every $\upalpha = \sum_i n_i \upalpha_i \in \Upsigma^+$ has $n_i \le 2$ for all $i$. In particular, $n_r \le 2$, which means $\Updelta_r^\upnu = \varnothing$ for $\upnu \ge 3$ and thus $\Upsigma^+ = \Upsigma_r^+ \sqcup \Updelta_r^1 \sqcup \Updelta_r^2$. We want to show that $\mc{A}_\upxi X = 0$ for every $X \in \mk{h}$. In view of \eqref{equation:tangent:orbit:an} and Proposition \ref{proposition:shape:orbit:an}\eqref{proposition:shape:orbit:an:a}, we may assume that $X$ lies in a single root space $\mk{g}_\uplambda, \, \uplambda \in \Upsigma^+$. If the $\upalpha_r$-string of $\uplambda$ is trivial, then $\mathcal{A}_\upxi X = 0$ thanks to Proposition \ref{proposition:shape:orbit:an}\eqref{proposition:shape:orbit:an:b}. Otherwise, imagine that $\upbeta \in \Upsigma^+$ is the root of minimum height in its nontrivial $\upalpha_r$-string. Then $\upbeta+\upalpha_r$ is a root and hence either $\upbeta$ or $\upbeta+\upalpha_r$ lies $\Updelta_r^1$. Therefore, it suffices to consider only the $\upalpha_r$-strings of the roots in $\Updelta_r^1$. We will analyze the $\upalpha_r$-string of $\upalpha_r$ at the end. For now, let us consider $\uplambda \in \Updelta_r^1 \mysetminus \{ \upalpha_r \}$. Since $\upalpha_r$ is the shortest simple root, we have $n_{\upalpha_r, \upalpha_{r-1}} = -2$ and thus $n_{\upalpha_r, \uplambda} = 0$, as follows from \eqref{equation:b:bcr:j:r_roots}. We also know that $\uplambda-\upalpha_r$ belongs to $\Upsigma_r^+$, as guaranteed by Corollary \ref{corollary:bifurcation}. Using this and \cite[Prop.\hspace{2pt}2.48(g)]{knapp}, we arrive at the conclusion that the $\upalpha_r$-string of $\uplambda$ consists of the roots $\uplambda - \upalpha_r \in \Upsigma_r^+$, $\uplambda \in \Updelta_r^1$ and $\uplambda + \upalpha_r \in \Updelta_r^2$.

Thanks to Proposition \ref{proposition:all:shapes:nj}, we already know that $\mc{A}_\upxi$ vanishes on $\mk{g}_{\uplambda - \upalpha_r}$. Take an element $X_{\uplambda + \upalpha_r} \in \mk{g}_{\uplambda + \upalpha_r}$. According to \cite[Prop.\hspace{2pt}4.2(iv)]{berndt_SL_CPC}, we can write $X_{\uplambda + \upalpha_r} = \ad(\upxi)^2  X_{\uplambda -\upalpha_r}$ for some $X_{\uplambda - \upalpha_r} \in \mk{g}_{\uplambda - \upalpha_r}$. We calculate:
\begin{align*}
\mathcal{A}_\upxi X_{\uplambda + \upalpha_r} &= \mathcal{A}_\upxi (\ad(\upxi)^2 X_{\uplambda - \upalpha_r}) \\
&= \frac{1}{2} ([\upxi, \ad(\upxi)^2 X_{\uplambda - \upalpha_r}] - [\uptheta \upxi, \ad(\upxi)^2 X_{\uplambda - \upalpha_r}])^\top && \text{(Proposition \ref{proposition:shape:orbit:an}\eqref{proposition:shape:orbit:an:b})} \\
&= -\frac{1}{2} [\uptheta \upxi, [\upxi,[\upxi,X_{\uplambda - \upalpha_r}]]]^\top && \text{($\uplambda + 2\upalpha_r \not\in \Upsigma$)}\\
&= ||\upalpha_r||^2 [\upxi, X_{\uplambda - \upalpha_r}]^\top && \text{(\cite[Lem.\hspace{2pt}2.4(iii)]{berndt_SL_CPC})}\\
&= 2 ||\upalpha_r||^2 \mathcal{A}_\upxi X_{\uplambda - \upalpha_r} = 0 && \text{(Propositions \ref{proposition:shape:orbit:an}\eqref{proposition:shape:orbit:an:c} and \ref{proposition:all:shapes:nj})}.
\end{align*}
Finally, consider $X_\uplambda \in \mk{w}_\uplambda$. Corollary \ref{corollary:new:normalizers} ensures that $\mk{g}_{\uplambda- \upalpha_r}$ normalizes $\mk{v}$, which implies that $\ad(\upxi)\mk{g}_{\uplambda- \upalpha_r}$ is contained in $\mk{v}_\uplambda$ and is thus orthogonal to $X_\uplambda$. Using \cite[Prop.\hspace{2pt}4.2(v)]{berndt_SL_CPC}, we obtain $\nabla_X \upxi = 0$ and hence $\mc{A}_\upxi X_\uplambda = - (\nabla_X \upxi)^\top = 0$. In view of Proposition \ref{proposition:shape:orbit:an}\eqref{proposition:shape:orbit:an:b}, this means that both $\ad(\upxi)$ and $\ad(\uptheta\upxi)$ vanish on $\mk{w}_\uplambda$---we will need this observation below. The only root left in $\Upsigma^+$ with possibly nontrivial $\upalpha_r$-string is $\upalpha_r$. If $\Upsigma$ is of type $\mm{B}_r$, then $2\upalpha_r$ is not a root. In this case, Proposition \ref{proposition:shape:orbit:an}\eqref{proposition:shape:orbit:an:b} together with Lemma \ref{lemma:a:k0}\eqref{lemma:a:k0:2} ensure that $\mc{A}_\upxi X_{\upalpha_r} = 0$ for every $X_{\upalpha_r} \in \mk{w}_{\upalpha_r}$. This concludes the case of $\mm{B}_r$-type spaces.

Let us then assume that $\Upsigma$ is of type $\mm{BC}_r$. We have to show that the shape operator $\mc{A}_\upxi$ is zero on $\mk{w}_{\upalpha_r} \oplus \mk{g}_{2\upalpha_r}$. We begin with the first summand. Take some $X_{\upalpha_r} \in \mk{w}_{\upalpha_r}$ as well as an auxiliary nonzero vector $Z \in \mk{g}_{-\upalpha_{r-1}}$. Notice that the $(-\upalpha_{r-1})$-string of $\uplambda_2 = \upalpha_{r-1} + \upalpha_r$ consists of $\uplambda_2$ and $\upalpha_r$. By virtue of Lemma \ref{lem:string_injective}, the restriction of $\ad(Z)$ to $\mk{g}_{\uplambda_2}$ gives an injective map $\mk{g}_{\uplambda_2} \hookrightarrow \mk{g}_{\upalpha_r}$. As we mentioned earlier, all roots in $\Updelta_r^1$ have the same multiplicity, so this map is actually an isomorphism. What is more, according to Corollary \ref{corollary:new:normalizers}, it carries $\mk{w}_{\uplambda_2}$ into $\mk{w}_{\upalpha_r}$ and $\mk{v}_{\uplambda_2}$ into $\mk{v}_{\upalpha_r}$. Consequently, this map has to send $\mk{w}_{\uplambda_2}$ isomorphically onto $\mk{w}_{\upalpha_r}$. This implies that there exists a vector $Y \in \mk{w}_{\uplambda_2}$ such that $[Z,Y] = X_{\upalpha_r}$. We can now compute:
\begin{align*}
    \mathcal{A}_\upxi X_{\upalpha_r} &= \frac{1}{2}([\upxi, X_{\upalpha_r}] - [\uptheta \upxi, X_{\upalpha_r}] )^\top && \text{(Proposition \ref{proposition:shape:orbit:an}\eqref{proposition:shape:orbit:an:b})} \\
    &= \frac{1}{2}[\upxi, [Z,Y]]^\top && \text{(Lemma \ref{lemma:a:k0}\eqref{lemma:a:k0:2})} \\
    &= \frac{1}{2} ([[\upxi,Z],Y] + [Z,[\upxi,Y]])^\top = 0 && \text{($\upalpha_r - \upalpha_{r-1} \not\in \Upsigma \;$ and $\; \restr{\ad(\upxi)}{\mk{w}_{\uplambda_2}} = 0$)}.
\end{align*}
Now we can show that $\mc{A}_\upxi$ actually vanishes on the whole $\mk{w}_{\upalpha_r} \oplus \mk{g}_{2\upalpha_r}$. Proposition \ref{proposition:shape:orbit:an}\eqref{proposition:shape:orbit:an:b} together with the above calculation imply that this subspace is invariant under $\mc{A}_\upxi$. As shape operators are diagonalizable, it suffices to show that the restriction of $\mc{A}_\upxi$ to this subspace does not have nonzero eigenvalues. Suppose $X_{\upalpha_r} + X_{2\upalpha_r} \in \mk{w}_{\upalpha_r} \oplus \mk{g}_{2 \upalpha_r}$ is an eigenvector of $\mc{A}_\upxi$ with eigenvalue $c$. Note that $\mathcal{A}_\upxi X_{2\upalpha_r}$ belongs to $\mk{w}_{\upalpha_r}$---by virtue of Proposition \ref{proposition:shape:orbit:an}\eqref{proposition:shape:orbit:an:b} and the fact that $3\upalpha_r$ is not a root. Hence, we have
\[
c X_{\upalpha_r} + c X_{2\upalpha_r} = \mathcal{A}_\upxi (X_{\upalpha_r} + X_{2\upalpha_r}) = \mathcal{A}_\upxi X_{2\upalpha_r} \in \mk{w}_{\upalpha_r},
\]
which implies either $c = 0$ or $X_{2\upalpha_r} = 0$. In the latter case, we have $\mathcal{A}_\upxi (X_{\upalpha_r} + X_{2\upalpha_r}) = \mathcal{A}_\upxi X_{\upalpha_r} = 0$ and thus still $c = 0$. This completes the case of $\mm{BC}_r$. Altogether, we have shown that the shape operator $\mc{A}_\upxi$ vanishes whenever $\upxi \in \mk{v}_{\upalpha_r}$ and hence the singular orbit $S$, being CPC, is totally geodesic.
\end{proof}

Now we can finally put all pieces of the puzzle together and prove the main result of this article.

\textit{Proof of the \hyperlink{thm:main}{Main theorem}.}
Our setup is as follows: $M = G/K$ is a symmetric space of noncompact type and $H$ is a connected Lie group acting on $M$ properly and isometrically. The actions in parts \eqref{thm:main:a}-\eqref{thm:main:e} of the \hyperlink{thm:main}{Main theorem} are indeed of cohomogeneity one: for \eqref{thm:main:a}-\eqref{thm:main:d}, this follows directly from Theorem \ref{thm:classification_of_c1_actions}, whereas for \eqref{thm:main:e} from the fact that canonical extension preserves the cohomogeneity of an action. Conversely, suppose the action of $H$ is of cohomogeneity one. We already know that it must be orbit-equivalent to one of the actions in Theorem \ref{thm:classification_of_c1_actions}. Parts \eqref{thm:classification_of_c1_actions:a}-\eqref{thm:classification_of_c1_actions:d} thereof correspond to the same parts in the \hyperlink{thm:main}{Main theorem}, so we need only deal with \eqref{thm:classification_of_c1_actions:e}. In other words, we need to show that every nilpotent construction action on $M$ is accounted for in parts \eqref{thm:main:a}-\eqref{thm:main:e} of the \hyperlink{thm:main}{Main theorem}. 

To begin with, the action of $H$ may arise via canonical extension: suppose $B_\Upphi$ is a boundary component of $M$ and $H_\Upphi \curvearrowright B_\Upphi$ is a C1-action whose canonical extension is orbit-equivalent to $H \curvearrowright M$. Since canonical extension is a transitive procedure, we can repeat this process inductively until no further canonical extension is possible. For the sake of brevity, we simply assume that the action of $H_\Upphi$ on $B_\Upphi$ does not arise via canonical extension from a proper boundary component of $B_\Upphi$. Now we apply Theorem \ref{thm:classification_of_c1_actions} to $H_\Upphi \curvearrowright B_\Upphi$. This action has a singular orbit (since so did $H \curvearrowright M$), so it has to correspond to parts \eqref{thm:classification_of_c1_actions:c}-\eqref{thm:classification_of_c1_actions:e} of the theorem. First, suppose it corresponds to \eqref{thm:classification_of_c1_actions:c} or \eqref{thm:classification_of_c1_actions:d}. As we assumed that $H_\Upphi \curvearrowright B_\Upphi$ does not come from nontrivial canonical extension, we simply have that either it has a totally geodesic singular orbit and the boundary component $B_\Upphi$ is irreducible, or else it is orbit-equivalent to a diagonal action and $B_\Upphi$ is a product of two homothetic rank-one spaces. In the latter case, we can actually require strong orbit-equivalence (even by a trivial isometry), as follows from the proof of Theorem A in \cite{DR_DV_Otero_C1}. Since strong orbit-equivalence is preserved by canonical extension, we see that the original action $H \curvearrowright M$ is then orbit-equivalent to an action described in part \eqref{thm:main:c} or \eqref{thm:main:d} of the \hyperlink{thm:main}{Main theorem}, respectively.

We are left to consider the situation where the action $H_\Upphi \curvearrowright B_\Upphi$ is given by part \eqref{thm:classification_of_c1_actions:e} of Theorem \ref{thm:classification_of_c1_actions}, i.e., it is orbit-equivalent to a nilpotent construction action. Using the terminology established in Subsection \ref{sec:admissibility:redundancy}, let $\upchi \in \wt{\mc{N}}$ be nilpotent construction data on $B_\Upphi$ such that $H_\Upphi \curvearrowright B_\Upphi$ is orbit-equivalent to the action of $\mc{H}(\mc{G}(\upchi))$ by means of some isometry $g \in I(M)$, i.e., $g$ carries the orbits of $H_\Upphi$ onto those of $\mc{H}(\mc{G}(\upchi))$. Since the maps $\mc{G}$ and $\mc{H}$ are $I(M)$-equivariant, we see that the action of $H_\Upphi$ has the same orbits as the action of $g^{-1}\mc{H}(\mc{G}(\upchi))g = \mc{H}(\mc{G}(g^{-1} \cdot \upchi))$ corresponding to the nilpotent construction data $g^{-1} \cdot \upchi$. In other words, we can simply assume that the action $H_\Upphi \curvearrowright B_\Upphi$ itself is a nilpotent construction action. Going further, recall that nilpotent construction actions are decomposable (see the discussion after Theorem \ref{thm:classification_of_c1_actions}). But decomposable C1-actions on reducible symmetric spaces of noncompact type always arise via nontrivial canonical extension (see \cite[Lem.\hspace{2pt}4.1]{DR_DV_Otero_C1}). This means that $B_\Upphi$ has to be irreducible. In other words, $\Upphi$ has to be a connected subset of the Dynkin diagram $\DD$ of $M$. But then $\Upphi$ is contained in a single connected component $\DD_i$ of $\DD$, which translates to the fact that $B_\Upphi$ has to be contained in a single de Rham factor $M_i$ of $M$. We may also assume that the singular orbit of $H_\Upphi$ is not totally geodesic, otherwise the action $H \curvearrowright M$ is once again given by part \eqref{thm:main:c} of the \hyperlink{thm:main}{Main theorem}. In particular, $B_\Upphi$ cannot be isometric to $\R H^n$ (see \cite[Th.\hspace{2pt}6.1]{berndt_tamaru_cohomogeneity_one}). 

Next, suppose $B_\Upphi$ is isometric to $\C H^{n+1}, \, \H H^{n+1} \, (n \ge 1), \, \Oo H^2, \, \mm{G}_2^2/\SO(4),$ or $\mm{G}_2(\C)/\mm{G}_2$. According to the solution of the nilpotent construction problem on these spaces in \cite{berndt_bruck,berndt_tamaru_rank_one,protohomogeneous,berndt_tamaru_cohomogeneity_one}, the action of $H_\Upphi$ must belong to $\mc{M}_{\C H^{n+1}}, \, \mc{M}_{\H H^{n+1}},$ or $\mc{M}_{\Oo H^2}$, or else it must be orbit-equivalent to the action of $H_{2,0}$ (see the discussion preceding the \hyperlink{thm:main}{Main theorem} to recall this notation). Using the classification of symmetric spaces of noncompact type (see \cite[pp.\hspace{2pt}336--340]{submanifolds_&_holonomy}), we can deduce that the de Rham factor $M_i$ has to be as described in parts \eqref{thm:main:e:1}-\eqref{thm:main:e:5} of \eqref{thm:main:e} in the \hyperlink{thm:main}{Main theorem}. For example, if $B_\Upphi$ is isometric to $\H H^{n+1}$, then the Dynkin diagram $\DD_i$ of $M_i$ has to contain $\mm{BC}_1$ as its subdiagram, and the multiplicities of the corresponding simple root $\upalpha$ and its double $2\upalpha$ have to be $4n$ and $3$, respectively. From the classification, we see that $M_i$ must then be isometric to the noncompact quaternionic Grassmannian $\Gr^*(k, \H^{2k+n})$ for some $k \ge 1$. Note that the restrictions on $k$ and $n$ in \eqref{thm:main:e}-\eqref{thm:main:e:1} come from the facts that $\mc{M}_{\C H^2} = \varnothing$ and $\SO(3,\H)/\U(3) \simeq \C H^3$. It remains to notice that the canonical extension of $H_\Upphi \curvearrowright B_\Upphi$ to $M$ can be regarded as the product of its canonical extension to $M_i$ with transitive actions on all the other de Rham factors (\cite[Lem.\hspace{2pt}4.1, 4.2]{DR_DV_Otero_C1}). We conclude that the original action $H \curvearrowright M$ is described in part \eqref{thm:main:e} of the \hyperlink{thm:main}{Main theorem} in this case. Consequently, we may assume that $B_\Upphi$ has rank at least two and is not a $\mm{G}_2$-space. By combining Propositions \ref{prop:NC_DEF}, \ref{prop:NC_A}, \ref{prop:NC_C}, and \ref{prop:NC_B_BC}, we deduce that the action $H_\Upphi \curvearrowright B_\Upphi$ has no other options but to arise from nontrivial canonical extension or have a totally geodesic singular orbit---and both of these possibilities have already been excluded. This concludes the proof of the \hyperlink{thm:main}{Main theorem}. \qed

\bibliographystyle{Iaomalpha}
\bibliography{main}

\end{document}